\documentclass[11pt]{amsart}
\usepackage{amsmath, amsthm, amscd, amssymb, enumerate, verbatim, graphicx}
\usepackage[all]{xy}
\usepackage{xcolor}

\newtheorem{thm}{Theorem}[section]
\newtheorem{prop}{Proposition}[section]
\newtheorem{cor}{Corollary}[section]
\newtheorem{lem}{Lemma}[section]
\newtheorem*{thm_AF_I}{Theorem I}
\newtheorem*{thm_AF_II}{Theorem II}
\newtheorem*{thm_FRY}{Theorem III}

\theoremstyle{definition}

\newtheorem{remark}{Remark}[section]
\newtheorem{claim}{Claim}[section]
\newtheorem{defn}{Definition}[section]
\newtheorem{notation}{Notation}[section]
\newtheorem{example}{Example}[section]
\newtheorem*{problem}{Problem}%[section] 
\newtheorem{fact}{Fact}[section]
\newtheorem*{property}{Property}%[section]
\newtheorem{compl}{Complement}[section] 
\newtheorem{setting}{Setting}[section] 
\newtheorem*{observation}{Observation}%[section]
\newtheorem*{assumption_A}{Assumption}

\def \bthm {\begin{thm}}
\def \bthmb {\begin{thm} \mbox{}}
\def \ethm {\end{thm}} 

\def \bprop {\begin{prop}}
\def \bpropb {\begin{prop} \mbox{}}
\def \eprop {\end{prop}} 

\def \blem {\begin{lem}}
\def \blemb {\begin{lem} \mbox{}}
\def \elem {\end{lem}} 

\def \bcor {\begin{cor}}
\def \bcorb {\begin{cor} \mbox{}}
\def \ecor {\end{cor}}

\def \bdefn {\begin{defn}} 
\def \bdefnb {\begin{defn} \mbox{}}
\def \edefn {\end{defn}} 

\def \bfact {\begin{fact}} 
\def \bfactb {\begin{fact} \mbox{}}
\def \efact {\end{fact}} 

\def \bnot {\begin{notation}} 
\def \bnotb {\begin{notation} \mbox{}}
\def \enot {\end{notation}} 

\def \brem {\begin{remark}} 
\def \bremb {\begin{remark} \mbox{}} 
\def \erem {\end{remark}} 

\def \bpty {\begin{property}} 
\def \bptyb {\begin{property} \mbox{}} 
\def \epty {\end{property}} 

\def \bexp {\begin{example}} 
\def \bexpb {\begin{example} \mbox{}} 
\def \eexp {\end{example}} 

\def \bcl {\begin{claim}}
\def \bclb {\begin{claim} \mbox{}} 
\def \ecl {\end{claim}} 

\def \bobs {\begin{observation}}
\def \bobsb {\begin{observation} \mbox{}} 
\def \eobs {\end{observation}} 

\def \bprob {\begin{problem}}
\def \bprobb {\begin{problem} \mbox{}}
\def \eprob {\end{problem}}

\def \btab {\begin{tabular}}
\def \btabb {\begin{tabular} \mbox{}} 
\def \etab {\end{tabular}} 

\def \bary {\begin{array}}
\def \baryb {\begin{array} \mbox{}} 
\def \eary {\end{array}} 

\def \bpf {\begin{proof}}
\def \bpfb {\begin{proof} \mbox{}} 
\def \epf {\end{proof}} 

\def \benum {\begin{enumerate}}
\def \eenum {\end{enumerate}} 

\def \bit {\begin{itemize}}
\def \eit {\end{itemize}} 

\def \edc {\end{document}} 

\newcommand{\relsblr}[2]{
\bary[c]{@{}c@{}}
\mbox{$#1$} \\[-2mm] 
\mbox{$#2$}
\eary
}

\def \LLRAdefn
{\btab[b]{c}
\mbox{\footnotesize def} \\[-1.5mm] 
\ $\Llra$ \ 
\etab}

\def \bec {\mbox{($\because$) }}
 
\def \upfa {\mbox{${}^\forall$}}
\def \bs {\boldsymbol} 

\def \e   {\varepsilon}
\def \cal {\mathcal}
\def \phi {\varphi}

\def \ds {\displaystyle}
\def \tf {\mbox{$\therefore$ }}
\newcommand{\id}{\mathrm{id}}
\newcommand{\supp}{\mathrm{supp}}
\newcommand{\sfrac}[2]{\mbox{\small $\ds \frac{#1}{#2}$}}
\newcommand{\ssum}[2]{\mbox{\small $\ds \sum_{#1}^{#2}$}}
\newcommand{\sprod}[2]{\mbox{\small $\ds \prod_{#1}^{#2}$}}

\def \oIZ {\overline{\Bbb Z}}

\def \IR {{\Bbb R}}

\def \IB {{\Bbb B}}
\def \IZ {{\Bbb Z}}
\def \IS {{\Bbb S}}

\renewcommand{\Theta}{\varTheta}
\renewcommand{\Sigma}{\varSigma}
\renewcommand{\Phi}{\varPhi}
\renewcommand{\Psi}{\varPsi}
\renewcommand{\Lambda}{\varLambda}
\renewcommand{\Gamma}{\varGamma}
\def \Llra {\Longleftrightarrow}
\def \Lra {\Longrightarrow}

\def \lra {\longrightarrow}
\def \lla {\longleftarrow}
\def \LLRA {\ $\Llra$ \ } 
\def \LRA {\ $\Lra$ \ }
 
\def \hra {\hookrightarrow}
 
\def \car {\curvearrowright}

\def \itemi {\item[(i)\,\mbox{}]}

\def \itemiii {\item[(iii)]}
\def \itemiv {\item[(iv)]}

\def \itemI {\item[(i)\ \mbox{}]}
\def \itemII {\item[(ii)\,\mbox{}]}

\def \itema {\item[(a)]} 
\def \itemb {\item[(b)]}

\def \hsf {\hspace*{2mm}}
\def \hsh {\hspace*{5mm}}
\def \hsp {\hspace*{10mm}}
\def \hspp {\hspace*{20mm}}
\def \hsppp {\hspace*{30mm}}

\begin{document}
\baselineskip 6 mm

\vspace*{-10mm} 

\vskip 10mm

\title[]{Boundedness of diffeomorphism groups of manifold pairs \\[2mm] 
--- Circle case ---}

\author[Kazuhiko Fukui]{Kazuhiko Fukui} 
\address{Professor Emeritus \\
Kyoto Sangyo University \\
Kyoto, Japan}
\email{fukui@cc.kyoto-su.ac.jp} 

\author[Tatsuhiko Yagasaki]{Tatsuhiko Yagasaki}
\address{Professor Emeritus \\
Kyoto Institute of Technology \\
Kyoto, Japan}
\email{yagasaki@kit.ac.jp}

\subjclass[2020]{Primary 57R50, 57R52;  Secondary 37C05.}  
\keywords{diffeomorphism, boundedness, uniformly perfect, commutator length}

\maketitle

\begin{abstract}{}
In this paper we study boundedness of conjugation invariant norms on diffeomorphism groups of manifold pairs. 
For the diffeomorphism group ${\cal D} \equiv {\rm Diff}(M,N)_0$ of a closed manifold pair $(M, N)$ with $\dim N \geq 1$, 
first we clarify the relation among the fragmentation norm, the conjugation generated norm, 
the commutator length $cl$ and the commutator length with support in balls $clb$ and show that 
${\cal D}$ is weakly simple relative to a union of some normal subgroups of ${\cal D}$. 
For the boundedness of these norms, this paper focuses on the case where $N$ is a union of $m$ circles.
In this case, the rotation angle on $N$ induces a quasimorphism $\nu : {\rm Isot}(M, N)_0 \to \IR^m$, 
which determines a subgroup $A$ of $ \IZ^m$ and 
a function $\widehat{\nu} : {\cal D} \to \IR^m/A$. 
If ${\rm rank}\,A = m$, these data leads to an upper  bound of $clb$ on ${\cal D}$ modulo 
the normal subgroup ${\cal G} \cong {\rm Diff}_c(M - N)_0$. 
Then, some upper  bounds of $cl$ and $clb$ on ${\cal D}$ are obtained from those on ${\cal G}$. 
As a consequence, the group ${\cal D}$ is uniformly weakly simple and bounded when $\dim M \neq 2,4$. 
On the other hand, if ${\rm rank}\,A < m$, then the group ${\cal D}$ admits a surjective quasimorphism, so  it is unbounded and not uniformly perfect. 
We examine the group $A$ in some explicit examples. 
\end{abstract} 

\thispagestyle{empty}

\tableofcontents

\section{Introduction and statement of results}

K.~Abe and K.~Fukui \cite{Abe1, AF1, AF2, AF3} studied the uniform perfectness of the group 
${\rm Diff}(M,N)_0$ of $C^\infty$ diffeomorphisms of a $C^\infty$ manifold $M$ which preserves a submanifold $N$. 
Let $P : {\rm Diff}(M,N)_0 \lra {\rm Diff}(N)_0$ denote the restriction map. 
Their main results are summarized as follows (\cite{AF1, AF3}). 

\begin{thm_AF_I} \mbox{}
\benum 
\item[{\rm (1)}] Suppose $M$ is a connected $C^\infty$ manifold without boundary and $N$ is a proper $C^\infty$ submanifold of $M$ with $\dim N \geq 1$. 
Then, the group ${\rm Diff}_c(M,N)_0$ is perfect. 
\item[{\rm (2)}] $cld\,{\rm Diff}_c(\IR^m, \IR^n)_0 \leq 2$ for $1 \leq n \leq m$. 
\eenum
\end{thm_AF_I}

\begin{thm_AF_II} Suppose $M$ is a closed $C^\infty$ manifold and $N$ is closed $C^\infty$ submanifold of $M$ with $\dim N \geq 1$. 
\benum 
\item[{\rm (1)}] If ${\rm Diff}(N)_0$ and ${\rm Diff}_c(M - N)_0$ are uniformly perfect and $|\pi_0 {\rm Ker}\,P| < \infty$, 
then ${\rm Diff}(M,N)_0$ is uniformly perfect. 
\item[{\rm (2)}] In the case $\dim N = 1$ $($i.e, $N$ is a finite disjoint union of circles$)$, if $|\pi_0 {\rm Ker}\,P| = \infty$, then 
${\rm Diff}(M,N)_0$ admits a unbounded quasimorphism, so that it is not uniformly perfect.
\eenum 
\end{thm_AF_II}

\noindent 
Using this criterion on $\pi_0 {\rm Ker}\,P$, they showed that for a knot $K$ of the 3-sphere $\IS^3$
the group ${\rm Diff}(\IS^3,K)_0$ is uniformly perfect if and only if $K$ is a torus knot. 

In this paper we continue the comprehensive study of conjugation invariant norms $cl$, $clb$, $\eta$ and $\zeta$ on the group ${\rm Diff}(M,N)_0$. 
Here, $cl$ is the commutator length, $clb$ is the commutator length with support in balls, $\eta$ is the fragmentation norm 
and $\zeta_g$ is the conjugation generated norm with respect to $g \in {\rm Diff}(M,N)_0^\times$. 
The norm $cl$ is related to the uniform perfectness and $\zeta_g$ is related to the uniform (weak) simplicity. 
The norm $clb$ seems to be rather artificial, but it is important for our purpose since it dominates the norms $cl$ and $\zeta_g$ 
and is dominated by $\eta$. 
We also note that, to define the norms $clb$ and $\eta$ on ${\rm Diff}(M^n, N^\ell)_0$, 
it is necessary to treat the family of finite disjoint unions of $n$-balls $D= \cup_i D_i$ in $M$ such that 
each $n$-ball $D_i$ satisfies either $D_i \subset M - N$ or $(D_i, D_i \cap N)$ is 
diffeomorphic to the standard ball pair $(\IB^n, \IB^\ell)$ (cf. \S4.1, \S4.2). 

First we observe the relations and finiteness of these norms on the group ${\rm Diff}_c(M, N)_0$. 

\bthm\label{} 
Suppose $M$ is a connected $C^\infty$ manifold without boundary and $N$ is a proper $C^\infty$ submanifold of $M$ with $\dim N \geq 1$. 
\benum
\item[{\rm (1)}] $cl\,f \leq clb\,f \leq 2\eta(f) < \infty$ \ $(f \in {\rm Diff}_c(M, N)_0)$. 
\item[{\rm (2)}] Suppose $N$ has finitely many connected components $N_i$ $(i \in [m])$. Then, \\
$\zeta_g(f) \leq 4 clb\,f$ \ $(f \in {\rm Diff}_c(M, N)_0 - {\cal K}{\cal P})$ 
and ${\rm Diff}_c(M, N)_0$ is weakly simple relative to ${\cal K}{\cal P}$. 
\eenum 
\ethm 

The statement (1) follows from Fragmentation Lemma (Proposition~\ref{prop_frag}) and Theorem I (2). 
For the statement (2), consider the group homomorphisms \\
\hspp $P_i : {\rm Diff}_c(M, N)_0 \lra {\rm Diff}_c(N_i)_0$ : $P_i(f) = f|_{N_i}$ \ \ $(i \in [m])$. \\
Then, the group ${\rm Diff}_c(M, N)_0$ includes the normal subgroups ${\rm Ker}\,P_i$ $(i \in [m])$ 
and any element of ${\cal K \cal P} := \bigcup_{i \in [m]} {\rm Ker}\,P_i$ does not normally generate ${\rm Diff}_c(M, N)_0$. 
On the other hand, for any $g \in {\rm Diff}_c(M, N)_0 - {\cal K \cal P}$ Epstein's argument (\cite{Ep}) yields  
the inequality in the statement (2) and the finiteness of $\zeta_g$ means that $g$ normally generates ${\rm Diff}_c(M, N)_0$. 
The notion of weak simplicity is introduced to describe this situation (cf.\,\S2.1).
In the case $m = 1$, the subset ${\cal K \cal P}$ itself is a normal subgroup of ${\rm Diff}_c(M, N)_0$ and 
this notion reduces to the notion of relative simplicity (cf.\,\cite{KKKMMO}). 

As for the evaluation of the norms $cl$ and $clb$, this paper focuses on the case that $\dim N = 1$.  
The case of $\dim N \geq 2$ is discussed in a succeeding paper. 
Suppose $M$ is a closed connected $C^\infty$ manifold of $\dim M \geq 2$ and $N$ is a disjoint union of circles $S_i$ $(i \in [m])$ in $M$. 
To detect the behavior of each $f \in {\rm Diff}(M,N)_0$ around the submanifold $N$, 
we can use a surjective quasimorphism \\
\hspp $\nu = (\nu_i)_{i \in [m]} : {\rm Isot}(M, N)_0 \lra \IR^m$ \\
defined using the rotation angle on the circles $S_i$ $(i \in [m])$. 
For each $i \in [m]$ take a universal cover $\pi_i : \IR \lra \IR/\IZ \approx S_i$ and fix a point $p_i \in S_i$. 
Then, the $i$-th component $\nu_i$ is defined by \\
\hspp $\nu_i(F) =$ [the rotation angle of the path $F(p_i, \ast)$ in $S_i$]. \\
It restricts to a group homomorphism $\nu| : {\rm Isot}(M, N)_{\id, \id} \lra \IZ^m$, which has more familiar values, 
\hspp $\nu_i(F) =$ [the degree of the loop $F(p_i, \ast)$ in $S_i$] \ $(F \in {\rm Isot}(M, N)_{\id, \id})$. \\
This determines a subgroup $A := {\rm Im}\,\nu|$ of $\IZ^m$ and the quasimorphism $\nu$ induces a surjective map \\
\hspp $\widehat{\nu} : {\rm Diff}(M, N)_0 \lra \IR^m/A$, \ \ $\widehat{\nu}(f) = \nu(F) + A \subset \IR^m$ \ \ $(F \in {\rm Isot}(M, N)_{\id, f})$.  \\
The subgroup $A$ is related to $\pi_0 {\rm Ker}\,P$ in Theorem II by $\pi_0 {\rm Ker}\,P = \IZ^m/A$. 
However, $A$ itself and the affine lattice $\widehat{\nu}(f)$ in $\IR^m$ contain more informations than the quotient $\IZ^m/A$.  
Due to their definitions, the maps $\nu$ and $\widehat{\nu}$ include no information on the normal subgroup of ${\rm Diff}(M, N)_0$, \\
\hspp ${\cal G} := {\rm Diff}(M; \, {\rm rel}\  {\cal U}(N))_0 \cong {\rm Diff}_c(M-N)_0$, \\
where ${\cal U}(N)$ denotes the neighborhood system of $N$ in $M$. 
To separate this factor we introduce the norms $cl$ and $clb$ modulo ${\cal G}$, which are denoted by $cl_{/\cal G}$ and $clb_{/\cal G}$ and 
satisfy the following inequalities (cf.\,\S2.2, \S5.4 Proposition~\ref{prop_nu_diagram_G}\,(2)) \\ 
\hspp $cl_{/\cal G}\,f \leq cl\,f \leq cl_{/\cal G}\,f + cld\,{\cal G}$ \ \ and \ \ $clb_{/\cal G}\,f \leq clb\,f \leq clb_{/\cal G}\,f + clbd\,{\cal G}$ \ \ \ $(f \in {\rm Diff}(M, N)_0)$. 

In \cite{FRY} we have studied the norms $cl$ and $clb$ 
on the group ${\rm Diff}^\infty(L; \ {\rm rel}\ {\cal U}(\partial L))_0 \cong 
{\rm Diff}^\infty_c({\rm Int}\,L)_0$ for a compact manifold $L$ 
, with a modification of Tsuboi's arguments in \cite{Ts2, Ts3, Ts4}. 
This was a basis for the study of these norms on the diffeomorphism groups of open manifolds in \cite{FRY}. 
Here we can apply these results to the group ${\cal G}$ so to obtain some upper bounds for $cld\,{\cal G}$ and $clbd\,{\cal G}$ for $\dim M \neq 2,4$. 
Therefore, the above formulation has a practical sense. 

For $\IR^m$ we consider the norm $\ds \| x \| := \max_{i \in [m]} |x_i|$ $(x = (x_i)_{i \in [m]} \in \IR^m)$. 
For a subset $C \subset \IR^m$ let $\| C \| := \{ \| x \| \mid x \in C \} \subset \IR$. 
We have the following estimates of $cl\,f$ and $clb\,f$ for $f \in {\rm Diff}(M, N)_0$ 
(cf.~Theorem~\ref{thm_clb_G}\,(2), Proposition~\ref{prop_nu_diagram_G}\,(1)). 

\bthm\label{thm_clbf} Suppose $M$ is a closed $C^\infty$ $n$-manifold of $n \geq 2$ and 
$N$ is a finite disjoint union of circles in $M$. 
Let ${\cal G} := {\rm Diff}(M; \, {\rm rel}\  {\cal U}(N))_0 \cong {\rm Diff}_c(M-N)_0$.
Then, for any $f \in {\rm Diff}(M, N)_0$ the following holds. 
\benum 
\item[{\rm (1)}] If $\min \| \widehat{\nu}(f) \| < \ell \in \IZ_{\geq 1}$, then \\
\hsp $cl_{/\cal G}\,f \leq clb_{/\cal G}\,f \leq 2\ell + 1$ \ \ and \ \ $cl\,f \leq 2\ell + 1 + cld\, {\cal G}$, \ \ $clb\,f \leq 2\ell + 1 + clbd\, {\cal G}$. 

\vskip 1mm 
\item[{\rm (2)}] $cl\,f  \geq cl_{/\cal G}\,f \geq \sfrac{1}{\,4\,} \big(\min \| \widehat{\nu}(f) \| + 1\big)$ \ \ if $f \in {\rm Diff}(M, N)_0 - {\cal G}$.
\eenum
\ethm 

The uniform boundedness of the norms $cl$ and $clb$ on the whole group ${\rm Diff}(M, N)_0$ is detected by ${\rm rank}\,A$. 
If ${\rm rank}\,A < m$, we can find a group epimorphism $\rho : \IR^m/A \to \IR$ 
such that $\phi := \rho\, \widehat{\nu} : {\rm Diff}(M, N)_0 \lra \IR$ is a surjective quasimorphism. 
This implies the following conclusion (cf. \S2.4.1, \S5.4 Proposition~\ref{prop_nu_diagram}\,(5)). 

\begin{prop}\label{prop_rank<m} Suppose $M$ is a closed $C^\infty$ $n$-manifold of $n \geq 2$ 
and $N$ is a disjoint union of $m$~circles in $M$. 
If ${\rm rank}\,A < m$, then the group ${\rm Diff}(M, N)_0$ admits a surjective quasimorphism, 
so that it is neither bounded nor uniformly perfect. 
\end{prop}

In the case that ${\rm rank}\,A = m$, the quotient group $\IZ^m/A$ is a finite abelian group and we can introduce the following quantities : 
\hsh $k := \ds \max_{i \in [m]} k_i \in \IZ_{\geq 1}$ 
\ \ and \ \ $\widehat{k} := 2\lfloor k/2\rfloor + 3$ \\[1mm] 
where $e_i$ $(i \in [m])$ is the standard basis of $\IR^m$ and $k_i := {\rm ord}\,[e_i] \in \IZ_{\geq 1}$ in $\IZ^m/A$ for each $i \in [m]$. 
Since any lattice $x + A \subset \IR^m$ $(x \in \IR^m)$ meets the rectangle $\prod_{i \in [m]} (-k_i/2, k_i/2]$, 
we obtain the following upper bounds for $cl$, $clb$ and $clb_{/{\cal G}}$ 
(cf. \S6.2 Theorem~\ref{thm_clb_G}, Corollary~\ref{cor_clb_G}, Theorem III). 

\bthm\label{thm_rank=m} Suppose $M$ is a closed $C^\infty$ $n$-manifold of $n \geq 2$ and 
$N$ is a disjoint union of $m$~circles in $M$. 
If ${\rm rank}\,A = m$, then \\ 
\hsp \ \ $clb_{/{\cal G}}d\,{\rm Diff}(M,N)_0 \leq \widehat{k}$ \ \ and \ \ 
$cld\,{\rm Diff}(M,N)_0 \leq \widehat{k} + cld\,{\cal G}$, \ \ $clbd\,{\rm Diff}(M,N)_0 \leq \widehat{k} + clbd\,{\cal G}$. \\
In addition, if $n \neq 2,4$ and $M$ is connected, then $cld\,{\cal G} \leq clbd\,{\cal G} < \infty$, so  
the group ${\rm Diff}(M,N)_0$ is uniformly weakly simple relative to ${\cal K \cal P}$ and is bounded. 
\ethm

Moreover, if $n = 2i+1$ $(i \geq 1)$, then $cld\,{\cal G} \leq 4$ and $clbd\,{\cal G} \leq 2n+4$. 
When $n = 2i$ $(i \geq 3)$, some upper bounds for $cld\,{\cal G}$ and $clbd\,{\cal G}$ 
are obtained in terms of handle decompositions or triangulations of $M - N$. 

In the case $m = 1$ (i.e., $N$ consists of a circle), $A$ is just a subgroup of $\IZ$ and $A = k\IZ$ for a unique $k \in \IZ_{\geq 0}$.  
Then, ${\rm rank}\,A < 1$ if and only if $k = 0$. 
In \cite{Abe1} it is shown that for a knot $K$ in the 3-sphere ${\Bbb S}^3$, 
${\rm Diff}({\Bbb S}^3, K)_0$ is uniformly perfect if and only if $K$ is a torus knot. 
We can extend this example to the following form. 

\bexp\label{exp_k=1} Suppose $M$ is a closed connected $C^\infty$ $n$-manifold ($n \geq 2)$ and $K$ is a circle in $M$. 
\benum[(1)] 
\item Suppose $M$ admits an $\IS^1$ action $\rho$ such that $K = \IS^1 \cdot p$ for some point $p$ of $K$ and 
the orbit map $\rho_p : \IS^1 \to K$, $\rho_p(z) = z \cdot p$, has degree $\ell$ (up to $\pm$). Then, $\ell \in A$ and $k| \ell$. 
\bit 
\itemI In the case $M$ is a Seifert fibered 3-manifold, 
if $K$ a regular fiber, then $k = 1$ and if $K$ is a $(p, q)$ multiple fiber, then $k|p$. 
\itemII In particular, if $K$ is a torus knot in $\IS^3$, then $k = 1$, 
since $K$ is a regular fiber of a standard Seifert fibering of $\IS^3$ with two singular fibers.  
\eit 

\item $A = \{ 0 \} \subset \IZ$ and $k = 0$ in the following cases : 
\bit 
\item[] $(\dagger)$ (a) $\pi_1 (M)$ has a trivial center and (b) $\pi_1(K) \to \pi_1 (M)$ is injective.  

\item[] $(\ddagger)$ $n = 3$, \ (a) $\pi_1 (M - K)$ has a trivial center and \\
\hsh (b) $\pi_1(D - K) \to \pi_1 (M - K)$ is injective for a tubular neighborhood $D$ of $K$  in $M$. 
\eit 

\bit 
\itemI For example, a non-torus knot $K$ in $S^3$ satisfies  the condition $(\ddagger)$ (\cite[Ch\,3]{BZ}). 
Note that the condition $(\ddagger)$ for $n \geq 4$ reduces to the condition $(\dagger)$ by the general position argument. 
\itemII The following example $(M, K)$ satisfies the condition $(\dagger)$ for $n \geq 4$. 
Suppose $G$ is a finitely presented group such that $Z(G) = 1$ 
and $G$ includes an element $a$ of infinite order (for example, $G = \IZ \ast H$, $H$ is a nontrivial group).  
Since $n \geq 4$, there exists a closed connected $C^\infty$ $n$-manifold $M$ with $\pi_1(M) \cong G$. 
Take a circle $K$ in $M$ which represents the element $a$ in $\pi_1(M)$, so that the inclusion $i : K \subset M$ induces an isomorphism 
$i_\ast : \pi_1(K) \cong \langle a \rangle < \pi_1(M)$. 
\eit 
\eenum 

The assertion (1) is shown as follows. The $S^1$ action $\rho$ induces an isotopy $F \in {\rm Isot}(M, K)_{\id, \id}$ : $F(x,t) = e^{2\pi it} \cdot x$. 
It follows that $A \ni \nu(F) = {\rm deg}\,F_p = {\rm deg}\,\rho_p = \ell$. 

The assertion (2) is verified in Section 7, together with further examples. 
\eexp 

This paper is organized as follows. Section 2 provides algebraic background necessary in this paper.
Section 3 includes basic notions and notations related to diffeomorphism groups of manifold pairs. 
In Section 4 we verify Fragmentation Lemma and clarify the relations among the norms $cl$, $clb$, $\zeta$ and $\eta$ 
on the group ${\rm Diff}(M, N)_0$. 
Sections 5 - 7 are concerned with the case where the submanifold $N$ is a union of circles. 
In Section 5 we study general properties of quasimorphisms on ${\rm Diff}(M, N)_0$ induced from the rotation angle on $N$. 
In Section 6 we discuss factorization of diffeomorphisms along $N$ and some upper bound of $clb_{/{\cal G}}$.  
Section 7 includes some examples. 

Though the main results on boundedness are obtained in the $C^\infty$ class, 
there are no constraints on the regularity for the construction of quasimorhisms and the unboundedness results in this paper. 
Therefore, throughout this paper we work in the $C^r$ category and impose additional conditions if necessary. 
This clarifies the necessary conditions in each statement. 
In Introduction we mention the case of closed manifolds. 
However, many results are verified even in the case of open manifolds and compact support. 
In a succeeding paper we study the case where $\dim\,N \geq 2$ (\cite{FY3}, cf.\,\cite{FY2}). 

\section{Conjugation invariant norms and quasimorphisms} 

Study of boundedness of a discrete group $G$ aims quantitative understanding of various conjugation invariant norms on $G$.  
Let $\overline{\IR}_{\geq 0} = [0, \infty]$, $\overline{\IZ}_{\geq 0} = \IZ_{\geq 0} \cup \{ \infty \}$ and  
$[m] := \{ 1, 2, \cdots, m \}$, $[m]_+ := \{ 0 \} \cup [m]$ for $m \in \IZ_{\geq 1}$ (and $[0] = \emptyset$).

\subsection{Conjugation invariant norms --- Weak simplicity} \mbox{} 

First we recall basic facts on conjugation-invariant norms \cite{BIP, Ts1}. 
Suppose $G$ is a group with the unit element $e$. 
A conjugation invariant norm on $G$ is a function $q : G \to [0,\infty]$ which satisfies the following conditions~: \ \ 
for any $g, h\in G$
\bit 
\itemi $q(g)=0$ iff $g=e$ \hsh 
(ii) $q(g^{-1})=q(g)$ \hsh 
(iii) $q(gh)\leq q(g)+ q(h)$ \hsh 
(iv) $q(hgh^{-1})= q(g)$. 
\eit 
For a conjugation invariant norm $q$ on $G$, 
the symbol $q < \infty$ means that $q(g) < \infty$ for any $g \in G$.  
The $q$-diameter of a subset $S$ of $G$ is defined by \ $q\hspace{0.2mm}d\,S := \sup q(S)$. 

\bdefn\label{} 
A group $G$ is said to be bounded if any conjugation invariant norm $q : G \to [0, \infty)$ is bounded 
(or equivalently, any bi-invariant metric on $G$ is bounded).
\edefn 

\bnot\label{not_basic} (Fundamental construction)

Suppose $S$ is a subset of $G$. The symbol $N(S)$ denotes the normal subgroup of $G$ generated by $S$. 
If $S$ is symmetric ($S = S^{-1}$) and conjugation invariant ($gSg^{-1} = S$ for any $g \in G$), then 
$N(S) = S^\infty := \bigcup_{k=0}^\infty S^k$ and 
the conjugation invariant norm \  
$q_{
\mbox{\tiny $(G,S)$}} : G \to \overline{\IZ}_{\geq 0}$ \ is defined by \\[2mm] 
\hspace*{20mm} 
$q_{
\mbox{\tiny $(G,S)$}
}(g) := 
\left\{ \hspace{-1mm}
\bary[c]{l}
\min \{ k \in {\Bbb Z}_{\geq 0} \mid \text{$g = g_1 \cdots g_k$ for some $g_1, \cdots, g_k \in S$}\}
\hsh (g \in N(S)), \\[2mm] 
\infty \hsh (g \in G - N(S)).
\eary \right.$ \\[2mm] 
Here, the empty product $(k=0)$ denotes the unit element $e$ in $G$ and $S^0 = \{ e \}$. 
\enot 

\bexpb The fundamental construction yields the following basic conjugation invariant norms. 
\benum 
\item the commutator length $cl = cl_G$ : \\
$cl := q_{(G, G^c)}$ :  \ \ $G^c := \{ [a,b] \mid a, b \in G \}$, \ $N(G^c) = [G,G] = (G^c)^\infty$ \hsp $\circ$ \ $cl\,G := cl_G G$ 
\item the conjugation generated norm \ $\zeta_g$ \ $(g \in G)$ : \\
$\zeta_g := q_{(G, C_g)}$ : \ 
\btab[t]{l}
$C_g := C(g) \cup C(g^{-1})$, \hsh $C(g)$ : the conjugation class of $g$ \\[2mm]
$N(C_g) = N(g) = (C_g)^\infty$ 
\etab 
\eenum 
\eexp 

The commutator length $cl$ is related to (uniform) perfectness and the conjugation generated norms \ $\zeta_g$ \ $(g \in G)$ \ are 
related to (uniform) simplicity (cf. \cite{Ts3}). 

\bfact\label{fact_cl+zeta} \mbox{}
\benum 
\item Suppose $S$ and $S'$ are symmetric and conjugation invariant subsets of $G$. 
\bit 
\itemI If $S' \subset S$, then $q_{G, S} \leq q_{G, S'}$. \hsh 
(ii) \ If $k \in \IZ_{\geq 1}$ and $S' \subset \bigcup_{i \in [k]_+} S^i$, then  $q_{G, S} \leq k q_{G, S'}$. 
\eit 

\item 
\bit 
\itemI For $H < G$ \hsh $(cl_G)|_H \leq cl_H$ and $cl_Gd H \leq cld\,H \equiv cl_Hd\, H$ 
\itemII (a) \ $G$ is perfect. \LLRA $cl < \infty$. \hsp (b) \ $G$ is uniformly perfect. \LLRAdefn $cl$ is bounded. 
\eit 
\item 
\bit 
\itemI If $g \in G$ and $h \in C_g$,then \ $C_h = C_g$ \ and \ $\zeta_h = \zeta_g$. 
\itemII For $g \in G$, \hsh $N(g) = G$ \LLRA $\zeta_g < \infty$. 
\itemiii (a) \ $G$ is simple. \LLRA $N(g) = G$ for any $g \in G^\times$ \LLRA $\zeta_g < \infty$ for any $g \in G^\times$ 
\item[] (b) \  $G$ is uniformly simple. \LLRAdefn $(\zeta_g)_{g \in G^\times}$ is uniformly bounded.
\eit 

\item $\zeta_g$ is bounded for some $g \in G^\times$ \LRA $G$ is bounded. 
\bit 
\itemI $\zeta_g \leq k$ \ for some $k \in \IZ_{\geq 0}$ \LRA $q \leq k q(g)$ \ for any conjugation invariant norm $q$ on $G$. 
\eit 
\eenum 
\efact 

In \cite{FY1} we introduced the notion of uniform simplicity relative to a normal subgroup. 
In \cite{KKKMMO} they have studied 
simple groups relative to a normal subgroup systematically. 
In this article we will need the notion of (uniform) simplicity relative to a union of normal subgroups 
for the study of the diffeomorphism groups of manifold pairs (cf. Section 4). 
To avoid any ambiguity and simplify terminology we introduce the notion of weak simplicity. 

\bdefn\label{defn_weak_simple} Suppose $G$ is a nontrivial group. 
\benum 
\item $S_G := \{ g \in G \mid N(g) \subsetneqq G \}$. 
\item $G$ is weakly simple. \LLRA $G$ is normally generated by some element of $G$ \LLRA $S_G \subsetneqq G$. \\ 
$G$ is uniformly weakly simple. \LLRA $S_G \subsetneqq G$ \ and \ $(\zeta_g)_{g \in G - S_G}$ is uniformly bounded.

\item $G$ is weakly simple relative to $S \subset G$ \LLRA $G$ is weakly simple \ and \ $S = S_G$. \\
$G$ is uniformly weakly simple relative to $S \subset G$ \LLRA $G$ is uniformly weakly simple \ and \ $S = S_G$. 

\item $G$ is (uniformly) relatively simple.\LLRA $G$ is (uniformly) weakly simple and $S_G \vartriangleleft G$. 
(cf. \cite{FY1, KKKMMO}) 
\eenum 
\edefn 

\bfact\label{fact_w-simple} Suppose $G$ is a nontrivial group. 
\benum 
\item $S_G = \bigcup \{ N(g) \mid g \in S_G \} = \bigcup \{ N \mid N \vartriangleleft G, N \neq G\}$. \hsh Note that $e \in S_G$. 
\item For $S \subset G$ \hsp $S = S_G$ \LLRA $S$ : $(\sharp)_1$ $+$ $(\sharp)_2$ \LLRA $S$ : $(\sharp)_1$ $+$ $(\sharp)_2'$ \\
Here, \ $(\sharp)_1$ \ $N(g) = G$ \ (or $\zeta_g < \infty$) \ $(\forall\, g \in G - S)$, \ \ 
$(\sharp)_2$ \ $N(g) \subsetneqq G$ \ (or $\zeta_g \not < \infty$) \ $(\forall\, g \in S)$. \\
\hsp \ $(\sharp)_2'$ \ $S$ is a union of a family of proper normal subgroups. \\
In the case that $S \vartriangleleft G$, the conditions $(\sharp)_2$ and $(\sharp)_2'$ are trivial. 

\item $G$ is weakly simple. \LLRA $\zeta_g < \infty$ for some $g \in G$. 

\item $G$ is uniformly weakly simple. \LRA $G$ is bounded. \hsp (by Fact~\ref{fact_cl+zeta}\,(3))

\item Suppose $f : G \, \to\hspace{-4.5mm}\to \, H$ is a group epimorphism between nontrivial groups. 
\bit 
\itemI $f(G - S_G) \subset G - S_H$. \hsp Note that $f(N(x)) = N(f(x))$ \ $(x \in G)$.
\itemII $G$ is weakly simple. \!\LRA $H$ is weakly simple. 
\eit 
\eenum 
\efact 

\noindent The assertion (4) justifies the notion of weak simplicity in the study of boundedness of groups. 

Related boundedness conditions are summarized as follows. \\[-1mm] 
\hspp \btab[t]{|c|c|c|c|} \hline
$q$ & $cl$ & $\zeta_g$ $(g \in G^\times)$ \, & $\zeta_g$ $(g \in G - S_G)$ $(S_G \subsetneqq G)$\makebox(0,14){} \\[1mm] \hline 
$q < \infty$ & perfect & simple & weakly simple \makebox(0,14){} \\[1mm] \hline 
$q \leq c < \infty$ & unifromly perfect & uniformly simple & uniformly weakly simple \makebox(0,14){} \\[1mm] \hline 
$\sup q$ & $cld\,G$ & $\sup\, (\zeta_g)_{g \in G^\times}$ & $\sup\, (\zeta_g)_{g \in G - S_G}$\makebox(0,14){} \\[1mm] \hline 
\etab \\[4mm] 
Further examples of conjugation invariant norms on diffeomorphism groups are introduced in Section~3.

\subsection{Conjugation invariant norms modulo a subgroup} \mbox{} 

\bdefnb For a function $q : G \to [0, \infty]$ and a nonempty subset $S$ of $G$ we define a function $q_{/S}$ by \\ 
\hsppp $q_{/S} : G \to [0,\infty]$, \ $q_{/S}(f) = \inf q(fS^{-1})$. 
\edefn 

If ${\rm Im}\,q \subset \overline{\IZ}_{\geq 0}$, then ${\rm Im}\,q_{/S} \subset \overline{\IZ}_{\geq 0}$. 
If $q : G \to [0,\infty]$ is conjugation invariant, then $q(gSg^{-1}) = q(S)$.

\blem\label{lem_q/S} Suppose $q : G \to [0, \infty]$ is a conjugation invariant norm and $G \supset S \neq \emptyset$. 
\benum 
\item[{\rm (1)}] 
\bit 
\itemI {\rm (a)} \ $q_{/S}(f) \leq q(fa)$ \ $(\upfa a \in S^{-1})$ \hsp {\rm (b)} \ $q_{/S}(f) = 0$ for any $f \in S$. 
\itemII $q_{/S}(f) = 0$ iff $f \in S$, when $\inf q(G - \{ e \}) > 0$ $($for example ${\rm Im}\,q \subset \overline{\IZ}_{\geq 0}$$)$. 
\eit 
\item[{\rm (2)}]  If $e \in S$, then $q_{/S} \leq q$ 
\item[{\rm (3)}]  $q_{/S}$ is  conjugation invariant if $S$ is conjugation invariant. 
\item[{\rm (4)}]  $q_{/S}(gf) \leq q(g) + q_{/S}(f)$ \ for any $f,g \in G$. 
\item[{\rm (5)}] $q \leq q_{/S} + qd\,S$ 
\item[{\rm (6)}] If $S < G$, then $q_{/S}$ has the following factorization. 
\hsh \smash{\raisebox{8mm}{$\xymatrix@M+1pt{
G \ar[d]_-{} \ar[r]^-{q_{/S}} & [0,\infty] \\
G/S  \ar[ur]_{\widetilde{q}} & 
}$}} \hsf  
\btab[t]{l}
$\widetilde{q}(fS) = \inf q(fS)$ 
\etab  
\eenum
\elem   

\bpfb 
\benum 
\item[{\rm (1)}] 
\bit 
\itemI (b) \ $q(fS^{-1}) \ni q(ff^{-1}) = q(e) = 0$ \hsh \tf \ $q_{/S}(f) = \inf q(fS^{-1}) = 0$
\itemII If $fS^{-1} \subset G - \{ e \}$, then $q(fS^{-1}) \subset q(G - \{ e \})$ and $q_{/S}(f) = \inf q(fS^{-1}) \geq \inf q(G - \{ e \}) > 0$. 
Hence, if $q_{/S}(f) = 0$, then $e \in fS^{-1}$ and $f \in S$. 
\eit 

\item[{\rm (3)}]  
$q(gfg^{-1}S^{-1}) = q(gfS^{-1}g^{-1}) = q(fS^{-1})$. \\
\tf \ $q_{/S}(gfg^{-1}) = \inf q(gfg^{-1}S^{-1}) = \inf q(fS^{-1}) = q_{/S}(f)$ 

\item[{\rm (4)}]  
For any $a \in S^{-1}$ it follows that $q_{/S}(gf) - q(g) \leq q(gfa) - q(g) \leq q(g) + q(fa) - q(g) = q(fa)$. \\
\tf \ $q_{/S}(gf) - q(g) \leq \inf q(fS^{-1}) = q_{/S}(f)$ 

\item[{\rm (5)}] For any $a \in S$ it follows that \\
\hsp $q(f) = q(fa^{-1}a) \leq q(fa^{-1}) + q(a) \leq g(fa^{-1}) + qd\,S$ \hsh \tf \ $q(f) - qd\,S \leq q(fa^{-1})$. \\
Then, \ $q(f) - qd\,S \leq \inf\,q(fS^{-1}) = q_{/S}(f)$ \ and \ $q(f) \leq q_{/S}(f) + qd\,S$.  
\eenum 
\vspace*{-7mm} 
\epf

\bexp\label{exp_cl/N} If $G$ is a group and $N \vartriangleleft G$, then \\
\hsp (a) \ $cl_{/N}\,x = cl_{G/N}[x]$ \ $(x \in G)$ \ and \ (b) \ $cl_{/N}d\,G = cl_{G/N}\,G/N$. 
\hsp \smash{\raisebox{7mm}{$\xymatrix@M+1pt{
G \ar[d]_-{} \ar[r]^-{cl_{/N}} & [0,\infty] \\
G/N  \ar[ur]_{cl_{G/N}} & 
}$}} 
\eexp
\vskip 2mm 

\subsection{Vector-valued quasimorphisms} \mbox{} 

In our argument we need to treat a family of non-independent quasimorphisms. 
This family is effectively treated as a quasimorphism with values in a real vector space with the max norm. 
This subsection includes some remarks on vector-valued quasimorphisms. 

For $m \in \IZ_{\geq 1}$ the symbol $[m]$ denotes the index set $\{ 1, 2, \cdots, m \}$. 
We equip the $m$-dimensional Euclidean space $\IR^m$ with the norm $\| x \| := \ds \max_{i \in [m]} |x_i|$.  \vspace*{0.5mm} 
As usual the symbol $e_i$ $(i \in [m])$ denotes the standard basis of $\IR^m$.
For a subset $A \subset \IR^m$ let $\| A \| = \{ \| a \| \mid a \in A \}$.  

Suppose $\Gamma$ is a group with a unit element $e$. 
We set $\Gamma^\times := \Gamma - \{ e \}$ and $\Gamma^c := \{ [a,b] \mid a,b \in \Gamma \}$ (the set of commutators in $\Gamma$).

\bdefn A function $\phi : \Gamma \to \IR^m$ is said to be a quasimorphism if \\
\hspace*{20mm} $\ds D_\phi := \sup_{a,b \in \Gamma} \|\phi(ab) - \phi(a) - \phi(b)\| < \infty$. \\[1mm] 
The quantity $D_\phi$ is called the defect of $\phi$. We set $C_\phi := \sup \| \phi(\Gamma^c) \| \in [0,\infty]$.
\edefn

Various notions for the usual real valued quasimorphisms have almost direct extension to multivalued quasimorphisms. 
For example, a quasimorphism $\phi : \Gamma \to \IR^m$ is said to be homogeneous if 
$\phi(a^n) = n\phi(a)$ for any $a \in \Gamma$ and $n \in \IZ$. 

\bfact\label{fact_multi-qm} \mbox{} 
\benum 
\item $\phi = (\phi_i)_{i \in [m]} : \Gamma \to \IR^m$ is a quasimorphism iff $\phi_i : \Gamma \to \IR$ is a quasimorphism for each $i \in [m]$.  
In this case we have $D_\phi = \ds \max_{i \in [m]} D_{\phi_i}$. 

\item Suppose $\psi :  \varLambda \to \varGamma$ is a group epimorphism. 
Then, a map $\phi : \Gamma \to \IR^m$ is a quasimorphism iff $\phi \psi$ is a quasimorphism. 
In this case we have $D_{\phi} = D_{\phi \psi}$. 

\item If $\phi : \Gamma \to \IR^m$ is a quasimorphism and $\psi : \IR^m \to \IR$ is an $\IR$-linear map, 
then $\psi \phi : \Gamma \to \IR$ is a quasimorphism. 
\eenum 
\efact 

\bfact\label{fact_ubdd_qm} \mbox{} (cf.\,\cite{Fuj, GG}, \cite{FY1}) \ 
If $\Gamma$ admits a unbounded real valued quasimorphism, then $\Gamma$ is neither bounded nor uniformly perfect.
\efact

\subsection{A lattice-valued map induced from a quasimorphism}\label{subsec_lattice} \mbox{} 

In this subsection we assume that a group $G$ is included in the diagram in Setting~\ref{setting_diagram} and 
under this setting we study the boundedness of $G$. 
In the next section we apply this result to diffeomorphism groups of manifold pairs $(M, N)$ with $\dim N = 1$. 

\begin{setting}\label{setting_diagram} \mbox{} Consider the following diagram and the condition $(\ast)$ : \\[2mm] 
\hspace*{15mm} 
\raisebox{0mm}{
$\xymatrix@M+1pt{
K \ar@{->>}[d]_-{\phi|_K} \ar@{}[r]|*{\subset} & I\ar@{->>}[d]^-{\phi} \ar@{->>}[r]^-f & G \\
A \ar@{}[r]|*{\subset} & \ \IR^m & 
}$} \hsh 
\btab[t]{l} 
$I$ is a group, \ $f$ is a group epimorphism, \\[2mm] 
$K := {\rm ker}\,f$, \ $\phi : I \to \IR^m$ is a surjective quasimorphism, \\[2mm]
$A := \phi(K)$ 
\etab \\[2.5mm] 
\hspace*{9mm} $(\ast_1)$ \ $\phi(uv) = \phi(u) + \phi(v)$ for any $u \in K$ and $v \in I$ \ \ and \ \ $(\ast_2)$ \ $A \subset \IZ^m \subset \IR^m$. 
\end{setting}

\bfactb\label{fact_setting}\mbox{} Under the condition $(\ast)$ we have the following conclusion. 
\benum 
\item By $(\ast)_1$ $\phi|_K : K \to \IR^m$ is a group homomorphism. Hence, \\
\hspp $A < \IZ^m$ \ \ and \ \ 
$\phi|_K : K \to\hspace{-4.5mm}\to A$ is a group epimorphism. \\
Let $\ell := {\rm rank}\,A$ and $g : \IR^m \, \to\hspace{-4.5mm}\to \, \IR^m/A$ denote the quotient homomorphism. 
\item There exists a unique surjective map \\
\hsh $\widehat{\phi} : G \, \to\hspace{-4.5mm}\to \,\IR^m/A$ with $\widehat{\phi} f = g \phi$. 
\hspace*{10mm} 
\smash{\raisebox{5mm}{
$\xymatrix@M+1pt{
& K \ar@{->>}[d]_-{\phi|_K} \ar@{}[r]|*{\subset} & I\ar@{->>}[d]^-{\phi} \ar@{->>}[r]^-f & G \ar@{->>}[d]^{\widehat{\phi}} \\
\IZ^m \ar@{}[r]|*{>} & A \ar@{}[r]|*{\subset} & \ \IR^m \ar@{->>}[r]^-g & \IR^m/A 
}$}} \hsh \raisebox{-2.5mm}{$(\sharp)$} \\[10mm]
In fact, if $w \in G$ and $v \in f^{-1}(w)$, then $f^{-1}(w) = vK = Kv$ and 
by $(\ast_1)$ we have a coset $\phi(f^{-1}(w)) = \phi(vK) = \phi(v) + A \in \IR^m/A$.
Hence, the map $\widehat{\phi}$ is defined by \\
\hspp $\widehat{\phi} : G \,\to\hspace{-4.5mm}\to \, \IR^m/A$ : $\widehat{\phi}(w) = \phi(v) + A$ \ \ $(v \in f^{-1}(w))$. \\
Note that $\widehat{\phi}(e) = A \in \IR^m/A$. 
Since $\phi(f^{-1}(w)) = \widehat{\phi}(w)$, for any $y \in \widehat{\phi}(w)$ there exists $v \in I$ with $f(v) = w$ and $\phi(v) = y$. 
\eenum 
\efact 

\bnot\label{not_lattice} 
The subgroup $A \subset \IZ^m$ is a finitely generated free abelian group of rank $\ell \leq m$. 
Each coset $z \equiv x + A \in \IR^m/A$ is an affine lattice in $\IR^m$. 
Hence, the set $\| z \| \equiv \{ \| y \| \mid y \in z \}$ is a discrete subset of $\IR_{\geq 0}$ and 
we have $\theta_z := \min \| z \| \in \IR_{\geq 0}$.   
Note that $\Theta_z := \{ y \in z \mid \| y \| = \theta_z \}$ (the set of points in $z$ nearest to the origin) is not necessarily a singleton. 
It is seen that $\theta_A = 0$ and $\theta_z > 0$ for any $z \in (\IR^m/A)^\times$. 
We put $\theta_\phi := \sup \{ \theta_z \mid z \in \IR^m/A \}$. 
For any $w \in G$ we have the coset $z = \widehat{\phi}(w) \in \IR^m/A$ and 
we also use the notations $\theta_w := \theta_{\widehat{\phi}(w)}$ and $\Theta_w := \Theta_{\widehat{\phi}(w)}$ for each $w \in G$.  
Note that $\theta_e = \theta_A = 0$ and that $\theta_\phi = \sup \{ \theta_w \mid w \in G \}$.  
\enot 

\bprop\label{prop_lower_bound} Suppose $C_\phi + D_\phi > 0$. 
Then, the following hold. 
\bit
\item[(1)] $cl\,w \geq \sfrac{\theta_w + D_\phi}{C_\phi + D_\phi}$ \ for any $w \in G^\times$. 
\hsp {\rm (2)} \ $cld\,G \geq \sfrac{\theta_\phi + D_\phi}{C_\phi + D_\phi}$ \ if $G \neq \{ e \}$. 
\eit 
\eprop 

\bpfb 
\benum 
\item If $w \in G - [G,G]$, then $cl\,w = \infty$ and the assertion is trivial. 
Now we assume that $w \in [G,G] - \{ e \}$ and put $n := cl\,w \in \IZ_{\geq 1}$. We can find $w_i \in G^c$ $(i \in [n])$ with $w = w_1 \cdots w_n$. 
Since $f(I^c) = G^c$, there exist $v_i \in I^c$ $(i \in [n])$ with $f(v_i) = w_i$. 
Let $v := v_1 \cdots v_n \in I$. Then, we have $f(v) = w$ and $\widehat{\phi}(w) = \phi(v) + A$. 
It follows that \\[2mm] 
\hsp $\|\phi(v)\| 
\leq \Big\|\phi(v) - \ssum{j=1}{n}\, \phi(v_j)\Big\| + \ssum{i=1}{n}\, \|\phi(v_j)\|  \leq (n-1)D_{\phi} + C_\phi n = (C_\phi + D_\phi)n - D_\phi$. \\[2mm] 
This implies that \ $n \geq \sfrac{1}{C_\phi + D_\phi}(\|\phi(v)\| + D_\phi) \geq \sfrac{\theta_w + D_\phi}{C_\phi + D_\phi}$ \ as required. 
\vskip 2mm 
\item For any $w \in G^\times$ it follows that $cld\,G \geq cl\,w \geq \sfrac{\theta_w + D_\phi}{C_\phi + D_\phi}$ and 
$(C_\phi + D_\phi)cld\,G - D_\phi \geq \theta_w \geq 0$. \\[1.5mm] 
Note that, since $G^\times \neq \emptyset$, we have $(C_\phi + D_\phi)cld\,G - D_\phi \geq 0$ and $(C_\phi + D_\phi)cld\,G - D_\phi \geq \theta_e = 0$. \\[1.5mm]
Hence, $(C_\phi + D_\phi)cld\,G - D_\phi \geq \theta_\phi$ and  $cld\,G \geq \sfrac{\theta_\phi + D_\phi}{C_\phi + D_\phi}$. 
\eenum 
\vspace*{-7mm} 
\epf 
\vskip 1mm 

The next lemma will be used in the subsequent section to obtain some upper bounds for the norms $cl$ and $clb$. 

\blem\label{lem_upper_bound} 
Suppose a function $\eta : G \to \IR_{\geq 0}$ satisfies the condition $(\flat)$ for a weakly increasing function $\sigma : \IZ_{\geq 1} \to \IZ_{\geq 1}$ $:$ 
\hsh $(\flat)$ \ If $v \in I$ and $\| \phi(v) \| < j \in \IZ_{\geq 1}$, then $\eta(f(v)) \leq \sigma(j)$. \\
Then, the following hold. 
\benum 
\item[{\rm (1)}] $\eta(w) \leq \sigma(\lfloor\theta_w\rfloor + 1)$ \ for any $w \in G$. \hsp {\rm (2)} \ $\sup \eta \leq \sigma(\lfloor \theta_\phi \rfloor + 1)$. 
\eenum 
\elem 

\bpfb
\benum
\item Take an $s \in \Theta_w$. Since $s \in \widehat{\phi}(w) = \phi(f^{-1}(w))$, 
we can find $v \in f^{-1}(w)$ with $\phi(v) = s$ (cf. Fact~\ref{fact_setting}\,(2)). 
Since $\| \phi(v) \| = \| s \| = \theta_w < j := \lfloor\theta_w\rfloor + 1 \in \IZ_{\geq 1}$, 
from the assumption $(\flat)$ it follows that \ \ $\eta(w) = \eta(f(v)) \leq \sigma(j) = \sigma(\lfloor\theta_w\rfloor + 1)$. 
\item $\eta(w) \leq \sigma(\lfloor\theta_w\rfloor + 1) \leq \sigma(\lfloor\theta_\phi \rfloor + 1)$ \ for any $w \in G$. 
Hence, $\sup \eta \leq \sigma(\lfloor \theta_\phi \rfloor + 1)$. 
\eenum
\vspace*{-7mm} 
\epf 
\vskip 1mm

Below we see that the boundedness of $G$ relies on the conditions [1] $\ell < m$ and [2] $\ell = m$. 

\subsubsection{\bf The case $\ell < m$} \mbox{}  

\bprop\label{prop_rank<m} If $\ell < m$, then there exists a group epimorphism $\chi : \IR^m/A \,\to\hspace{-4.5mm}\to \, \IR$ such that 
$\chi \widehat{\phi} : G \,\to\hspace{-4.5mm}\to \, \IR$ is a surjective quasimorphism. 
Hence, the group $G$ is unbounded and not uniformly perfect. 
\eprop 

\bpfb
Consider the $\IR$-subspace $\IR\langle A \rangle$ of $\IR^m$ generated by $A$. 
Then $\dim \IR\langle A \rangle = \ell < m$ and we have a surjective $\IR$-linear map $\eta : \IR^m/\IR\langle A \rangle \,\to\hspace{-4.5mm}\to \, \IR$. 
Consider the composition $\chi := \eta \rho : \IR^m/A \,\to\hspace{-4.5mm}\to \, \IR$ of $\eta$ and 
the natural group homomorphism $\rho : \IR^m/A \,\to\hspace{-4.5mm}\to \, \IR^m/\IR\langle A \rangle$.  
Then, by Fact~\ref{fact_multi-qm}\,(3) $\chi \widehat{\phi} f = \chi g \phi = (\eta \rho g) \phi$ is a quasimorphism,  
since $\phi$ is quasimorphism and $\eta$ and $\rho g$ are $\IR$-linear maps. 
Since $f$ is group epimorphism, $\chi \widehat{\phi}$ is also a quasimorphism by Fact~\ref{fact_multi-qm}\,(2). 
\epf 

\begin{compl} Suppose $\chi : \IR^m/A \,\to\hspace{-4.5mm}\to \, \IR$ is a group epimorphism such that 
$\chi \widehat{\phi} : G \,\to\hspace{-4.5mm}\to \, \IR$ is a surjective quasimorphism and $H < G$. 
\benum
\item If $\widehat{\phi}|_H : H \,\to \hspace*{-4.5mm} \to \,\IR^m/A$ is surjective, then 
$\chi \widehat{\phi}|_H : H \,\to \hspace*{-4.5mm} \to \, \IR$ is a surjective quasimorphism. 
\item If $J < I$, $H = f(J)$ and $\phi|_J : J \,\to \hspace*{-4.5mm} \to \, \IR^m$ is surjective, 
then $\widehat{\phi}|_H : H \,\to \hspace*{-4.5mm} \to \,\IR^m/A$ is surjective. 
\eenum 
\end{compl}

\brem\label{rem_infty} $\{ \theta_z \mid z \in \IR^m/A \} = [0, \infty)$ and $\theta_\phi = \infty$ in the case $\ell < m$. 
\erem 

\subsubsection{\bf The case $\ell = m$} \mbox{} 

The quotient group $\IZ^m/A$ is a finite abelian group. 
Hence, for each $i \in [m]$ the coset $[e_i] \in \IZ^m/A$ has a finite order $k_i := {\rm ord}\,[e_i] \in \IZ_{\geq 1}$. 
We put $\ds k := \max_{i \in [m]} \,k_i$ and 
$J_A := \ds \sprod{i \in [m]}{} (-k_i/2, k_i/2] \, \subset \IR^m$.  
\vspace*{-2mm} 
\blem\label{lem_k} \mbox{}
\benum
\item[{\rm (1)}]  
Since $k_ie_i \in A$ $(i \in [m])$, for any coset $z \in \IR^m/A$ we have $z \cap J_A \neq \emptyset$, that is, 
there exists $y \in J_A$ with $z = [y]$. 
Here, $y$ is not necessarily unique for $z$.

\item[{\rm (2)}] For any $w \in G$ there exists $y \in \widehat{\phi}(w) \cap J_A$ and we have $\theta_w \leq \| y \| \leq \sfrac{1}{\,2\,}k$ and 
$\theta_\phi \leq \sfrac{1}{\,2\,}k$. 
\vskip 1mm 
\item[{\rm (3)}] 
Suppose a function $\eta : G \to \IR_{\geq 0}$ satisfies the condition $(\flat)$ for a weakly increasing function $\sigma : \IZ_{\geq 1} \to \IZ_{\geq 1}$ $:$ 
\hsh $(\flat)$ \ If $v \in I$ and $\| \phi(v) \| < j \in \IZ_{\geq 1}$, then $\eta(f(v)) \leq \sigma(j)$. \\
Then, $\sup \eta \leq \sigma(\lfloor \theta_\phi \rfloor + 1) \leq \sigma(\lfloor k/2 \rfloor + 1)$. 
\eenum 
\elem

\subsubsection{\bf The case $m = 1$} \mbox{} 

\bfact\label{fact_m=1} In the case $m = 1$, \\
\hsp \,(i)\, there exists a unique $k \in \IZ_{\geq 0}$ with \ $\IZ > A = k\IZ$ \ and \\
\hsp (ii)\, $\ell \equiv {\rm rank}\,A < m$ \ iff \ $k = 0$ \ \ (or $\ell = m$ iff $k \geq 1$). 
\hspace*{15mm}
\smash{\raisebox{15mm}{
$\xymatrix@M+1pt{
K \ar@{->>}[d]_-{\phi|_K} \ar@{}[r]|*{\subset} & I\ar@{->>}[d]^-{\phi} \ar@{->>}[r]^-f & G \ar@{->>}[d]^{\widehat{\phi}} \\
k\IZ \ar@{}[r]|*{\subset} & \IR \ar@{->>}[r]^-g & \IR/k\IZ
}$}} 
\benum
\item[{[1]}] the case $k = 0$ : The map $\widehat{\phi} : G \, \to\hspace{-4.5mm}\to \ \IR/0\IZ \cong \IR$ is a surjective quasimorphism. 
Hence, $G$ is unbounded and not uniformly perfect. 
\item[{[2]}] the case $k \geq 1$ : 
\bit 
\item[(1)] 
\bit 
\itemI $k = {\rm ord}\,[e_1]$ in $\IZ/k\IZ$, \ $J_A = (-k/2, k/2]$, 
\itemII for any $w \in G$ the set $\widehat{\phi}(w) \cap J_A$ is a singleton and $\widehat{\phi}(w) \cap J_A \subset \Theta_w$,  
\itemiii $\{ \theta_w \mid w \in G \} = [0,k/2]$ and $\theta_\phi = k/2$. 
\eit 
\vskip 1.5mm 

\item[(2)] $cld\,G \geq \sfrac{k/2 + D_\phi}{C_\phi + D_\phi}$ if $C_\phi + D_\phi > 0$ \ (Proposition~\ref{prop_lower_bound}).
\vskip 1mm 

\item[(3)]
Suppose a function $\eta : G \to \IR_{\geq 0}$ satisfies the condition $(\flat)$ for a weakly increasing function $\sigma : \IZ_{\geq 1} \to \IZ_{\geq 1}$ $:$ 
\hsh $(\flat)$ \ If $v \in I$ and $\| \phi(v) \| < j \in \IZ_{\geq 1}$, then $\eta(f(v)) \leq \sigma(j)$. \\
Then, $\sup \eta \leq \sigma(\lfloor k/2 \rfloor + 1)$. 
\eit 
\eenum 
\efact 

\subsubsection{\bf Restriction and Reduction of Diagarm $(\sharp)$ in Fact~\ref{fact_setting}} \mbox{} 

Recall the diagram $(\sharp)$ in Fact~\ref{fact_setting}. 

\bnot\label{not_J} In the diagram $(\sharp)$, suppose $J < I$. Then, 
we have the subgroups $N := f(J) < G$ and $L := K \cap J < K$. 
These yields the exact sequence 
\hsh $\xymatrix@M+3pt
{\relax
1 \ar[r] & L \ar@{}[r]|*{\subset} & J \ar[r]^{f|} & N \ar[r] & 1.
}$ 
\enot 

\bfact\label{fact_restriction} \mbox{} In the diagram $(\sharp)$, suppose $J \vartriangleleft I$ and $\phi(J) = \IR^m$. 
\bit 
\itemI We have the subgroups $H := f(J) < G$, $L := K \cap J < K$ and the exact sequence \\
\hspp $\xymatrix@M+3pt 
{\relax
1 \ar[r] & L \ar@{}[r]|*{\subset} & J \ar@{->>}[r]^{f|} & N \ar[r] & 1.
}$ 

\itemII The diagram $(\sharp)$ has the following diagram $(\sharp)_0$, which satisfies the conditions in Setting~\ref{setting_diagram}. \\[1mm] 
\hsp \hsh $(\sharp)_0$ \hsh $\xymatrix@M+1pt{
1 \ar[r] & L \ar@{->>}[d]_-{\phi|_L\,} \ar@{}[r]|*{\subset} 
& J \ar@{->>}[d]^-{\,\phi|_J} \ar@{->>}[r]^-{f|} & H \ar@{->>}[d]^{\, \widehat{\phi|_J}} \ar[r] & 1 \\ 
1 \ar[r] & B \ar@{}[r]|*{\subset} & \ \IR^m \ar@{->>}[r]^-h & \IR^m/B \ar[r] & 1
}$ \hsp 
\btab[t]{l} 
$B := \phi(L) < A < \IZ^m$ \\[2mm] 
$D_{\phi|_J} \leq D_\phi$ 
\etab 
\eit 
\efact 

\bfact\label{fact_reduction} \mbox{}
\benum 
\item In the diagram $(\sharp)$, suppose $J \vartriangleleft I$ and $\phi(uv) = \phi(u)$ for any $u \in I$ and $v \in J$. 
\bit 
\itemI We have the normal subgroups $N := f(J) \vartriangleleft G$, $L := K \cap J \vartriangleleft K$ and the exact sequence \\
\hspp $\xymatrix@M+3pt
{\relax 
1 \ar[r] & K/L \ar@{^{(}->}[r] & I/J \ar@{->>}[r]^{\widetilde{f}} & G/N \ar[r] & 1.
}$
\itemII The diagram $(\sharp)$ has the reduction $(\sharp)_1$ below, which satisfies the conditions in Setting~\ref{setting_diagram}.
Note that $D_{\widetilde{\phi}} = D_{\phi}$, $C_{\widetilde{\phi}} = \sup\,\| \widetilde{\phi}\big(\big(I/J\big)^c\big)\| = \sup\,\| \phi(I^c) \| = C_\phi$, 
$\theta_{[w]} = \theta_w$ $(w \in G)$ and $\theta_{\widetilde{\phi}} = \theta_\phi$.
\eit 
\item 
In the diagram $(\sharp)$, if $L \vartriangleleft K$, $L \vartriangleleft I$ and $L < {\rm Ker}\,\phi|_K$, then 
$J := L$ satisfies the condition in (1) and induces the reduction $(\sharp)_2$ below. 
In fact, for any $u \in I$ and $v \in L$ it follows that $uv = v'u$ for some $v' \in L$ and $\phi(uv) = \phi(v'u) = \phi(v') + \phi(u) = \phi(u)$. 
\eenum 
\hsh  
\btab[t]{ll}
$(\sharp)_1$ & $(\sharp)_2$ \\[1mm] 
{\small $\xymatrix@M+3pt
{\relax
1 \ar[r] & K \ar@{}[r]|*{\subset} \ar@{->>}[d] \ar@/_5ex/@{->>}[dd]_-{\raisebox{15mm}{\footnotesize $\phi|_K$}}
& I \ar@{->>}[r]^f \ar@{->>}[d] \ar@/_4ex/@{->>}[dd]_-{\raisebox{15mm}{\footnotesize $\phi$}} 
& G \ar@{->>}[d] \ar[r] \ar@/_5ex/@{->>}[dd]_-{\raisebox{15mm}{\footnotesize $\widehat{\phi}$}} & 1 \\ 
1 \ar[r] & K/L \ar@{^{(}->}[r] \ar@{->>}[d]^-{\,\widetilde{\phi|_K}}
& I/J \ar@{->>}[r]^{\widetilde{f}} \ar@{->>}[d]^-{\,\widetilde{\phi}} 
& G/N \ar@{->>}[d]^-{\widetilde{\widehat{\phi}}\,=\, \widehat{\widetilde{\phi}}} 
\ar[r] & 1 \\ 
0 \ar[r] & A \ar@{}[r]|*{\subset} & \IR^m \ar@{->>}[r]^-{g} & \IR^m/A \ar[r] & 0 \\
}$} \hsh & 
{\small $\xymatrix@M+3pt
{\relax
1 \ar[r] & K \ar@{}[r]|*{\subset} \ar@{->>}[d] \ar@/_5ex/@{->>}[dd]_-{\raisebox{15mm}{\footnotesize $\phi|_K$}}
& I \ar[r]^f \ar@{->>}[d] \ar@/_4ex/@{->>}[dd]_-{\raisebox{15mm}{\footnotesize $\phi$}} & G \ar@{=}[d] \ar[r] & 1 \\ 
1 \ar[r] & K/L \ar@{}[r]|*{\subset} \ar@{->>}[d]^-{\,\widetilde{\phi|_K} \,=\, \widetilde{\phi}|} 
& I/L \ar[r]^{\widetilde{f}} \ar@{->>}[d]^-{\,\widetilde{\phi}} & G \ar@{->>}[d]^-{\widehat{\phi}} \ar[r] & 1 \\ 
0 \ar[r] & A \ar@{}[r]|*{\subset} & \IR^m \ar[r]^-{g} & \IR^m/A \ar[r] & 0 \\
}$} 
\etab 
\efact

\begin{compl} For a group  $I$ and $K \vartriangleleft I$ the following hold. 
\benum 
\item Consider the subset $S := \{ [x,y] \mid x \in K, y \in I \}$ of $I$ and the subgroup $[K, I] := \langle S \rangle$ of $I$ generated by $S$. 
It follows that 
\bit 
\itemI $S$ is symmetric (i.e., $S^{-1} \subset S$), since  $[x,y]^{-1} = [y,x] = [yxy^{-1}, y^{-1}] \in S$ and
\itemII $S$ is conjugation invariant (i.e., $zSz^{-1} \subset S$ \ $(z \in I)$), since $z[x,y]z^{-1} = [zxz^{-1}, zyz^{-1}]$. 
\eit 
Hence, $[K, I] \vartriangleleft I$ and $[K,I] = S^\infty := \bigcup_{k \geq 0} S^k$.

\item Suppose $L \vartriangleleft K$ and $L \vartriangleleft I$. Then, the short exact sequence \hsh 
$1 \lra K/L \subset I/L \, \to\hspace{-4.5mm}\to \, I/K \to 1$ \\  
is a cantral extension iff $[K,I] \subset L$. 
\eenum 
\end{compl}

\section{Diffeomorphism groups of manifold pairs} 

\subsection{Basic notions} \mbox{} 

For sets $X, Y$, the symbol $F(X,Y)$ denotes the set of maps (or functions) $f : X \to Y$. 
For any set $W$ and any map $\phi : W \to F(X,Y)$ we have the associated map 
$\widetilde{\phi} : X \times W \to Y$ : $\widetilde{\phi}(x,w) = \phi(w)(x)$. 
This induces a canonical bijection $F(W, F(X, Y)) \cong F(X \times W, Y)$, by which we identify $\phi$ and $\widetilde{\phi}$ and sometimes denote $\widetilde{\phi}$ by the same symbol $\phi$. By ${\rm pr}_X : X \times Y \to X$ we denotes the projection onto $X$.

For a topological space $X$ 
the symbols ${\cal O}(X)$, ${\cal F}(X)$ and ${\cal K}(X)$ denote 
the collections of open subsets, closed subsets and compact subsets of $X$ respectively. 
For subsets $A$, $B$ of $X$, 
the symbols ${\rm Int}_XA$, $Cl_XA$ and ${\cal U}(A)$ denote the topological interior, closure and the neighborhood system of $A$ in $X$ and 
the symbol $A \Subset B$ 
means $A \subset {\rm Int}_XB$. For topological spaces $X$ and $Y$ the symbol $C^0(X,Y)$ denotes the set of $C^0$ maps $f : X \to Y$.  
The symbol $I$ always denotes the interval $[0,1]$.  The set of $C^0$ homotopies is denoted by 
${\rm Homot}(X, Y) := C^0(X \times I, Y)$.
Note that ${\rm Homot}(X) := {\rm Homot}(X, X)$ is a monoid with respect to the composition of homotopies : 
$FG := (F_tG_t)_{t \in I}$ for $F= (F_t)_{t \in I}$ and $G = (G_t)_{t \in I} \in {\rm Homot}(X)$. 
For any $f \in C^0(X,Y)$, the constant homotopy $(f)_{t \in I} \in {\rm Homot}^r(X, Y)$ is denoted by the same symbol $f$. 

In this paper a $C^r$ $n$-manifold $(r \in \oIZ_{\geq 0})$ means a separable metrizable $C^r$ manifold possibly with boundary of dimension $n$. 
A closed/open manifold means a compact/noncompact manifold without boundary.  
For $C^r$ manifolds $M$ and $N$ the symbol $C^r(M, N)$ denotes the set of $C^r$ maps $f : M \to N$. 
The set of $C^r$ homotopies is denoted by ${\rm Homot}^r(M, N) := C^r(M \times I, N)$ and ${\rm Homot}^r(M) = {\rm Homot}^r(M, M)$. 

More generally, for sets $X$, $Y$, 
when a notion of $C^r$ maps in $F(X, Y)$ is defined, we use the symbol $C^r(X,Y)$ to denote 
the subset of $C^r$ maps $f : X \to Y$. 
For example, when $M, N, L$ are $C^r$ manifolds, 
from a viewpoint of diffeology, we can define \\
\hsp \hsh $C^r(L, C^r(M, N)) := \{ \phi \in F(L, C^r(M, N)) \mid \widetilde{\phi} : M \times L \to N : C^r \}$. \\
In this article we use this notation in this sense, so that $C^r(L, C^r(M, N)) \cong C^r(M \times L, N)$ canonically.
 
We use the following standard notations for various objects parametrized by $I$ (paths, homotopies etc.)

\bnot\label{not_path} Let $X$ be a set. \\
{[\,I\,]} Let $\alpha, \beta, \gamma \in F(I, X)$ and ${\cal A} \subset F(I, X)$. 
\vspace*{-2mm} 
\benum
\item When $\alpha(1) = \beta(0)$, 
the concatenation $\alpha \ast \beta \in F(I, X)$ is defined by 
$(\alpha \ast \beta)(t) = 
\left\{ \hspace*{-1mm} \bary[c]{ll}
\alpha(2t) & (t \in [0,1/2]) \\[1mm]
\beta(2t-1) & (t \in [1/2,1]).
\eary \right.$
\item The inverse $\gamma^- \in F(I, X)$ is defined by $\gamma^{-}(t) = \gamma(1-t)$ $(t \in I)$. \\
More generally, for $\sigma \in F(I,I)$ we have the change of parameter $\gamma_\sigma := \gamma \circ \sigma$. 
\item A homotopy $\eta : \alpha \simeq \beta$ in $F(I, X)$ means a map $\eta \in F(I, F(I, X))$ with $\eta(0) = \alpha$ and $\eta(1) = \beta$. 
\bit 
\itemI If $\eta(t)(0) = \eta(0)(0)$ and $\eta(t)(1) = \eta(0)(1)$ $(t \in I)$, then $\eta$ is called a homotopy relative ends and denoted by $\eta : \alpha \simeq_\ast \beta$. 
\itemII If $\eta(t) \in {\cal A}$ $(t \in I)$, then $\eta$ is called a homotopy in ${\cal A}$. 
\eit 
\item For $p,q \in X$ we set \ \ 
${\cal A}_p := \{ \alpha \in {\cal A} \mid \alpha(0) = p \}$ \ and \ ${\cal A}_{p,q} := \{ \alpha \in {\cal A} \mid \alpha(0) = p, \alpha(1) = q \}$. 
\eenum 
{[II]} When subspaces of $C^r$ paths and $C^r$ path homotopies in $X$ : \\
\hsp \hsh $C^r(I, X) \subset F(I, X)$ \ and \ $C^r(I, C^r(I, X)) \subset F(I, F(I,X))$ \ are defined, \\
\hsh \ the following notions related to paths/loops are defined. 
\benum
\item ${\cal P}^r(X) := C^r(I, X)$, \ \ $\Omega^r(X, p) := {\cal P}(X)_{p,p}$, \ \ 
$\Omega^r_0(X, p) := \{ \alpha \in \Omega^r(X, p) \mid \alpha \simeq_\ast \ast \ \mbox{(in $C^r$)} \}$ 
\item (the universal covering) \\
$X_p := \{ \alpha(1) \mid \alpha \in {\cal P}^r(X)_p \} \subset X$, \ \ 
$\widetilde{X}_p := {\cal P}^r(X)_p\big/\simeq_\ast$, \ \ $\pi : \widetilde{X}_p \to X_p$ : $\pi([\alpha]) = \alpha(1)$ 
\item (the fundamental group) \\
$\pi_1(X,p) := \Omega^r(X, p)\big/\simeq_\ast$, \ \ $\pi_1(X,p) \car \widetilde{X}_p$ : $[\alpha] \cdot [\gamma] = [\alpha \ast \gamma]$ 
\eenum 
\enot 

In particular, these general notions are applied to the sets $X = C^r(M, N)$ and \\
\hsp \hsh ${\rm Homot}^r(M, N) = C^r(M \times I, N) \cong C^r(I, C^r(M, N))$ \ for $C^r$ manifolds $M$, $N$ and $r \in \oIZ_{\geq 0}$. \\
Any $C^r$ homotopy $F = (F_t)_{t \in I} : M \times I \to N$ is identified with 
the associated $C^r$ map $\widehat{F} : I \to C^r(M, N) : t \mapsto F_t$. 
Under this correspondence, the above notations for $\widehat{F}$ coincide with the usual notations for $F$. 
(We need a change of parameter for the $C^r$ concatenation $\alpha \ast \beta$ $(r \geq 1)$.) 
The same remark also applied to $C^r$ homotopies between $C^r$ homotopies (relative to ends) \\
\hsp \hsh $\Phi \in C^r(M \times I^2, N) \cong C^r(I, C^r(M \times I, N))$, \ $\Phi : F \simeq G$ \ ($\Phi : F \simeq_\ast G$). \\
The set ${\rm Homot}^r(M)$ is a monoid with respect to the composition of homotopies and 
${\rm Homot}^r(M)_{\id, \id}$ is a submonoid of ${\rm Homot}^r(M)$.

A $C^r$ isotopy $F$ of $M$ is a $C^r$ diffeomorphism $F : M \times I \to M \times I$ of the form 
$F(x,t) = (F_t(x), t)$ $((x,t) \in M \times I)$. Sometimes we also denote $pr_M F : M \times I \to M$ by the same symbol $F$. 
The symbols ${\rm Diff}^r(M)$ and ${\rm Isot}^r(M)$ denote the groups of $C^r$ diffeomorphisms and $C^r$ isotopies of $M$ respectively. 
Under the group epimorphism \ $R : {\rm Isot}^r(M) \,\lra\hspace{-5.5mm}\lra \, {\rm Diff}^r(M)$ : $R(F) = F_1$, 
any subset ${\cal I}$ of ${\rm Isot}^r(M)$ induces the associated subset ${\cal G}_{\cal I} := R({\cal I})$ of ${\rm Diff}^r(M)$. 
When $\id_{M \times I} \in {\cal I}$, we have $\id_M \in {\cal G}_{\cal I}$ and the identity components of ${\cal I}$ and ${\cal G}_{\cal I}$ are defined by 
\ \ ${\cal I}_0 := \{ F \in {\cal I} \mid F_0 = \id_M \}$ \ \ and \ \ $({\cal G}_{\cal I})_0 := R({\cal I}_0)$. 
If ${\cal I} < {\rm Isot}^r(M)$, then ${\cal I}_0 < {\rm Isot}^r(M)$ and ${\cal G}_{\cal I}, ({\cal G}_{\cal I})_0 < {\rm Diff}^r(M)$. 

The support of $F \in {\rm Isot}^r(M)$ is defined by \\
\hsppp ${\rm supp}\,F := cl_M\big( \cup_{t \in I} {\rm supp}\,F_t \big)
= pr_M ({\rm supp}\,[F : M \times I \to M \times I])$. 

\bfact\label{fact_supp} \mbox{} Suppose $f, g \in {\rm Diff}^r(M)$, $F, G \in {\rm Isot}^r(M)$ and $U \in {\cal O}(M)$. 
\benum
\item $\supp\,fg \,\subset\, \supp\,f \cup \supp\,g$, \ \ $\supp\,f^{-1} = \supp\,f$, \ \ $\supp\,gfg^{-1} = g(\supp\,f)$ \\
$f = \id$ on $U$ \LRA $\supp\,f \subset M - U$  
\item $\supp\,FG \,\subset\, \supp\,F \cup \supp\,G$, \ \ $\supp\,F^{-1} = \supp\,F$, \ \ $\supp\,gFg^{-1} = g(\supp\,F)$ \\ 
$F = \id$ on $U \times I$ \LRA $\supp\,F \subset M - U$ 
\eenum 
\efact

\subsection{Groups of diffeomorphisms and isotopies of manifold pairs} \mbox{} 

For $r \in \oIZ_{\geq 0}$ a $C^r$ manifold pair means a pair $(M, N)$ of a $C^r$ manifold $M$ and a proper $C^r$ submanifold $N$ of $M$ (cf. Hirsch \cite{Hir}). In the case $r = 0$, we assume that a $C^0$ submanifold $N$ is locally flat in $M$. 
For a $C^r$ manifold pair $(M, N)$, a $C^r$ diffeomorphism of $(M, N)$ means a $C^r$ diffeomorphism $f$ of $M$ onto itself with $f(N) = N$, and 
a $C^r$ isotopy $F$ of $(M,N)$ means a $C^r$ isotopy $F$ of $M$ with $F_t(N) = N$ $(t \in I)$. 

By the symbols ${\rm Diff}^r(M, N)$ and ${\rm Isot}^r(M, N)$ we denote the groups of $C^r$ diffeomorphisms and $C^r$ isotopies of $(M, N)$ respectively. 
They are joined by the group epimorphism \\
\hspp $R : {\rm Isot}^r(M, N) \,\lra\hspace{-5.5mm}\lra \, {\rm Diff}^r(M, N)$ : $R(F) = F_1$. \\
The identity components of these groups are defined by \\
\hsp ${\rm Isot}^r(M, N)_0 := \{ F \in {\rm Isot}^r(M, N) \mid F_0 = \id_M \}$ \ \ and \ \ ${\rm Diff}^r(M, N)_0 = R\big({\rm Isot}^r(M, N)_0\big)$. \\
More generally, for any subset ${\cal I}$ of ${\rm Isot}^r(M, N)$ and the associated subset ${\cal G}_{\cal I} := R({\cal I})$ of ${\rm Diff}^r(M, N)$, 
, the identity components of ${\cal I}$ and ${\cal G}_{\cal I} := R({\cal I})$ are defined by \\
\hspace*{5mm} ${\cal I}_0 := \{ F \in {\cal I} \mid F_0 = \id_M \} = {\cal I} \cap {\rm Isot}^r(M, N)_0$ \ \ and \ \ $R({\cal I})_0 := R({\cal I}_0)$. 

In the subsequent sections we need to impose appropriate conditions ${\cal C}$ for $f \in {\rm Diff}^r(M, N)$ and 
${\cal C}_I$ for $F \in {\rm Isot}^r(M, N)$. We use the following notations to denote the corresponding subsets : \\ 
\hsf \hsf  
\btab[t]{l}
${\rm Isot}^r(M, N; {\cal C}) := \{ F \in {\rm Isot}^r(M, N) \mid F_t : {\cal C} \ (t \in I) \}$, \hsf   
${\rm Isot}^r(M, N; {\cal C}_I) := \{ F \in {\rm Isot}^r(M, N) \mid F : {\cal C}_I \}$, \\[2mm]
\hsp ${\rm Diff}^r(M, N; {\cal C}) := \{ f \in {\rm Diff}^r(M, N) \mid f : {\cal C} \}$, \\[2mm] 
\hsp ${\rm Diff}^r(M, N; {\cal C})_0 := R\big( {\rm Isot}^r(M, N; {\cal C})_0\big)$, \hsh 
${\rm Diff}^r(M, N; {\cal C}_I)_0 := R\big({\rm Isot}^r(M, N; {\cal C}_I)_0\big)$. 
\etab 
\vskip 1mm 

\begin{compl}\label{compl_diff_pair} \mbox{}  
\benum[(1)] 
\item In the case [$F$ : ${\cal C}_I$] \LLRA [$F_t : {\cal C} \ (t \in I)$], we have \\
\hsp \hsh ${\rm Isot}^r(M, N; {\cal C}) = {\rm Isto}^r(M, N; {\cal C}_I)$ \ \ 
and \ \ ${\rm Diff}^r(M, N; {\cal C})_0 = {\rm Diff}^r(M, N; {\cal C}_I)_0$. \\ 
For example, for a subset $X$ of $M$, we see that [$F = \id$ on $X \times I$] \LLRA [$F_t = \id$ on $X$ $(t \in I)$]. 
Hence, in this case, we can denote these conditions by the same symbol [rel $X$] and 
define the subgroups ${\rm Isot}^r(M, N; {\rm rel}\,X)$, ${\rm Diff}^r(M, N; {\rm rel}\,X)$ and ${\rm Diff}^r(M, N; {\rm rel}\,X)_0$ without any ambiguity.  
We also define the subgroups associated to the neighborhood system ${\cal U}(X)$ of $X$ in $M$, \\
\hspp ${\rm Isot}^r(M, N; {\rm rel}\,{\cal U}(X)) := \bigcup \{ {\rm Isot}^r(M, N; {\rm rel}\,U) \mid U \in {\cal U}(X) \}$ \ \ and \\
\hspp ${\rm Diff}^r(M, N; {\rm rel}\,{\cal U}(X)) := \bigcup \{ {\rm Diff}^r(M, N; {\rm rel}\,U) \mid U \in {\cal U}(X) \}$. 

\bit 
\itemI The group ${\rm Isot}^r(M, N)$ includes the following normal subgroups : \\
\hsp ${\rm Isot}^r(M; {\rm rel}\,N)$, \ 
${\rm Isot}^r(M; {\rm rel}\,N)_0$, \ 
${\rm Isot}^r(M; {\rm rel}\,{\cal U}(N))$, \ 
${\rm Isot}^r(M; {\rm rel}\,{\cal U}(N))_0$

\itemII The group ${\rm Diff}^r(M, N)$ includes the following normal subgroups : \\
\hsp ${\rm Diff}^r(M; {\rm rel}\,N)$, \ 
${\rm Diff}^r(M; {\rm rel}\,N)_0$, \ 
${\rm Diff}^r(M; {\rm rel}\,{\cal U}(N))$, \ 
${\rm Diff}^r(M; {\rm rel}\,{\cal U}(N))_0$
\eit 

\item 
However, we have to be careful for conditions on support, since \\ 
{}[\,${\rm supp}\,F$ : compact] \ 
$\relsblr{\lra \,}{\not\hspace*{-1.5mm}\lla}$ \ [${\rm supp}\,F_t$ : compact $(t \in I)$] 
\ \ and \ \ [\,${\rm supp}\,F \Subset X$] 
$\relsblr{\lra \,}{\not\hspace*{-1.5mm}\lla}$ \ [${\rm supp}\,F_t \Subset X$ $(t \in I)$] \\
for a subset $X \subset M$. 
For these conditions we use the following notations to denote the corresponding subgroups.
Note that \ $\supp\,FG \subset \supp\,F \cup \supp\,G$ \ and \ $\supp\,F^{-1} = \supp\,F$ \ for $F, G \in {\rm Isot}^r(M)$. \\
\hsh 
\btab[t]{l}
${\rm Isot}^r_c(M, N) := \{ F \in {\rm Isot}^r(M, N) \mid {\rm supp}\,F : \mbox{compact} \}$ \\[2mm]
${\rm Diff}^r_c(M, N) := \{ f \in {\rm Diff}^r(M, N) \mid \supp\,f : \mbox{compact} \}$, \hsf   
${\rm Diff}^r_c(M, N)_0 := R\big({\rm Isot}^r_c(M, N)_0\big)$, \\[2mm]
${\rm Isot}^r(M, N; {\rm supp}\, \Subset X) := \{ F \in {\rm Isot}^r(M, N) \mid {\rm supp}\,F \Subset X \}$, \\[2mm] 
${\rm Diff}^r(M, N; {\rm supp}\, \Subset X) := \{ f \in {\rm Diff}^r(M, N) \mid \supp\,f \Subset X \}$, \\[2mm] 
${\rm Diff}^r(M, N; {\rm supp}\, \Subset X)_0 := R\big({\rm Isot}^r(M, N; {\rm supp}\, \Subset X)_0\big)$. 
\etab 
\eenum 
\end{compl} 

As usual, the symbol $N$ is omitted from these notations when $N = \emptyset$.  

\bnot\label{not_G} The groups ${\rm Isot}^r(M, N)_0$ and ${\rm Diff}^r(M, N)_0$ include the following normal subgroups. 
\bit 
\itemI ${\cal J}_{(c)}(M, N) := {\rm Isot}^r_{(c)}(M; \, {\rm rel}\  {\cal U}(N))_0 
= {\rm Isot}^r_{(c)}(M, N; {\rm supp} \subset M - N)_0 \ \subset {\rm Isot}^r_{(c)}(M, N)_0$ 
\itemII $\bary[t]{@{}l@{ \ }l}
{\cal G}_{(c)}(M, N) & := R({\cal J}_{(c)}(M, N)) = {\rm Diff}^r_{(c)}(M; \, {\rm rel}\  {\cal U}(N))_0 
= {\rm Diff}^r_{(c)}(M, N; {\rm supp} \subset M - N)_0 \\[1.5mm]
& \subset {\rm Diff}^r_{(c)}(M, N)_0
\eary$ 
\eit 
\enot 

The restriction induces  the following group homomorphisms : \\ 
\hsp $P_I : {\rm Isot}^r(M, N)_0 \lra \ {\rm Isot}^r(N)_0$ :  $P_I(F) = F|_{N \times I}$, \hsp ${\rm ker}\,P_I = {\rm Isot}^r(M; {\rm rel}\,N)_0$, \\ 
\hsp $P : {\rm Diff}^r(M, N)_0 \lra {\rm Diff}^r(N)_0$ : $P(f) = f|_N$, \hsp ${\rm ker}\,P= {\rm Diff}^r(M, N)_0 \cap {\rm Diff}^r(M; {\rm rel}\,N)$. \\
Moreover, for each connected component $L$ of $N$, we also have the following group homomorphisms : \\
\hsp $P_{I,L} : {\rm Isot}^r(M, N)_0 \lra \ {\rm Isot}^r(L)_0$ :  $P_{I, L}(F) = F|_{L \times I}$, \hsp ${\rm ker}\,P_{I,L} = {\rm Isot}^r(M, N; {\rm rel}\,L)_0$, \\ 
\hsp $P_L : {\rm Diff}^r(M, N)_0 \lra {\rm Diff}^r(L)_0$ : $P(f) = f|_L$, \hsp ${\rm ker}\,P_L= {\rm Diff}^r(M, N)_0 \cap {\rm Diff}^r(M; {\rm rel}\,L)$. 

\bnot\label{not_M}
For $r \in \oIZ_{\geq 0}$ and $n \geq \ell \geq 0$, let ${\cal M}^r(n, \ell)$  denote the class of pairs $(M, N)$ such that 
$M$ is a $C^r$ $n$-manifold possibly with boundary and $N$ is a (possibly empty) $C^r$ $\ell$-submanifold possibly with boundary of $M$ such that 
$N \in {\cal F}(M)$ and $N$ is proper in $M$ (i.e., $\partial N = N \cap \partial M$). 
Note that, if $(M, N) \in {\cal M}^r(n,n)$, then $M$ is a disjoint union of $N$ and another $C^r$ $n$-manifold possibly with boundary and 
${\rm Diff}^r(M, N)_0 = {\rm Diff}^r(M)_0$. 
\enot 

\subsection{\bf Isotopy extension theorem} \mbox{}

The isotopy extension theorem and its variants are essential to obtain some extension, restriction and factorizaitons of isotopies {\rm (cf. \cite[Ch 8]{Hir}))}. 
For example, when $(M, N) \in {\cal M}^r(n, \ell)$ $(r \in \oIZ_{\geq 2}, n \geq \ell \geq 0)$, Isotopy extension theorem
implies the following restriction maps are group epimorphisms. \\
\hsp $P_{I} : {\rm Isot}^r_c(M, N; \supp \subset {\rm Int}\,M)_0 \lra {\rm Isot}^r_c(N; \supp \subset {\rm Int}\,N)_0$ : $P_{I}(F) = F|_{N \times I}$ \\
\hsp $P : {\rm Diff}^r_c(M, N; \supp \subset {\rm Int}\,M)_0 \lra {\rm Diff}^r_c(N; \supp \subset {\rm Int}\,N)_0$ : $P(f) = f|_N$. 

The next statement is a simple example of restriction of an isotopy. 

\bprop\label{prop_isot_ext}  

Suppose $M$ is a $C^r$ manifold $(r \in \oIZ_{\geq 2})$, 
$F \in {\rm Isot}^r(M)_0$, $K \in {\cal K}(M)$, $K \Subset U \subset M$ and $F(K \times I) \Subset U$. 
Then, there exists $G \in {\rm Isot}^r_c(M; {\rm rel}\,M - U)_0$ with $G = F$ on $K \times I$ and $\supp\,G \subset \supp\,F$. 
\eprop 

\bpfb 
Take a compact neighborhood $V$ of $K$ in $M$ with $F(K \times I) \Subset V \subset U$. 
Let $(X, \partial_t) \in {\cal X}^{r-1}(M \times I)$ denote the velocity vector field of $F$ on $M \times I$. 
Take $\rho \in C^r(M; [0,1])$ with $\rho \equiv 0$ on $M - V$ and $\rho \equiv 1$ on $F(K \times I)$, 
and consider the vector field $(\rho X, \partial_t) \in {\cal X}^{r-1}(M \times I)$.  
Since $V \in {\cal K}(M)$ and $X \equiv 0$ on $(M - \supp\,F) \times I$, 
the vector field $(\rho X, \partial_t)$ induces $G \in {\rm Isot}^r_c(M; {\rm rel}\,M - (V \cap \supp\,F))_0$ 
which satisfies the required conditions.  
\epf 

\section{Fragmentation Lemma and the norms $\eta$, $\zeta$, $clb$ on ${\rm Diff}^r_c(M,N)_0$} 

In this section we discuss relations and finiteness of the fragmentation norm $\eta$, the conjugation-generated norms $\zeta = (\zeta_g)_g$ 
and the commutator length $clb$ with support in balls on the diffeomorphism group ${\rm Diff}^r_c(M,N)_0$ of a manifold pair $(M, N)$. 
The arguments are direct extensions
 of the absolute case, 
except that we are concerned with the weak simplicity, since the group ${\rm Diff}^r_c(M,N)_0$ is not simple. 

\subsection{Commutator length with support in balls}  \mbox{}

Let $\| \ \|_0$ denote the standard Euclidean norm in $\IR^n$. 
The model ball pair $(\IB^n, \IB^\ell) \subset (\IR^n, \IR^\ell)$ is defined by 
$\IB^n := \{ x \in \IR^n \mid \| x \|_0 \leq 1 \}$ and $\IB^\ell := \IB^n \cap \IR^\ell$. 

Suppose $(M, N) \in {\cal M}^r(n, \ell)$ $(r \in \oIZ_{\geq 0}, n > \ell \geq 1)$. 

\bnotb For $n$-balls in $M$ we use the following notations. \\[1mm] 
\hsh ${\cal B}^r(M) :=$ the set of $C^r$ $n$-balls in ${\rm Int}\,M$. Here, a $C^0$ $n$-ball means a locally flat $n$-ball. \\
\hsh ${\cal B}^r(M, N) := {\cal B}^r(M, N)_1 \cup {\cal B}^r(M, N)_2$ : \\
\hsp \hsh ${\cal B}^r(M, N)_1 := \{ D \in {\cal B}^r(M) \mid \mbox{$(D, D \cap N)$ is $C^r$ diffeomorphic to $(\IB^n, \IB^\ell)$}.\}$ \\
\hsp \hsh ${\cal B}^r(M, N)_2 := \{ D \in {\cal B}^r(M) \mid D \subset M - N \}$ \\
\hsh ${\cal F}{\cal B}^r(M, N)$ : the set of finite disjoint unions $D$ of $n$-balls $D_i \in {\cal B}^r(M, N)$ $(i \in [m])$ \\
\hsh ${\cal F}{\cal B}^r(M, N)_1$ : the set of finite disjoint unions $D$ of $n$-balls $D_i \in {\cal B}^r(M, N)_1$ $(i \in [m])$ \\
\hsh ${\cal F}{\cal B}^r(M, N)_1^\ast := \{ D \in {\cal F}{\cal B}^r(M, N)_1 \mid D = \bigcup_{i \in [m]} D_i, \ D_i \in {\cal B}^r(M, N_i)_1\}$, \\
\hspace*{50mm} when $N$ has finitely many connected components $N_i$ $(i \in [m])$. 

Here, we include the subclass ${\cal B}^r(M, N)_2$ to incorporate 
the absolute case where $N = \emptyset$ in Definition~\ref{def_clb} below (cf. Fact~\ref{fact_cl-clb}\,(2)). 
For example, to estimate $clbd\,{\rm Diff}^r_c(M, N)_0$ we need the term $clbd\,{\rm Diff}^r_c(M - N)_0$  as explained in Introduction. 
\enot 

\bdefn\label{def_clb} \mbox{}
\benum 
\item Let ${\cal S}_b \equiv {\cal S}_b(M, N) := \bigcup \{ {\rm Diff}^r(M, N; \supp \Subset D)_0^c \mid D \in {\cal F}{\cal B}^r(M, N) \}$. 
Each element of ${\cal S}_b$ is called a commutator with support in balls in ${\rm Diff}^r(M, N)_0$. 
Note that ${\cal S}_b$ is symmetric and conjugation invariant in ${\rm Diff}^r(M, N)_0$. 

\item The conjugation invariant norm induced from ${\cal S}_b$ is called 
the commutator length with support in balls and denoted by $clb : {\rm Diff}^r(M, N)_0 \to \oIZ_{\geq 0}$. 

\item If ${\cal D} \equiv {\rm Diff}^r(M, N)_0 \vartriangleright {\cal G} \supset {\cal S}_b$, then 
we have $q_{{\cal G}, {\cal S}_b} : {\cal G} \to \oIZ_{\geq 0}$. 
Since $clb$ is just the extension of $q_{{\cal G}, {\cal S}_b}$ by $\infty$, we denote $q_{{\cal G}, {\cal S}_b}$ by the same symbol $clb$. 
For example, this applies to ${\cal G} = {\rm Diff}^r_c(M, N)_0$ or ${\rm Diff}^r_c(M, N; \supp \subset {\rm Int}\,M)_0$. 
\eenum
\edefn 

\bfact\label{fact_cl-clb} Let ${\cal D} \equiv {\rm Diff}^r_c(M, N; \supp \subset {\rm Int}\,M)_0$ and consider the norms $cl, clb : {\cal D} \to \oIZ_{\geq 0}$.   
\benum 
\item We have \ ${\cal D}^c \supset {\cal S}_b$, \ $cl \leq clb$ \ and \ $cld\,{\cal D} \leq clbd\,{\cal D}$. 
\item When $M$ is connected and $N \neq \emptyset$, the following holds. 
\bit 
\itemI Any $D \in {\cal F}{\cal B}^r(M, N)$ admits $E \in {\cal F}{\cal B}^r(M, N)_1$ with $D \Subset E$. 
\itemII ${\cal S}_b = \bigcup \{ {\rm Diff}^r(M, N; \supp \Subset D)_0^c \mid D \in {\cal F}{\cal B}^r(M, N)_1 \}$. 
\eit 
\eenum 
\efact 

Consider the following condition $P(r,n,\ell)$ for $r \in \oIZ_{\geq 0}$ and $n > \ell \geq 0$. 

\begin{assumption_A} $P(r,n,\ell)$ : \ 
The group ${\rm Diff}^r_c(\IR^n, \IR^\ell)_0$ is perfect. 
\end{assumption_A}

\bexp\label{exp_assumption} \mbox{}  
\benum 
\item $P(r,n,\ell)$ holds in the following case : \hsh $r = \infty, \ n > \ell \geq 1$ \hspace*{5mm} (Theorem I, cf.\,\cite{AF1})
\item $P(\infty,n,0)$ $(n \in \IZ_{\geq 1})$ does not hold, since $H_1({\rm Diff}^\infty_c(\IR^n, 0)_0) \cong \IR$ (\cite{Fu1}).  
\eenum 
\eexp 

\begin{compl} In \cite[Theorem 2]{Ry2} it is stated that $P(r,n,\ell)$ holds for $r \in \oIZ_{\geq 1}, \ n > \ell \geq 1$, \ $r \neq n+1$. 
We could not have confirmed this statement yet.

\end{compl}

\bfact\label{fact_ball} \mbox{} 
\benum 
\item Let ${\cal E} := {\rm Diff}^r_c(\IR^n, \IR^\ell)_0$ and consider \ $cl, clb : {\cal E} \to \oIZ_{\geq 0}$ : \\ 
\,(i) \ ${\cal E}^c = {\cal S}_b(\IR^n, \IR^\ell)$, \ $cl = clb$ \ and \ $cld\, {\cal E} = clbd\, {\cal E}$ \hsh (ii) \ $cld\, {\cal E} \leq 2$ \ under Assumption $P(r,n,\ell)$.
\vskip 1mm 

\item Suppose $(M,N) \in {\cal M}^r(n, \ell)$ $(r \in \oIZ_{\geq 0}, n > \ell \geq 1)$ and $D \in {\cal F}{\cal B}^r(M, N)$. 
Let ${\cal D} \equiv {\rm Diff}^r_c(M, N; \supp \subset {\rm Int}\,M)_0$ and ${\cal D}_D := {\rm Diff}^r(M, N; \, \supp \Subset D)_0$. 
Consider $clb : {\cal D} \to \oIZ_{\geq 0}$ and $cl : {\cal D}_D \to \oIZ_{\geq 0}$. 
\bit 
\itemI ${\cal S}_b \supset {\cal D}_D^c$, \ $clb|_{{\cal D}_D} \leq cl_{{\cal D}_D}$ \ and \ $clbd\,{\cal D}_D \leq cld\,{\cal D}_D$. 
\itemII $cld\,{\cal D}_D \leq 2$  \ under Assumption $P(r,n,\ell)$.
\eit 
\eenum 
\efact 

\subsection{Fragmentation norm on ${\rm Diff}^r_c(M,N)_0$} \mbox{} 

Suppose $(M, N) \in {\cal M}^r(n, \ell)$ $(r \in \oIZ_{\geq 0}, n \geq \ell \geq 0)$. 

\bdefn\label{def_frag} A fragmentation of $f \in {\rm Diff}^r_c(M, N; \supp \subset {\rm Int}\,M)_0$ of length $m \in \IZ_{\geq 0}$ 
means a factorization \ $f = f_1 \cdots f_m$ \ such that 
$f_i \in {\rm Diff}^r(M, N; \supp \Subset D_i)_0$ for some $D_i \in {\cal F}{\cal B}^r(M, N)$ $(i \in [m])$. 
\edefn 

\bprop\label{prop_frag} 
Each $f \in {\rm Diff}^r_c(M, N; \supp \subset {\rm Int}\,M)_0$ has a fragmentation. 
\eprop 

\bpf 
We only treat the case that $r \geq 1$. In the $C^0$ case we need to apply Deformation lemma in \cite{EK} carefully. 
For $C \in {\cal K}(M)$ we set ${\cal D}_C := {\rm Diff}^r(M, N; {\rm supp}\, \subset C)_0$. 

Given $f$, there exists $F \in {\rm Isot}^r_c(M, N; \supp \subset {\rm Int}\,M)_{\id, f}$. 
Set $K := \supp\,F \in {\cal K}({\rm Int}\,M)$. 
Then, for any $C^1$ neighborhood ${\cal U}$ of $\id_M$ in ${\cal D}_K$ 
we can find a subdivision $0 = t_0 < t_1 < \cdots < t_k = 1$ of the interval $I$ such that $g_i := F_{t_i}F^{\ -1}_{t_{i-1}} \in {\cal U}$ \ $(i \in [k])$. 
Since $f = g_k \cdots g_2g_1$, it suffices to show the following claim.

\bit 
\item[$(\ast)$] For any $K \in {\cal K}({\rm Int}\,M)$ and any $C^1$ neighborhood ${\cal V}$ of $\id_M$ in ${\cal D}_K$ 
there exists a $C^1$ neighborhood ${\cal U}$ of $\id_M$ in ${\cal D}_K$ such that any $g \in {\cal U}$ admits a fragmentation in ${\cal V}$. 
\eit 

Given $K$, take $K_i, L_i \in {\cal K}(M)$, $D_i \in {\cal B}^r(M, N)$ $(i \in [m])$ such that $K = \bigcup_{i \in [m]} K_i$ and 
$K_i \Subset L_i \Subset D_i$.
For any ${\cal V}$, inductively we can find a sequence \ 
${\cal V} \equiv {\cal V}_m \supset \cdots \supset {\cal V}_1 \supset {\cal V}_0 \equiv {\cal U}$ \ 
of $C^1$ neighborhoods of $\id_M$ in ${\cal D}_K$ which satisfies the following conditions for each $i \in [m]$. 
\bit 
\item[$(\flat)$] Any $h \in {\cal V}_{i-1}$ has a factorization $h = g'h'$ such that \\
(a) $g', h' \in {\cal V}_i$, $g' \in {\cal D}_{K \cap L_i}$, \ (b) $g' = h$, \ $h' = \id$ on $K_i$ \ and \ (c) if $h(x) = x$, then $g'(x) = h'(x) = x$. 
\eit 
In fact, if we take ${\cal V}_{i-1}$ sufficiently small for ${\cal V}_i$ and identify $(D_i, D_i \cap N)$ with $(\IB^n, \IB^\ell)$, 
then $g' \in {\cal D}_{K \cap L_i}$ and $G' \in {\rm Isot}^r(M, N; \supp \subset K \cap L_i)_{0, g'}$ can be defined on $L_i$ by \\
\hsp $g'(x) = (1 - \rho(x))x + \rho(x)h(x)$ \ \ and \ \  $G'(x,t) = (1-t)x + tg'(x) = (1 - t\rho(x))x + t\rho(x)h(x)$, \\
where $\rho \in C^r(M, [0,1])$ with $\rho \equiv 1$ on $K_i$ and $\supp\,\rho \Subset L_i$. 

Then, for any $g = h_0 \in {\cal U}$ there exists $g_i, h_i \in {\cal V}_i \subset {\cal D}_K$ $(i \in [m])$ such that 
\bit 
\itema $h_{i-1} = g_ih_i$, \ $g_i \in {\cal D}_{K \cap L_i}$, \ (b) $h_i = \id$ on $K_i$ \ and \ 
(c) if $h_{i-1}(x) = x$, then $g_i(x) = h_i(x) = x$. 
\eit 
It is seen that $g = g_1 \cdots g_ih_i$ $(i \in [m])$ by (a) and $h_i = \id$ on $\bigcup_{j \in [i]} K_j$ $(i \in [m])$ by (b), (c). 
Hence, it follows that 
$h_m = \id$, \ $g = g_1 \cdots g_m$ \ and \ $g_i \in {\cal V}_i \cap {\cal D}_{K \cap L_i} \subset{\cal V} \cap  {\rm Diff}^r(M, N; \supp \Subset D_i)_0$ $(i \in [m])$. 
\epf 

\bdefn The fragmentation norm on ${\rm Diff}^r_c(M,N)_0$ is defined by \\
\hsh $\eta : {\rm Diff}^r_c(M, N; \supp \subset {\rm Int}\,M)_0 \lra \IZ_{\geq 0}$ : \ 
$\eta(f) := \min \{ k \in \IZ_{\geq 0} \mid \mbox{$f$ has a fragmentation of length $k$.} \}$ 
\edefn 

Proposition~\ref{prop_frag} and Fact~\ref{fact_ball}\,(2) deduce the following conclusion. 

\bprop $clb \leq 2\eta$ on ${\rm Diff}^r_c(M, N; \supp \subset {\rm Int}\,M)_0$ \ under Assumption $P(r,n,\ell)$. 
\eprop 

\subsection{Relation between the norms $\zeta$ and $clb$} \mbox{} 

Next we conform the relation between the norms $\zeta$ and $clb$ on the group ${\rm Diff}^r_c(M, N)_0$ (cf. \cite{BIP, Ts3}). 
Suppose $(M, N) \in {\cal M}^r(n, \ell)$ $(r \in \oIZ_{\geq 0}, n > \ell \geq 1)$. 
We consider the norms $\zeta = (\zeta_g)_g$ and $clb$ on the group ${\rm Diff}^r_c(M, N; \supp \subset {\rm Int}\,M)_0$. 
For notational simplicity, let \\
\hsp ${\cal D} = {\rm Diff}^r_c(M, N; \supp \subset {\rm Int}\,M)_0$ \ \ and \ \ ${\cal D} > {\cal D}_{A} := {\rm Diff}^r_c(M, N; \supp \subset A)_0$ \ for $A \subset {\rm Int}\,M$. \\
Note that $g({\cal D}_A)g^{-1} = {\cal D}_{g(A)}$ \ for any $g \in {\cal D}$. 

\begin{fact}\label{lem_C_g^4}  $($cf.~\cite{BIP}, \cite[Lemma 3.1]{Ts2}, {etc}.$)$ \ \ 
Suppose $g \in {\cal D}$, $A \subset {\rm Int}\,M$ and $g(A) \cap A = \emptyset$. 

\benum 
\item[{\rm (1)}] 
$({\cal D}_{A})^c \subset C_g^4$ \ in ${\cal D}$. 
More precisely, for any $a, b \in {\cal D}_{A}$ the following identity holds :  \\[0.5mm] 
\hspace*{10mm} $[a,b] =g\big(g^{-1}\big)^c g^{bc}\big(g^{-1}\big)^b$ \ in \ ${\cal D}$, \hsh 
where $c := a^{g^{-1}} \in {\cal D}_{g^{-1}(A)}$. Note that $cb = bc$.  
\vskip 0.5mm 
\item[{\rm (2)}] $({\cal D}_{A'})^c \subset C_g^4$ \ in ${\cal D}$, if $A' \subset {\rm Int}\,M$ and 
$h(A') \subset A$ for some $h \in {\cal D}$.  
\eenum 
\end{fact} 

\begin{proof} (2) 
Let $k := h^{-1}gh \in {\cal D}$. Then, $k(A') \cap A' = \emptyset$ and by (1) we have $({\cal D}_{A'})^c \subset C_{k}^4 = C_g^4$.  
\end{proof} 

For each $L \in {\cal C}(N)$ we have the group epimorphism $P_L : {\cal D} \lra {\rm Diff}^r_c(L; \supp \subset {\rm Int}\,L)_0$ : $P_L(f) = f|_L$. 
Let ${\cal K \cal P} \equiv {\cal K \cal P}(M, N) := \bigcup_{L \in {\cal C}(N)} {\rm Ker}\,P_L \subset {\cal D}$, 
where we set ${\cal K \cal P}(M, N) := \{ \id_M \}$ when $N = \emptyset$.  
Note that $f \in {\cal D} - {\cal K \cal P}$ iff $f \neq \id_M$ and $f|_L \neq \id_L$ for each $L \in {\cal C}(N)$. 

\begin{prop}\label{prop_zeta_leq_4clb_pi} 
Suppose $M$ is connected and $N$ has finitely many connected components $N_i$ $(i \in [m])$. 
\benum
\item[{\rm (1)}] ${\cal S}_b(M, N) \subset C_g^4$ and $\zeta_g \leq 4 \,clb \leq clbd\,{\cal D}$ in ${\cal D}$ for any $g \in {\cal D} - {\cal K \cal P}$. 
\item[{\rm (2)}] \,{\rm (i)} 
Under Assumption $P(r,n,\ell)$ 
\bit 
\item[] {\rm (a)} \ $\zeta_g \leq 4\,clb \leq 8 \eta < \infty$ in ${\cal D}$ for any $g \in {\cal D} - {\cal K \cal P}$ and \ 
\item[] {\rm (b)} \ ${\cal D}$ is weakly simple relative to ${\cal K \cal P}$. 
\eit 
{\rm (ii)} If $clbd\,{\cal D} < \infty$, then {\rm (a)}   
${\cal D}$ is uniformly weakly simple relative to ${\cal K \cal P}$ 
and {\rm (b)} ${\cal D}$ is bounded. 
\eenum
\end{prop}

\begin{proof}
(1) Given $g$, since $g \neq \id_M$ and $g|_{N_i} \neq \id_{N_i}$ $(i \in [m])$, 
there exists $E = \bigcup_{i \in [m]_+} E_i \in {\cal F}{\cal B}^r(M, N)$ 
such that $E_i \in {\cal B}^r(M, N_i)_1$ $(i \in [m])$, $E_0 \in {\cal B}^r(M, N)_2$ and $g(E) \cap E = \emptyset$.
Since $M$ is connected, for any $D \in {\cal F}{\cal B}^r(M, N)$ we can find 
$H \in {\rm Isot}^r_c(M, N; \supp \subset {\rm Int}\,M)_0$ such that $h : = H_1 \in {\cal D}$ satisfies $h(D) \subset E$. 
Then, from Fact~\ref{lem_C_g^4}\,(2) we have $({\cal D}_{D})^c \subset C_g^4$. 
Hence, ${\cal S}_b(M, N) = \bigcup \{ ({\cal D}_{D})^c \mid D \in {\cal F}{\cal B}^r(M, N) \} \subset C_g^4$
and Fact~\ref{fact_cl+zeta}\,(1)(ii) implies $\zeta_g \leq 4 \,clb$. 
\bit
\item[(2)] \,(i) The statements follow from Proposition~\ref{prop_frag} and Fact~\ref{fact_w-simple}\,(2). \\
(ii) The assertion (b) follows from Fact~\ref{fact_w-simple}\,(4).
\eit 
 \vspace*{-7mm} 
\end{proof}

\section{Quasimorphisms induced from the rotation angle on a cricle} 

\subsection{Rotation angle on a circle} \mbox{}

Suppose $S$ is a $C^r$ circle $(r \in \oIZ_{\geq 0})$. 
To introduce a signed angle on $S$ with the total angle $=$ 1, we fix a $C^r$ diffeomorphism $\IR/\IZ \approx S$. 
This determines a commutative $C^r$ group structure $+$ on $S$.
The $C^r$ universal covering $\pi_S : \IR \to \IR/\IZ \approx S$ is a group epimorphism. 
For any $C^r$ manifold $X$ the space $C^r(X, S)$ forms a commutative $C^r$ group under the sum of maps $(\phi + \psi)(x) = \phi(x) + \psi(x)$ $(x \in X)$. 

\bnot\label{not_S} \mbox{} 
\benum 
\item Each $a \in S$ induces the translation $\theta_a(x) = x+a$ on $S$. 
We call these translations the rotations on $S$. 
We have the subgroup $L_S = \{ \theta_a \mid a \in S\}$ of ${\rm Diff}^r(S)_0$ and 
the $C^r$ group isomorphism $\theta : S\to L_S : a \mapsto \theta_a$. 

\item 
For $a \in S$ let $\e_a \in C^r(X, S)$ denote the constant map $\e_a(x) = a$ $(x \in X)$. 
We set $\phi + a := \phi + \e_a = \theta_a \circ \phi$ for $\phi \in C^r(X, S)$.

\item The path space ${\cal P}(S) = C^0(I, S)$ forms a group. 
We have $\alpha + \beta$, \ $-\alpha$, \ $\alpha + a \equiv \theta_a \circ \alpha$, \ $\e_a$, etc.  \ for $\alpha, \beta \in {\cal P}(S)$ and $a \in S$. 
For each $\sigma \in \IR$ we define $\delta_\sigma \in {\cal P}^r(S)_0$ by $\delta_\sigma(t) = \pi_S(\sigma t)$ $(t \in I)$. 
\eenum 
\enot 

\bdefn\label{defn_lambda} 
Each $C^0$ path $\gamma \in {\cal P}(S)$ has a lift $\widetilde{\gamma} \in {\cal P}(\IR)$ (i.e., $\pi_S \widetilde{\gamma} = \gamma$).  
The rotation angle of $\gamma$ is defined by $\lambda_S(\gamma) = \widetilde{\gamma}(1) - \widetilde{\gamma}(0)$, 
which is independent of the choice of the lift $\widetilde{\gamma}$. 
\edefn 

The function $\lambda \equiv \lambda_S : {\cal P}(S) \to \IR$ has the following properties. 

\bfact\label{fact_lambda} Suppose $\alpha, \beta, \gamma \in {\cal P}(S)$. 
\benum 
\item 
\bit 
\itemI $\lambda(\alpha \ast \beta) = \lambda(\alpha) + \lambda(\beta)$ if $\alpha(1) = \beta(0)$. 
\itemII $\lambda(\gamma_-) = - \lambda(\gamma)$.  \hsh (iii) $\lambda(\alpha + a) = \lambda(\alpha)$ \ $(a \in S)$. 
\eit 
\item 
\bit 
\itemI $\alpha \simeq_\ast \beta$ iff $\alpha(0) = \beta(0)$ and $\lambda(\alpha) = \lambda(\beta)$.   

\itemII $\lambda(\gamma_\sigma) = \lambda(\gamma)$ for any $\sigma \in {\cal C}^0(I, I)$ with $\sigma(0) = 0$ and $\sigma(1) = 1$. 
\eit 
\item 
\bit 
\itemI $\lambda : {\cal P}(S) \to \IR$ is a group epimorphism.  

\itemII The map $\delta : \IR \to {\cal P}^r(S)_0$: $\sigma \mapsto \delta_\sigma$ is 
a group monomorphism and $\lambda(\delta_\sigma) = \sigma$. 
\eit 

\item If $\gamma \in \Omega(S)$ (i.e., $\gamma$ is a closed path ($\gamma(0) = \gamma(1)$)), then 
\bit 
\itemI $\lambda(\gamma) = \deg \gamma \in \IZ$ \hsh 
(ii) \ $\lambda(f\gamma) = (\deg f)\, \lambda(\gamma)$ for any $f \in {\cal C}^0(S,S)$. 
\eit 
\item $|\lambda(f\gamma) - \lambda(\gamma)| < 1$ for any $f \in {\rm Diff}^0(S)_0 = {\rm Diff}^0_+(S)$. 
\eenum
\efact

\bpfb 
\benum 
\item[(3)] (i) Since $\alpha + \beta \simeq_\ast (\alpha \ast \alpha(1)) + (\beta(0) \ast \beta) = (\alpha + \beta(0)) \ast (\beta + \alpha(1))$, it follows that \\
\hsp $\lambda(\alpha + \beta) = \lambda((\alpha + \beta(0)) \ast (\beta + \alpha(1))) 
= \lambda(\alpha + \beta(0)) + \lambda(\beta + \alpha(1)) = \lambda(\alpha) + \lambda(\beta)$
\vskip 1mm 
\item[(5)] Take a lift $\widetilde{f} \in {\rm Diff}^0(\IR)_0 = {\rm Diff}^0_+(\IR)$ of $f$ and a lift $\widetilde{\gamma} \in {\cal P}(\IR)$ of $\gamma$. 
It follows that 
\bit 
\itema $\widetilde{f}$ is monotonically increasing and $\widetilde{f}(x+1) = \widetilde{f}(x) + 1$ $(x \in \IR)$ and 
\itemb $\widetilde{f} \widetilde{\gamma} \in  {\cal P}(\IR)$ is a lift of $f\gamma \in {\cal P}(S)$ so that  
$\lambda(f\gamma) = \widetilde{f}(\widetilde{\gamma}(1)) - \widetilde{f}(\widetilde{\gamma}(0))$. 
\eit 
Let $n := \lfloor{\lambda(\gamma)}\rfloor \in \IZ$. 
Then, $\widetilde{\gamma}(1) = \widetilde{\gamma}(0) + \lambda(\gamma) \in \widetilde{\gamma}(0) + [n, n+1)$ and 
$\widetilde{f}(\widetilde{\gamma}(1)) \in \widetilde{f}(\widetilde{\gamma}(0)) + [n, n+1)$ by (a). 
Hence, $\lambda(\gamma), \lambda(f\gamma) \in [n,n+1)$ so that $|\lambda(f\gamma) - \lambda(\gamma)| < 1$. 
\eenum 
\vspace*{-7mm} 
\epf 

\begin{remark}\label{rem_lambda} 
If $f \in C^0(S, S)$ and ${\rm deg}\,f = n$, then any lift $\widetilde{f} \in C^0(\IR, \IR)$ of $f$ 
satisfies $\widetilde{f}(t+1) = \widetilde{f}(t) + n$ $(t \in \IR)$. 
In particular, if ${\rm deg}\,f = 1$ (i.e., $f$ is an orientation-preserving homotopy equivalence), then $\widetilde{f}(t+1) = \widetilde{f}(t) + 1$. 
However, $|\lambda(f\gamma) - \lambda(\gamma)|$ may take arbitrarily large value if $f \not\in {\rm Diff}^0_+(S)$ 
(i.e., the lift $\widetilde{f}$ is not monotonically increasing). 
\end{remark} 

\subsection{Rotation angle of homotopies on a circle} \mbox{}

Next we consider the rotation angle for homotopies on the circle $S$. 
We fix a distinguished point $p$ in $S$. Then, each $F \in {\rm Homot}(S)$ induces a path $F_p := F(p, \ast) \in {\cal P}(S)$ and 
we have the function \\
\hsp \hsh $\mu_p : {\rm Homot}(S) \to \IR$ : $\mu_p(F) = \lambda_S(F_p)$. 

\bfact\label{fact_mu} The function $\mu \equiv \mu_p$ has the following properties. Suppose $F, G, H \in {\rm Homot}(S)$.
\benum 
\item (i) \ $\mu(G \ast H) = \mu (G) + \mu(H)$ if $G_1 = H_0$. 
\hsp (ii) \ $\mu(F^-) = - \mu(F)$ 
\item $\mu(G) = \mu(H)$ if $G \simeq_\ast H$.  
\item $\mu(G_\sigma) = \mu(G)$ for any $\sigma \in {C^0(I, I)}$ with $\sigma(0) = 0$ and $\sigma(1) = 1$. 

\item $\mu(GH) = \mu(G_0H) + \mu(GH_1) = \mu(GH_0) + \mu(G_1H)$. 
\item $|\mu(fG) - \mu(G)| <1$ for any $f \in {\rm Diff}^0(S)_0$. 
\item $|\mu(GH) - \mu(G) - \mu(H)| < 1$ if $G_1 \in {\rm Diff}^0(S)_0$ and $H_0 = \id_S$. 
\item $|\lambda(G_q) - \lambda(G_p)| < 1$ if $G_0, G_1 \in {\rm Diff}^0(S)_0$ and $q \in S$. 
\eenum 
\efact 

\bpfb The assertions (1) $\sim$ (6) follow from Fact~\ref{fact_lambda} and the following observations : \\
\hsp (1)\ (i) \ $(G \ast H)_p = G_p \ast H_p$, \ (ii) \ $(F^-)_p = (F_p)_-$, \hsh 
(2) \ $\Phi : G \simeq_\ast H$ $\Lra$ $\Phi_p : G_p \simeq_\ast H_p$, \\
\hsp (3) \ $G_\sigma \simeq_\ast G$, \hsh 
(4) \ $GH \simeq_\ast (G_0H) \ast (GH_1) \simeq_\ast (GH_0) \ast (G_1H)$, \\ 
\hsp (5) \ $(fG)_p = f G_p$, \hsh (6) \ $\mu(GH) - \mu(G) - \mu(H) = \mu(G_1H) - \mu(H)$. 

(7) There exists points $\widetilde{p}, \widetilde{q} \in \IR$ with $\pi_S(\widetilde{p}) = p$, $\pi_S(\widetilde{q}) = q$ and $\widetilde{q} \in [\widetilde{p}, \widetilde{p}+1)$. 
Take a lift $\widetilde{G} \in {\rm Homot}(\IR)$ of $G$ (i.e., $\pi_S \widetilde{G} = G (\pi_S \times \id_I)$). 
Then, $\widetilde{G}_{\widetilde{p}}, \widetilde{G}_{\widetilde{q}} \in {\cal P}(\IR)$ are lifts of $G_p, G_q \in {\cal P}(S)$ respectively.
Since $\widetilde{G}_t$ is a lift of $G_t$ for each $t \in I$ and $G_0, G_1 \in {\rm Diff}^0(S)_0$, it follows that \ \ for $t = 0,1$, 
\bit 
\itema $\widetilde{G}_t$ is monotonically increasing \ and \ $\widetilde{G}_t(x+1) = \widetilde{G}_t(x) + 1$ $(x \in \IR)$ \ \ and 
\itemb $\widetilde{G}_t(\widetilde{p}) \leq \widetilde{G}_t(\widetilde{q})  < \widetilde{G}_t(\widetilde{p} + 1) = \widetilde{G}_t(\widetilde{p}) + 1$. 
\eit  
Hence, we have \hsh 
$\widetilde{G}_1(\widetilde{p}) - (\widetilde{G}_0(\widetilde{p}) + 1) 
< \widetilde{G}_1(\widetilde{q}) - \widetilde{G}_0(\widetilde{q}) 
< \widetilde{G}_1(\widetilde{p}) + 1 - \widetilde{G}_0(\widetilde{p})$. \\ 
Since \hsh $\lambda(G_p) = \widetilde{G}_{\widetilde{p}}(1) - \widetilde{G}_{\widetilde{p}}(0) 
= \widetilde{G}_1(\widetilde{p}) - \widetilde{G}_0(\widetilde{p})$ \ \ and \ \ 
$\lambda(G_q) = \widetilde{G}_1(\widetilde{q}) - \widetilde{G}_0(\widetilde{q})$, it follows that \\
\hsp $\lambda(G_p) - 1 < \lambda(G_q) < \lambda(G_p) + 1$ \ \ so that \ \ $|\lambda(G_q) - \lambda(G_p)| < 1$. 
\epf 

The subset ${\rm Homot}(S)_{\id, \id}$ is a submonoid of ${\rm Homot}(S)$ 
and the restriction $\mu_p : {\rm Homot}(S)_{\id, \id} \to \IZ$ is a monoid epimorphism. 

For isotopies we have the following conclusions. 

\begin{lem}\label{lem_mu_isot} Suppose $F, G \in {\rm Isot}^0(S)_0$ and $h \in {\rm Diff}^0(S)_0$. \\[1mm] 
\hsh \btab[t]{lll}
{\rm (1)} \ $|\mu(hF) - \mu(F)| < 1$ \hsh & {\rm (2)} \ $|\mu(Fh) - \mu(F)| < 1$ \hsh & {\rm (3)} \ $|\mu(FG) - \mu(F) - \mu(G)| < 1$ \\[2mm]
{\rm (4)} \ $\mu(F^{-1}) = - \mu(f^{-1}F)$ & {\rm (5)} \ $|\mu(F) + \mu(F^{-1})| < 1$ & {\rm (6)} \ $|\mu([F,G])| < 3$ 
\etab 
\end{lem}

\begin{proof} 
The assertions (1) and (3) follow from Fact~\ref{fact_mu}\ (5) and (6) respectively.
\bit 
\item[(2)] Since $(Fh)_p = F_{h(p)}$, we have \ $|\mu(Fh) - \mu(F)| = |\lambda(F_{h(p)}) - \lambda(F_p)|< 1$ \ by Fact~\ref{fact_mu}\ (7). 
\eit 

Let $f := F_1, g := G_1 \in {\rm Diff}^0(S)_0$. 

\bit \item[(4)] Since $\id_S = F^{-1} F \simeq_\ast (F^{-1} \ast f^{-1})(\id \ast F) = F^{-1} \ast (f^{-1} F)$, it follows that \\
\hsh $0 = \mu(\id_S) = \mu(F^{-1} \ast (f^{-1} F)) = \mu(F^{-1}) + \mu(f^{-1}F)$ \ \ and \ \ $\mu(F^{-1}) = - \mu(f^{-1}F)$. 

\item[(5)] $|\mu(F) + \mu(F^{-1})| = |\mu(F) - \mu(f^{-1}F)| < 1$ \ \ by (1).
\vskip 1mm 

\item[(6)] It follows that 
\vskip 0.5mm 
\bit 
\itemI $\bary[t]{@{}l@{ \ }l}
[F,G] = FGF^{-1}G^{-1} & \simeq_\ast 
(F \ast f)(G \ast g)(\id_S \ast F^{-1})(G^{-1} \ast g^{-1}) \\[1.5mm]
& = (F G\, \id_{S^1} G^{-1}) \ast (fgF^{-1}g^{-1}) = F \ast (fgF^{-1}g^{-1}), 
\eary$ 
\vskip 1mm 
\itemII $\bary[t]{@{}l@{ \ }l}
\mu([F,G]) & = \mu(F) + \mu(fgF^{-1}g^{-1}) \\[1.5mm]
& = \mu(F) + \mu(F^{-1}) + \mu(fgF^{-1}g^{-1}) - \mu(F^{-1}g^{-1}) + \mu(F^{-1}g^{-1}) - \mu(F^{-1}), 
\eary$
\vskip 1.5mm 
\itemiii (a) \ $|\mu(F) + \mu(F^{-1})| < 1$ \ \ by (4), \hsh (b) \ $|\mu(F^{-1}g^{-1}) - \mu(F^{-1})| < 1$ \ \ by (2) \ and \\[1mm] 
(c) \ $\mu(fgF^{-1}g^{-1}) - \mu(F^{-1}g^{-1})| < 1$ \ \ by Fact~\ref{fact_mu}\,(5). 
\eit 
\vskip 1mm 
This implies that $|\mu([F,G])| < 3$. 
\eit
\vspace*{-7mm} 
\end{proof} 

\bfact\label{fact_mu_homo} \mbox{} 
\benum
\item 
\bit 
\itemI The function $\mu_p : {\rm Isot}^r(S)_0 \to \IR$ is a surjective quasimorphism with defect 1. 
\itemII The function $\mu_p| : {\rm Isot}^r(S)_{\id, \id} \to \IZ$ is a group epimorphism. 
\eit 
\item We have the following diagrams.  
The first one corresponds to $(\sharp)$ in Fact~\ref{fact_setting}. 
The conditions in Setting~\ref{setting_diagram} follow from (1) and Fact~\ref{fact_mu}\,(4). 
This diagram reduces to the second one 
by Fact~\ref{fact_reduction}\,(2) with taking $L := \Omega_0({\rm Diff}^r(S), \id) < \Omega({\rm Diff}^r(S), \id) = {\rm Isot}^r(S)_{\id,\id}$. \\[2mm] 
\hspace*{-5mm} 
\raisebox{-3mm}{
$\xymatrix@M+1pt{
{\rm Isot}^r(S)_{\id,\id} \ar@{->>}[d]_-{\mu_p|} \ar@{}[r]|*{\subset} & {\rm Isot}^r(S)_0 \ar@{->>}[d]^-{\mu_p} \ar@{->>}[r]^-R 
& {\rm Diff}^r(S)_0 \ar@{->>}[d]^{\widehat{\mu_p}} \\
\IZ \ar@{}[r]|*{\subset} & \IR \ar@{->>}[r] & \IR/\IZ
}$}
\hsh 
\raisebox{-3mm}{
$\xymatrix@M+1pt{
\pi_1({\rm Diff}^r(S), \id) \ar@{->>}[d]_-{\widetilde{\mu_p|}} \ar@{}[r]|*{\subset} 
& \widetilde{\rm Diff}{}^r(S)_0 \ar@{->>}[d]^-{\widetilde{\mu_p}} \ar@{->>}[r]^-{\widetilde{R}} 
& {\rm Diff}^r(S)_0 \ar@{->>}[d]^{\widehat{\mu_p}} \\
\IZ \ar@{}[r]|*{\subset} & \IR \ar@{->>}[r] & \IR/\IZ
}$}
\eenum 
\efact 

\subsection{Deformation of isotopies on a circle} \mbox{} 

This subsection reviews some basic deformation properties of isotopies on a circle. 
Suppose $S$ is a $C^r$ circle $(r \in \oIZ_{\geq 0})$ equipped with a $C^r$ universal covering $\pi_S : \IR \to \IR/\IZ \approx S$ and a distinguished point $p$. 

\bnot\label{not_rotation} \mbox{} 
\benum 
\item 
Each $\alpha \in {\cal P}^r(S)$ induces 
a $C^r$ isotopy $\theta_\alpha := (\theta_{\alpha(t)})_{t \in I} \in {\rm Isot}^r(S)$. Note that 
\bit 
\itemI $\theta_\alpha =\id_S + \alpha = \theta \circ \alpha \in {\cal P}^r(L_S) \hra {\rm Isot}^r(S)$, $\theta_\alpha(x,t) = x + \alpha(t)$ 
and that $\theta_\eta = (\theta_{\eta(s,t)})_{(s,t) \in I^2}$. 

\itemII The map $\theta : {\cal P}^r(S) \ni \alpha \lra \theta_\alpha \in {\rm Isot}^r(S)$ is a $C^r$ group monomorphism.  
\itemiii Each $C^r$ path-homotopy $\eta : \alpha \simeq_\ast \beta$ rel ends induces a $C^r$ isotopy 
$\theta_\eta = (\theta_{\eta_t})_{t \in I} : \theta_\alpha \simeq_\ast \theta_\beta$  rel ends between isotopies in ${\rm Isot}^r(S)$. 
\eit 

\item We use the following notations : 
\bit 
\itemI $f + a := f + \e_a \equiv \theta_a \circ f \in C^r(S, S)$ for $f \in C^r(S, S)$ and $a \in S$ (and $\e_a \in C^r(S, S)$). 

\itemII $F + \alpha := \theta_\alpha \circ F = (F_t + \alpha(t))_{t \in I} \in {\rm Homot}(S)$ 
for $F \in {\rm Homot}(S)$ and $\alpha \in {\cal P}(S)$. 

$F + a := F + \e_a = \theta_a \circ F = (F_t + a)_{t \in I}$ for $a \in S$ (and $\e_a \in {\cal P}(S)$). 
\eit 
\eenum
\enot 

\bfact\label{fact_inverse_lambda} \mbox{} 
Suppose $F, G \in {\rm Homot}(S)$, $f \in C^0(S,S) \subset {\rm Homot}(S)$ and $\alpha \in {\cal P}(S)$, $a \in S$. 
\benum 
\item $(F + G)_p = F_p + G_p$, \hsh $(F + \alpha)_p = F_p + \alpha$, \hsh $(F + a)_p = F_p + a$, \\
$(\theta_\alpha f)_p = \alpha + f(p)$, \hsh $(f\theta_\alpha)_p = f(\alpha + p)$, \hsh $(F \theta_a)_p = F_{p+a}$ 

\item $\mu(F + G) = \mu(F) + \mu(G)$, \hsh $\mu(F + \alpha) = \mu(F) + \lambda(\alpha)$, \hsh $\mu(F + a) = \mu(F)$
\item $\mu(\theta_\alpha f) = \lambda(\alpha) = \mu(\theta_\alpha)$, \hsh $\mu(\theta_\alpha F) = \mu(F) + \lambda(\alpha)$, \hsh  
$\mu(f \theta_\alpha) = (\deg\,f) \lambda(\alpha)$ if $\alpha$ is a closed path. \\
$\mu(F \theta_a) = \lambda(F_{p+a})$, \hsh  
$\bary[t]{@{}l@{ \ }l}
\mu(F \theta_\alpha) & = \mu(F_0 \theta_\alpha) + \mu(F \theta_{\alpha(1)}) = \lambda(F_0(\alpha + p)) + \lambda(F_{p+\alpha(1)}) \\[2mm]
& = \mu(F \theta_{\alpha(0)}) + \mu(F_1 \theta_\alpha) =  \lambda(F_{p+\alpha(0)}) + \lambda(F_1(\alpha + p)) 
\eary$ 
\vskip 2mm 
\item Consider the map \ $\eta \equiv \eta_p : {\rm Homot}(S) \lra {\cal P}(S)$ : $\eta(F) = F_p - p$. \ We have \ (i) \ $\mu = \lambda\eta$ \\ 
 (ii) \ $\eta(\theta_\alpha) = \alpha$ \ $(\alpha \in {\cal P}(S))$ \ and \ 
(iii) \ the map $\rho := \theta \delta : \IR \to {\rm Isot}^r(S)$ is a right inverse of $\mu$. 
\eenum 
\efact 

\bfact\label{fact_diff_sdr} \mbox{} For $a,b \in S$ and $s \in \oIZ_{\geq 0}$, $s \leq r$ we set ${\cal D}^s_{a,b} := \{ f \in {\rm Diff}^s(S)_0 \mid f(a) = b \}$ and 
$\theta_{a,b} := \theta_{b-a} \in L_S$ (i.e., the rotation on $S$ with $\theta_{a,b}(a) = b$). 

\benum 
\item There exists a canonical $C^0$ strong deformation retraction (SDR) 
$\psi : {\cal D}^0_{0,0} \searrow \{ \id_S \}$ of ${\cal D}^0_{0,0}$ onto the singleton $\{ \id_S \}$. 
In fact, each $f \in {\cal D}^0_{0,0}$ has a unique lift $\widetilde{f} \in \widetilde{\cal D}^0_{0,0}$, where \\ 
\hsp \hsh $\widetilde{\cal D}^0_{0,0} := \{ h \in {\rm Diff}^0(\IR) \mid h(x+1) = h(x) + 1 \ (x \in \IR), \, h(0) = 0\}$, \\
and $\psi(f,t) \in {\cal D}^0_{0,0}$ $(t \in I)$ is covered by $\widetilde{f}_t := (1-t)\widetilde{f} + t \,\id_{\IR} \in \widetilde{\cal D}^0_{0,0}$ under the covering $\pi_S$. 

\item The group $S$ acts on ${\rm Diff}^0(S)_0$ by \ $a \cdot f = \theta_a f = f + a$ \ $(a \in S, f \in {\rm Diff}^0(S)_0)$.
The SDR $\psi$ extends to an $S$-equivariant $C^0$ SDR $\phi : {\rm Diff}^0(S)_0 \searrow L_S$ defined by \ 
$\phi(f,t) = \psi(f - f(0), t) + f(0)$. 

\item For any $a,b \in S$ we have a $C^0$ SDR \ $\psi_{a,b} : {\cal D}^0_{a,b} \searrow \{ \theta_{a,b} \}$ : $\psi_{a,b}(f,t) = \theta_b \psi(\theta_{-b} f\theta_a,t)\theta_{-a}$. 
In the partition ${\rm Diff}^0(S)_0 = \bigcup_{a \in S} {\cal D}^0_{0,a}$, 
the SDR $\phi$ restricts to $\psi_{0,a}$ on ${\cal D}^0_{0,a}$.  

\item There exists an $S$-equivariant $C^0$ diffeomorphism \ $\chi : S \times {\cal D}^0_{0,0} \approx {\rm Diff}^0(S)_0$ : \ $\chi(a, f) = f + a$, 
where $S$ acts on $S \times {\cal D}^0_{0,0}$ by $a \cdot (b, g) = (a+b, g)$ $(a \in S, (b,g) \in S \times {\cal D}^0_{0,0})$. 
Under the diffeomorphism $\chi$,  
the SDR $\phi$ corresponds with the SDR $\phi'  : S \times {\cal D}^0_{0,0} \searrow S \times \{ \id_S \}$ : $\phi'(a, g, t) = (a, \psi(g,t))$. 

\item For any $s \in \oIZ_{\geq 0}$, $s \leq r$, the $C^0$ SDR's $\psi$, $\phi$ and $\phi_{a,b}$ restrict to $C^s$ SDR's \\
\hsp $\psi| : {\cal D}^s_{0,0} \searrow \{ \id_S \}$, \ \ $\phi| : {\rm Diff}^s(S)_0 \searrow L_S$ \ \ and \ \  
$\psi_{a,b}| : {\cal D}^s_{a,b} \searrow \{ \theta_{a,b} \}$. \\
The restriction of $\chi$ also induces an $S$-equivariant $C^s$ diffeomorphism \ $\chi| : S \times {\cal D}^s_{0,0} \approx {\rm Diff}^s(S)_0$. 
\eenum 
\efact

\blem\label{lem_isot} \mbox{} 
Suppose $F \in {\rm Isot}^r(S)_{f,g}$. Let $\sigma :=\mu(F) \equiv \lambda(F_p)$. 
Then, there exists $G \in {\rm Isot}^r(S)_{f,g}$ such that $F \simeq_\ast G$ in ${\rm Isot}^r(S)_{f,g}$ and 
$G_p = \delta_\sigma + f(p)$. 
\elem  

\bpfb
Consider the path $\gamma := \delta_\sigma + f(p) - F_p \in {\cal P}^r(S)$. 
It follows that $\gamma(0) = 0$ since $\delta_\sigma(0) = 0$ and $F_p(0) = f(p)$ and 
that $\lambda(\gamma) = \lambda(\delta_\sigma) - \lambda(F_p) = \sigma - \sigma = 0$. % by Fact~\ref{fact_lambda}\,(3)(ii).  
Hence, by Fact~\ref{fact_lambda}\,(2)(i) 
we have a path-homotopy $\eta : \e_0 \simeq_\ast \gamma$.
 Let $G := F + \gamma = \theta_\gamma \circ F \in {\rm Isot}^r(S)_{f,g}$. 
 Then, it folllows that \\
 \hsh $F + \eta  := (F + \eta_t)_{t \in I} \equiv \theta_\eta \circ F : F \simeq_\ast G$ \ in ${\rm Isot}^r(S)_{f,g}$ \ and \ 
 $G_p = F_p + \gamma = \delta_\sigma + f(p)$ \ as required. 
\epf 

\bfact\label{fact_mu=0} \mbox{} 
\benum
\item If $F \in {\rm Isot}^r(S)_{\id,\id}$ and $\mu(F) = 0$, then $F \simeq_\ast \id_S$ in ${\rm Isot}^r(S)_{\id,\id}$. 

\item Suppose $F, G \in {\rm Isot}^r(S)_{f,g}$. Then, $F \simeq_\ast G$ in ${\rm Isot}^r(S)_{f,g}$ iff $\mu(F) = \mu(G)$. 
\item In Fact~\ref{fact_mu_homo}\,(2)  
\bit 
\itemI $\widetilde{\mu|} : \pi_1({\rm Diff}^r(S),\id_S) \to \IZ$ \ is a group isomorphism, 
\itemII the map $(\widetilde{R}, \widetilde{\mu}) : \widetilde{\rm Diff}{}^r(S)_0 \lra {\rm Diff}^r(S)_0 \times \IR$ : 
$(\widetilde{R}, \widetilde{\mu})([F]) = (F_1, \mu(F))$ \ is injective. 
\eit 
\eenum 
\efact 

\bpfb
\benum
\item By Lemma~\ref{lem_isot}  there exists $G \in {\rm Isot}^r(S)_{\id, \id}$ such that $F \simeq_\ast G$ and $G_p = \e_p$.
By Fact~\ref{fact_diff_sdr}\,(3) there exists a $C^r$ SDR $\psi_{p,p} : {\cal D}^r_{p,p} \searrow \{ \id_S \}$. 
Then, we have $\Lambda : G \simeq_\ast \id_S$ : $\Lambda_{s,t} = \psi_{p,p}(G_s, t)$ $((s,t) \in I^2)$. 
\item Suppose $\mu(F) = \mu(G) = \sigma$. By Lemma~\ref{lem_isot} 
there exists $F', G' \in {\rm Isot}^r(S)_{f,g}$ such that $F \simeq_\ast F'$, $G \simeq_\ast G'$ and $F'_p = \delta_\sigma + f(p) = G'_p$.
Then, $H := (G')^{-1}F' \in {\rm Isot}^r(S)_{\id,\id}$ and $H_p = \e_p$, so that $\mu(H) = 0$. 
Hence, by (1) $H \simeq_\ast \id_S$ and $F' = G'H \simeq_\ast G'$. This implies that $F \simeq_\ast G$. 
\item 
\bit 
\itemI By (1) we have the exact sequence \ \  
$1 \lra \Omega_0({\rm Diff}^r(S), \id) \ \subset \ {\rm Isot}^r(S)_{\id, \id} \stackrel{\mu|}{\lra} \IZ \lra 0$. 
\itemII The assertion follows from (2). 
\eit 
\eenum
\vspace*{-7.5mm} 
\epf 

\subsection{A quasimorphism on ${\rm Isot}^r(M, L)_0$ induced from the rotation angle on cricles $L$} \mbox{} 

\bnot\label{not_(M,L)}
For $r \in \oIZ_{\geq 0}$ and $n \geq 2, m \geq 1$, let ${\cal C}^r(n, m)$ denote the class of pairs $(M, L)$ such that \\
(i) $M$ is a $C^r$ $n$-manifold possibly with boundary and $L$ is a disjoint union of $m$ $C^r$ circles $S_i$ $(i \in [m])$ in ${\rm Int}\,M$
and 
(ii) each circle $S_i$ is equipped with a $C^r$ universal covering $\pi_i : \IR \to \IR/\IZ \approx S_i$ and a distinguished point $p_i$.
\enot 

Suppose $(M, L) \in {\cal C}^r(n, m)$ and $L = \bigcup_{i \in [m]} S_i$. 
The condition (ii) determines the rotation angle $\lambda_i : {\cal P}(S_i) \to \IR$ $(i \in [m])$ and induce  the following functions for each $i \in [m]$. 

\bnot\label{not_(M,L)} \mbox{} 
\benum
\item the surjective quasimorphism $\mu_i : {\rm Isot}^r(S_i)_0 \, \to\hspace{-4.5mm}\to \, \IR$ with $D_{\mu_i} = 1$, \\ 
\hsh  which restricts to a group epimorphism $\mu_i| : {\rm Isot}^r(S_i)_{\id,\id} \, \to\hspace{-4.5mm}\to \, \IZ$, 
\item the restriction map $P_{I, i} : {\rm Isot}^r_{(c)}(M, L)_0 \,\lra\hspace{-5.5mm}\lra \, {\rm Isot}^r(S_i)_0$ is a group epimorphism, \\
\hsh  which restricts to a group homomorphism $P_{I, i}| : {\rm Isot}^r_{(c)}(M, L)_{\id,\id} \lra {\rm Isot}^r(S_i)_{\id,\id}$, 
\item $\nu_i := \mu_i \circ P_{I, i} : {\rm Isot}^r_{(c)}(M, L)_0 \, \to\hspace{-4.5mm}\to \,  \IR$ is a surjective quasimorphism with $D_{\nu_i} = 1$, \\ 
\hsh which restricts to a group homomorphism $\nu_i| = \mu_i| \circ P_{I, i}| : {\rm Isot}^r_{(c)}(M, L)_{\id,\id} \to \IZ$. 
\eenum 
\enot 

The quasimorphisms $\nu_i$ $(i \in [m])$ are combined together to yield a vector-valued quasimorphism \\
\hspp $\nu := (\nu_i)_{i \in [m]} : {\rm Isot}^r_{(c)}(M, L)_0 \, \to\hspace{-4.5mm}\to \,  \IR^m$. 

\bfact\label{fact_nu} \mbox{} 
\benum 
\item $\nu := (\nu_i)_{i \in [m]}$ is a surjective vector-valued quasimorphism with $D_\nu^{(c)} = 1$,  \\
\hsh which restricts to a group homomorphism $\nu| = (\nu_i|)_{i \in [m]} : {\rm Isot}^r_{(c)}(M, L)_{\id,\id} \lra \IZ^m$. 

\item The quasimorphism $\nu$ satisfies the condition $(\ast)$ in Setting~\ref{setting_diagram} : \\
\hsp $\nu(GH) = \nu(G) + \nu(H)$ \ for any $G \in {\rm Isot}^r_{(c)}(M,L)_{\id,\id}$ and $H \in {\rm Isot}^r_{(c)}(M,L)_0$. \\
In fact, working with $\nu^0 : {\rm Isot}^0_{(c)}(M, L)_0 \, \to\hspace{-4.5mm}\to \,  \IR^m$, it is seen that \\
\hspp $GH \simeq_\ast G \ast H$ \ and \ $\nu(GH)  = \nu^0(G \ast H) = \nu(G) + \nu(H)$.

\item $C_\nu^{(c)} := \sup \| \nu({\rm Isot}^r_{(c)}(M,L)_0^{\ c}) \|\leq 3$
\eenum 
\efact 

\bpfb (3) The group homomorphism $P_{I, i}$ restricts to the map $P_{I, i} : {\rm Isot}^r(M, L)_0^{\ c} \,\lra\hspace{-5.5mm}\lra \, {\rm Isot}^r(S_i)_0^{\ c}$. 
Hence, for any $H \in {\rm Isot}^r(M,L)_0^{\ c}$ it follows that  
$|\nu_i(H)| = |\nu_i P_{I, i}(H)| < 3$ $(i \in [m])$ by Lemma~\ref{lem_mu_isot}\,(6) and that 
$\| \nu(H) \| = \ds \max_{i \in [m]} |\nu_i(H)| < 3$. This implies $C_\nu \leq 3$.
\epf 

Hence, we can apply the notations and the whole conclusions in Subsection~\ref{subsec_lattice}, which are listed below. 

\bprop\label{prop_nu_diagram} \mbox{} Suppose $(M, L) \in {\cal C}^r(n, m)$ $(r \in \oIZ_{\geq 0}$, $n \geq 2, m \geq 1)$ and $L = \bigcup_{i \in [m]} S_i$. 
\benum 
\item[{\rm (1)}] The quasimorphism $\nu$ is incorpolated into the following diagram : \\[1mm] 
\hsp \raisebox{-8mm}{$(\dagger)$} \hspace*{-6mm} 
\raisebox{0mm}{
$\xymatrix@M+1pt{
& {\rm Isot}^r_{(c)}(M,L)_{\id,\id} \ar@{->>}[d]_-{\nu|} \ar@{}[r]|*{\subset} & {\rm Isot}^r_{(c)}(M,L)_0 \ar@{->>}[d]^-{\nu} \ar@{->>}[r]^-{R \ \ } 
& {\rm Diff}^r_{(c)}(M,L)_0 \ar@{->>}[d]^{\widehat{\nu}_{(c)}} \\
 \ar@{}[r]|*{\ \ \IZ^m \ \ >} & A_{(c)} \ar@{}[r]|*{\subset} & \IR^m \ar@{->>}[r] & \IR^m/A_{(c)}
}$} \hsf 
\btab[t]{l} 
\mbox{} \\[10mm] 
$\ell_{(c)} := {\rm rank}\,A_{(c)} \leq m$
\etab 
\vskip 3mm 
\bit 
\itemI For $f \in {\rm Diff}^r_{(c)}(M,L)_0$ and $F \in R^{-1}(f)$ \\
\hsp $R^{-1}(f) = F\,{\rm Isot}^r_{(c)}(M,L)_{\id,\id}$ \ \ and \ \  $\widehat{\nu}_{(c)}(f) = \nu(R^{-1}(f)) = \nu(F) + A_{(c)}$. 
\eit 
\vskip 1mm 

\item[{\rm (2)}]  The diagram $(\dagger)$ reduces to the following form by {\rm Fact~\ref{fact_reduction}\,(2)} 
with $L = \Omega_0({\rm Diff}^r_{(c)}(M,L)_0)$. \\[1mm] 
\hsp \raisebox{-8mm}{$(\ddagger)$} \hspace*{2mm} 
$\xymatrix@M+3pt
{\relax
1 \ar[r] & \pi_1 {\rm Diff}^r_{(c)}(M,L)_0 \ar@{}[r]|*{\subset} \ar@{->>}[d]_-{\widetilde{\nu|}_{(c)}\,} 
& \widetilde{\rm Diff}{}^r_{(c)}(M,L)_0 \ar[r]^{\pi} \ar@{->>}[d]^-{\,\widetilde{\nu}_{(c)}} & {\rm Diff}^r_{(c)}(M,L)_0 \ar@{->>}[d]^-{\,\widehat{\nu}_{(c)}} \ar[r] & 1 \\ 
0 \ar[r] & A_{(c)} \ar@{}[r]|*{\subset} & \IR^m \ar[r]^-{} & \IR^m/A_{(c)} \ar[r] & 0 \\
}$ 
\vskip 2mm 

\item[{\rm (3)}]  
\bit 
\itemI $\theta_f^{(c)} := \min \| \widehat{\nu}_{(c)}(f) \|$ \ $(f \in {\rm Diff}^r_{(c)}(M,L)_0)$. \\[1mm] 
$\theta_\nu^{(c)} := \sup \{ \theta_z^{(c)} \equiv \min \| z \| \mid z \in \IR^m/A_{(c)} \} = \sup \{ \theta_f^{(c)} \mid f \in {\rm Diff}^r_{(c)}(M,L)_0 \}$.  
\vskip 1mm 
\itemII $D_\nu^{(c)} = 1$, \ $C_\nu^{(c)} \leq 3$. 
\vskip 1.5mm 
\itemiii {\rm (a)} $cl\,f \geq \sfrac{1}{\,4\,}(\theta_f^{(c)} + 1)$ \ for any $f \in {\rm Diff}^r_{(c)}(M,L)_0^{\ \times}$. 
\hsf {\rm (b)} \ $cld\,{\rm Diff}^r_{(c)}(M,L)_0 \geq \sfrac{1}{\,4\,}(\theta_\nu^{(c)} + 1)$. 
\eit 

\item[{\rm (4)}]  the case $\ell_{(c)} < m$ : \\
There exists a group epimorphism $\chi : \IR^m/A_{(c)} \,\mbox{$\to\hspace{-4.5mm}\to$} \, \IR$ such that 
$\chi \widehat{\phi} : {\rm Diff}^r_{(c)}(M,L)_0 \,\to\hspace{-4.5mm}\to \, \IR$ is a surjective quasimorphism. 
Hence, the group ${\rm Diff}^r_{(c)}(M,L)_0$ is unbounded and not uniformly perfect. 

\item[{\rm (5)}]  the case $\ell_{(c)} = m$ : 
\bit 
\itemI $\IZ^m/A_{(c)}$ is a finite abelian group. Let $k_i^{(c)} := {\rm ord}([e_i], \IZ^m/A_{(c)}) \in \IZ_{\geq 1}$ $(i \in [m])$ and $\ds k^{(c)} := \max_{i \in [m]} \,k_i^{(c)}$. 
\itemII $\theta_f^{(c)} \leq \theta_\nu^{(c)} \leq \sfrac{1}{\,2\,}k^{(c)}$ \ $(f \in {\rm Diff}^r_{(c)}(M,L)_0)$. 
\eit 

\item[{\rm (6)}]  the case $m = 1$ : \\
\hsf $A_{(c)} = k^{(c)}\IZ < \IZ$ for a unique $k^{(c)} \in \IZ_{\geq 0}$. \\
\hsf $\ell^{(c)} < m$ \ iff \ $k^{(c)} = 0$ \ $($or \,$\ell^{(c)} = m$ \ iff \ $k \geq 1)$. 
\hspace*{-7mm}
\smash{
\raisebox{10mm}{
$\xymatrix@M+1pt{
{\rm Isot}^r_{(c)}(M,L)_{\id,\id} \ar@{->>}[d]_-{\nu|} \ar@{}[r]|*{\subset} & {\rm Isot}^r_{(c)}(M,L)_0 \ar@{->>}[d]^-{\nu} \ar@{->>}[r]^-{R \ \ } 
& {\rm Diff}^r_{(c)}(M,L)_0 \ar@{->>}[d]^{\widehat{\nu}^{(c)}} \\
k^{(c)}\IZ \ar@{}[r]|*{\subset} & \IR \ar@{->>}[r] & \IR/k^{(c)}\IZ
}$}} \\[0mm]
{\rm [1]} the case $k^{(c)} = 0 :$ \\
\hsh The map \ $\widehat{\nu}_{(c)} : {\rm Diff}^r_{(c)}(M,L)_0 \, \to\hspace{-4.5mm}\to \ \IR/0\IZ \cong \IR$ \ is a surjective quasimorphism. 
Hence, \\ 
\hsh ${\rm Diff}^r_{(c)}(M,L)_0$ is unbounded and not uniformly perfect. \\
{\rm [2]} the case $k^{(c)} \geq 1 : $  
\bit 
\item[] {\rm (i)} \ $\theta_\nu^{(c)} = k^{(c)}/2$ \hsp {\rm (ii)} \ $cld\,{\rm Diff}^r_{(c)}(M,L)_0 \geq \sfrac{1}{\,8\,}(k^{(c)}+2)$  
\eit 
\eenum 
\eprop 

\bnot\label{not_diagram_Diff_G} 
In the diagram $(\dagger)$ in {\rm Proposition~\ref{prop_nu_diagram}\,$(1)$}, consider the normal subgroups \\
\hsp \hsf ${\cal J}_{(c)} \vartriangleleft {\rm Isot}^r_{(c)}(M, L)_0$, \ 
${\cal G}_{(c)} \vartriangleleft {\rm Diff}^r_{(c)}(M, L)_0$ \ and \ ${\cal L}_{(c)} \vartriangleleft {\rm Isot}^r_{(c)}(M,L)_{\id,\id}$, \ where \\[1mm] 
\hsp \hsh ${\cal J}_{(c)} := {\rm Isot}^r_{(c)}(M; \, {\rm rel}\  {\cal U}(L))_0 = {\rm Isot}^r_{(c)}(M, L; {\rm supp} \subset M - L)_0$ \hsp $($cf. {\rm Notation~\ref{not_G}}$)$, \\[1mm] 
\hsp \hsh ${\cal G}_{(c)} := R({\cal J}_{(c)}) = {\rm Diff}^r_{(c)}(M; \, {\rm rel}\  {\cal U}(L))_0 = {\rm Diff}^r_{(c)}(M, L; {\rm supp} \subset M - L)_0$, \\[1mm] 
\hsp \hsh ${\cal L}_{(c)} := {\cal J}_{(c)} \cap {\rm Isot}^r_{(c)}(M,L)_{\id,\id} 
= {\rm Isot}^r_{(c)}(M; \, {\rm rel}\  {\cal U}(L))_{\id,\id} = {\rm Isot}^r_{(c)}(M, L; {\rm supp} \subset M - L)_{\id,\id}$. \\[1mm] 
Since $\nu(FG) = \nu(F)$ for any $F \in {\rm Isot}^r_{(c)}(M,L)_0$ and $G \in {\cal J}_{(c)}$, 
by {\rm Fact~\ref{fact_reduction}\,(1)} with $J = {\cal J}_{(c)}$, the diagram $(\dagger)$ reduces to the following diagram. \\[1mm] 
\hsp \raisebox{-8mm}{$(\natural)$} \hspace*{-10mm} 
\raisebox{0mm}{
$\xymatrix@M+2pt{
& {\rm Isot}^r_{(c)}(M,L)_{\id,\id}\big/{\cal L}_{(c)} \ar@{->>}[d]_-{\widetilde{\nu|}\ } \ar@{^{(}->}[r] 
& {\rm Isot}^r_{(c)}(M,L)_0\big/{\cal J}_{(c)} \ar@{->>}[d]^-{\ \widetilde{\nu}} \ar@{->>}[r]^-{\widetilde{R} \ \ } 
& {\rm Diff}^r_{(c)}(M,L)_0\big/{\cal G}_{(c)} \ar@{->>}[d]^{\ \widetilde{\widehat{\nu}}_{(c)}} \\
 \ar@{}[r]|*{\ \ \IZ^m \ \ >} & A_{(c)} \ar@{}[r]|*{\subset} & \IR^m \ar@{->>}[r] & \IR^m/A_{(c)}
}$} \\[1mm] 
The diagram $(\natural)$ satisfies the conditions in {\rm Setting~\ref{setting_diagram}} and has the following properties. 
$(${\rm Fact~\ref{fact_reduction}\,(1)}$)$ \\[1mm] 
\hsp $D_{\widetilde{\nu}}^{(c)} = D_{\nu}^{(c)} = 1$, \ \ $C_{\widetilde{\nu}}^{(c)} = C_\nu^{(c)} \leq 3$, \ \ 
$\theta_{[f]}^{(c)} = \theta_f^{(c)}$ $(f \in {\rm Diff}^r_{(c)}(M,L)_0)$ \ and \ \ $\theta_{\widetilde{\nu}}^{(c)} = \theta_\nu^{(c)}$.
\enot  
 
\bprop\label{prop_nu_diagram_G} Let $(M, L) \in {\cal C}^r(n, m)$ $(r \in \oIZ_{\geq 0}$, $n \geq 2, m \geq 1)$. 
\benum 
\item[{\rm (1)}] The commutator lengths \ $cl, cl_{/{\cal G}_{(c)}} : {\rm Diff}^r_{(c)}(M,L)_0 \lra \oIZ_{\geq 0}$ 
and $\widetilde{cl} : {\rm Diff}^r_{(c)}(M,L)_0\big/ {\cal G}_{(c)} \lra \oIZ_{\geq 0}$ are related as follows. 
\hfill $($cf. {\rm Proposition~\ref{prop_lower_bound}}, {\rm Example~\ref{exp_cl/N}}$)$. \\[1mm] 
{\rm (a)} \ $cl\,f \geq cl_{/{\cal G}_{(c)}}\,f = \widetilde{cl}\,[f]$ \ $(f \in {\rm Diff}^r_{(c)}(M,L)_0)$. \\[1mm] 
{\rm (b)} \ $\widetilde{cl}\,[f] \geq \sfrac{1}{\,4\,}(\min \| \widehat{\nu}_{(c)}(f) \| + 1)$ \ $(f \in {\rm Diff}^r_{(c)}(M,L)_0 - {\cal G}_{(c)})$. \\[0mm] 
{\rm (c)} \ $cld\,{\rm Diff}^r_{(c)}(M,L)_0 \geq cl_{/{\cal G}_{(c)}}d\,{\rm Diff}^r_{(c)}(M,L)_0 = \widetilde{cl}d\ {\rm Diff}^r_{(c)}(M,L)_0\big/ {\cal G}_{(c)} \geq \sfrac{1}{\,4\,}(\theta_\nu^{(c)} + 1)$. \\[1mm]
\hspace*{100mm} $(\theta_\nu^{(c)} \equiv \sup \{ \min \| z \| \mid z \in \IR^m/A_{(c)} \})$
\vskip 1mm 
\item[{\rm (2)}] For $f \in {\rm Diff}^r_{(c)}(M,L)_0$ the following estimates hold. \hfill $($cf. {\rm Lemma~\ref{lem_q/S}}$)$ \\ 
\hsh $cl_{/\cal G_{(c)}}(f) \leq clb_{/\cal G_{(c)}}(f)$, \ \  $cl\,f \leq cl_{/\cal G_{(c)}}(f) + cld\,\cal G_{(c)}$ \ and \ $clb\,f \leq clb_{/\cal G_{(c)}}(f) + clbd\,\cal G_{(c)}$. 
\eenum  
\eprop  

\section{Factorization of isotopies along a union of circles} 

\subsection{Factorization of isotopies along a circle} \mbox{} 

In this subsection we discuss factorization of isotopies around a circle. 

\begin{setting}\label{setting_circle} 
Suppose $(M, S) \in {\cal C}^r(n, 1)$ $(r \in \oIZ_{\geq 2}, n \geq 2)$. 
Take a tubular neighborhood $D$ of $S$ in $M$, a normal unit disk bundle $E$ of $S$ in $M$ and 
a $C^r$ diffeomorphism $\chi : E \approx D$ such that $\chi = \id$ on $S$. 
The circle is equipped with a universal covering $\pi : \IR \, \to\hspace{-4.5mm}\to \, \IR/\IZ \approx S$ and a distinguished point $p$ in $S$, 
which determine the quasimorphism \\
\hspp $\nu_S : {\rm Isot}^r(M, S)_0 \lra \IR$ : $\nu_S(F) = \lambda_S(F_p)$.  
\end{setting}

\bnot\label{notation_D_J} We use the following notations. 
\benum 
\item The symbol ${\cal I}(S)$ denotes the set of (closed) arcs in $S$. 
For $J \in {\cal I}(S)$ we set $J' := S - {\rm Int}\,J \in {\cal I}(S)$. The rotation angle $\lambda(J)$ is defined as the length of $J$ in $S$ (or 
$\lambda(J) := |\lambda(\alpha)|$ for any parametrization $\alpha : I \to J$). 
Note that $J, K \in {\cal I}(S)$ and $\lambda(J), \lambda(K) > 1/2$, then ${\rm Int}\, J \cap {\rm Int}\, K \neq \emptyset$. 
\item For $J \in {\cal I}(S)$ a tubular neighborhood of $J$ in $(M, S)$ means a subset $D_J$ of $D$ of the form : \\
\hspp $D_J = \chi(sE|_J) \subset D$ \ \ for some $s \in (0,1]$. \\
The horizontal bounday of $D_J$ is defined by $\partial_+ D_J := \chi(s (\partial E)|_J)$, where 
$\partial E$ forms the normal unit sphere bundle of $S$ in $M$. We set $D_J|_A := \chi(sE|_A)$ for $A \subset J$. 
\eenum
\enot 

The next lemma follows from Proposition~\ref{prop_isot_ext} (Isotopy extension property). 

\blem\label{lem_isot_ext_D} Suppose $r \in \oIZ_{\geq 2}$, $J, K \in {\cal I}(S)$, $K \subset {\rm Int}\, J$, $D_K \Subset D_J$ are tubular neighborhood of $K$ and $J$ in $(M, S)$ respectively 
and $S \cup D_J \Subset U \subset M$. 
If $F \in {\rm Isot}^r(M, S; {\rm supp} \Subset U)_0$ and $F(D_K \times I) \Subset D_J \times I$, then there exists 
a factorization $F = GH$ such that  
\benum
\item[{\rm (1)}] $G \in {\rm Isot}^r(M, S; \supp \Subset D_J)_0$ and  $G = F$ on $D_K \times I$, 
\item[{\rm (2)}] $H := G^{-1}F \in {\rm Isot}^r(M, S; \supp \Subset U, {\rm rel}\,D_K)_0$. 
\eenum 
\elem

\bpf 
Let $(X, \partial_t)$ denote the velocity vector field of the isotopy $F$ on $M \times I$. 
Since $F(S \times I) = S \times I$, the vector field $(X, \partial_t)$ is tangent to $S \times I$. 
Take a $C^\infty$ function $\rho : M \times I \to I$ such that 
$\rho \equiv 0$ on $[\mbox{a nbd of $(M - {\rm Int}\,D_J)$}] \times I$ and $\rho \equiv 1$ on $F(D_K \times I)$. 
Then, the vector field $(\rho X, \partial_t)$ on $M \times I$ is complete and tangent to $S \times I$, 
so that it induces an isotopy $G$ of $(M,S)$, which satisfies the required conditions by the choice of $\rho$. 
\epf 

The next lemma represents a standard technique to decrease the length of factorization into commutators with support in balls.
It enable us to reduce the length 2 in Theorem I\,(2) to the length 1 under some combination of balls.

\begin{lem}\label{lem_intersection} 
Suppose $K, L \in {\cal I}(S)$ and $D_K, D_L$ are tubular neighborhood of $K$ and $L$ in $(M, S)$ respectively. 
If ${\rm Int}\,K \cap {\rm Int}\,L \neq \emptyset$, 
then for any $g \in {{\rm Diff}^r(M, S; \supp \Subset D_K)_0}$ and $h \in {{\rm Diff}^r(M, S; \supp \Subset D_L)_0}$ 
there exists a factorizaiton $gh = \widehat{g}\,\widehat{h}$ for some $\widehat{g} \in {{\rm Diff}^r(M, S; \supp \Subset D_K)_0^{\,c}}$ and 
$\widehat{h} \in {{\rm Diff}^r(M, S; \supp \Subset D_L)_0}$. 
\end{lem} 

\begin{proof} \mbox{} 
Take $G \in {{\rm Isot}^r(M,S; \supp \Subset D_K)_0}$ with $G_1 = g$. 
There exists $J \in {\cal I}(S)$ and a tubular neighborhood $D_J$ of $J$ in $(M, S)$ 
such that $J \subset {\rm Int}\,K$, $D_J \Subset D_K$ and $\supp \,G \Subset D_J$.  
Take $E \in {\cal I}(S)$ and a tubular neighborhood $D_E$ of $E$ in $(M, S)$ 
such that $E \subset {\rm Int}\,K \cap {\rm Int}\,L$, $D_E \Subset D_K$ and $D_E \Subset D_L$. 
We can find an isotopy $\Phi \in {{\rm Isot}^r(M,S; \supp \Subset D_K)_0}$ with $\Phi_1(D_J) = D_E$. 
Let $\phi := \Phi_1 \in {{\rm Diff}^r(M,S; \supp \Subset D_K)_0}$ and 
$\widehat{g} := [g, \phi] \in {\rm Diff}^r(M,S; \supp \Subset D_K)_0^{\,c}$. 
Since $\supp(\phi G \phi^{-1}) = \phi(\supp\,G) \Subset \phi(D_J) = D_E$,  
it follows that \\
\hsp $h':=\phi g \phi^{-1} \in {\rm Diff}^r(M,S; \supp \Subset D_E)_0$, \hsh 
$\widehat{h} := h'h \in {\rm Diff}^r(M,S; \supp \Subset D_L)_0$ \hsh and  \\
\hsp $gh = g(\phi g^{-1} \phi^{-1}) (\phi g \phi^{-1})h = \widehat{g}\,h'h = \widehat{g}\,\widehat{h}$. 
\epf 

The next lemma describes how to factor a short isotopy $F \in {\rm Isot}^r(M,S)_0$ 
into isotopies with support in balls modulo the normal subgroup ${\rm Isot}^r(M, S; {\rm supp} \subset M - S)_0$. 

\blem\label{lem_factorization_S}
Suppose $r \in \oIZ_{\geq 2}$, $V \subset U$ are neighborhoods of $S$ in $M$, \\
\hsp \hsh $f \in {\rm Diff}^r(M,S; \supp \Subset U)_0$, $F \in {\rm Isot}^r(M,S; \supp \Subset U)_{\id, f}$ \ and \ $q \in S$, $F_q(I) \subsetneqq S$. \\
\hsf {\rm [1]} There exists a factorization $F = GHE$ which satisfies the following conditions. 
\benum
\item[{$(1)$}] $G \in {\rm Isot}^r(M,S; {\rm supp} \Subset D_J)_0$ and $G = F$ on $D_K \times I$, where 
\bit 
\itemI $J, K \in {\cal I}(S)$, \ $F_q(I) \subset {\rm Int}\,J$, \ $\lambda(J) > 1/2$, \ $q \in  K \subset {\rm Int}\,J$, \ $pr_M F(K \times I) \subset {\rm Int}\,J$
\itemII $D_J$ and $D_K$ are tubular neighborhoods of $J$ and $K$ in $(M, S)$ respectively, \\
$D_J \Subset V$, $D_K \Subset D_J$ and \ $pr_M F(D_K \times I) \Subset D_J$. 
\eit 
$H' := G^{-1}F \in {\rm Isot}^r(M,S; {\rm supp} \Subset U, {\rm rel}\, D_K)_0$ . 

\item[{$(2)$}] $H \in {\rm Isot}^r(M,S; {\rm supp} \Subset D_{L})_0$ and $H = H'$ on $(C_{L} \cup D_K) \times I$, where 
\bit 
\itemI $L \in {\cal I}(S)$, \ $K' \equiv S - {\rm Int}\,K \subset {\rm Int}\, L$, \ $\lambda(L) > 1/2$, 
\itemII $D_{L} \supset C_{L}$ are tubular neighborhoods of $L$ in $(M, S)$, $D_L \Subset V$, $D_L|_{L \cap K} \subset D_K$,  \\
$C_{L} \subset D_{L} - \partial_+ D_{L}$ and $pr_M H'(C_L \times I) \subset D_{L} - \partial_+ D_{L}$. 
\eit 

\item[{$(3)$}] $E := H^{-1}H' \in {\rm Isot}^r(M,S; {\rm supp} \Subset U, {\rm rel}\,C_{L} \cup D_K)_0 \subset 
{\rm Isot}^r(M,S; {\rm supp} \Subset U -S)_0$. 
\eenum  
\hsf {\rm [2]} There exists a factorization $f = ghe$ which satisfies the following conditions. 
\benum
\item[{$(1)$}] $g := G_1 \in {\rm Diff}^r(M,S; {\rm supp} \Subset D_{J})_0$, where $J \in {\cal I}(S)$, $\lambda(J) > 1/2$ and $D_J \Subset V$.
\item[{$(2)$}] $h := H_1 \in {\rm Diff}^r(M,S; {\rm supp} \Subset D_{L})_0$, where $L \in {\cal I}(S)$, $\lambda(L) > 1/2$ and $D_L \Subset V$.
\item[{$(3)$}] $e := E_1 \in {\rm Diff}^r(M, S; {\rm supp} \Subset U - S)_0$. 
\eenum 
\elem

\bpf {[1]} (1) We can find $J, K, D_J, D_K$ which satisfy the conditions (i) and (ii). 
Then, the isotopy $G$ is obtained by Lemma~\ref{lem_isot_ext_D} (Isotopy extension property). 

(2) We can find $L, D_L, C_L$ which satisfy the conditions (i) and (ii).
The isotopy $H$ is obtained by the argument similar to the proof of Lemma~\ref{lem_isot_ext_D}.
Consider the velocity vector field $(X', \partial_t)$ of the isotopy $H'$ on $M \times I$. 
Note that $X' = 0$ on $D_K \times I$, since $H' = \id$ on $D_K \times I$. 
Take a $C^r$ function $\rho : M \times I \to I$ such that $\rho = 0$ on 
$[\mbox{a nbd of $M - {\rm Int}\,D_L$}] \times I$ and $\rho = 1$ on $H'(C_L|_{K'} \times I)$.
Then, the vector field $(\rho X', \partial_t)$ induces an isotpy $H$ on $M \times I$, which satisifes the required conditions. 
\epf 

\bprop\label{prop_factorization_1} Suppose $r \in \oIZ_{\geq 2}$, $V \subset U$ are neighborhoods of $S$ in $M$ and 
$f \in {\rm Diff}^r(M, S; \supp \Subset U)_0$.
If there exists $F \in {\rm Isot}^r(M, S; \supp \Subset U)_{\id, f}$ with $|\nu_S(F)| < \ell \in \IZ_{\geq 1}$, 
then $f$ has a factorization 
$f = g_{2\ell} \cdots g_{1} h$ which satisfies the following conditions. 
\bit 
\itemI $g_i \in {\rm Diff}^r(M,S; {\rm supp} \Subset D_{K_i})_0$, where $K_i \in {\cal I}(S)$, $\lambda(K_i) > 1/2$ and $D_{K_i} \Subset V$ \ $(i \in [2\ell])$.
\itemII $g_i \in {\rm Diff}^r(M,S; {\rm supp} \Subset D_{K_i})_0^c$ \ $(i = 2, \cdots, 2\ell)$. 
\itemiii $h \in {\rm Diff}^r(M, S; {\rm supp} \Subset U - S)_0$.
\eit 
\eprop 

\bpf \mbox{} 
\benum[(1)] 
\item Let $\sigma := \nu_S(F) = \mu_p(F|_{S \times I})$. 
Then, we can apply Lemma~\ref{lem_isot} to $F|_{S \times I} \in {\rm Isot}^r(S)_{\id, f|_S}$ 
to obtain $G \in {\rm Isot}^r(S)_{\id, f|_S}$ such that $F|_{S \times I} \simeq_\ast G$ in ${\rm Isot}^r(S)_{\id, f|_S}$ and 
$G_p = \delta_\sigma + p$. 
By Corollary~\ref{cor_isot_ext_02} for a compact neighborhood $W$ of $S$ in ${\rm Int}\,U$ 
there exists $F' \in {\rm Isot}^r(M,S)_{\id, f}$ such that $F' = G$ on $S \times I$ and $F' = F$ on $(M - W) \times I$.
Then, $F' \in {\rm Isot}^r(M, S; \supp \Subset U)_{\id, f}$ 
since $\supp\,F' \subset W \cup \supp\,F \Subset U$, and 
$F'_p = G_p = \delta_\sigma + p$ so that $\nu_S(F') = \mu_p(F'|_{S \times I}) = \lambda(F'_p) = \lambda(\delta_\sigma) = \sigma$. 
Therefore, replacing $F$ by $F'$, we may assume that $F_p = \delta_\sigma + p$. 

\item We use the following notations : 
\bit 
\itemI $t_i := i/\ell \in I$, \hsh $f_i := F_{t_i} \in {\rm Diff}^r(M, S; \supp \Subset U)_0$, \hsh $p_i := f_i(p) \in S$ \hsh $(i \in [\ell]_+)$

\itemII $I_i := [t_{i-1},t_i] \subset I$, \hsh $f^{(i)} := f_i f_{i - 1}^{\ -1} \in {\rm Diff}^r(M, S; \supp \Subset U)_0$, \\
$F^{(i)} := (F|_{M \times I_i})\big(f_{i-1}^{\ -1}\times \id_{I_i}) \in {\rm Isot}^r(M, S; \supp \Subset U)_{\id, f^{(i)}}$ \ \ ($I_i$ is identified with $I$) \ \ $(i \in [\ell])$
\eit 
It follows that 
\bit 
\itemiii $f_0 = \id_M$ and $f = f_\ell = f^{(\ell)}f^{(\ell-1)} \cdots f^{(i)} \cdots f^{(1)}$, 
\itemiv $F^{(i)}_{p_{i-1}}(I_i)$ is a closed interval or a single point in $S$ \ \ $(i \in [\ell])$. 
\eit 
(Proof) Since $p_{i-1} = f_{i-1}(p)$, for each $t \in I_i$ we have \\
\hsf $F^{(i)}_{p_{i-1}}(t) = F^{(i)}(p_{i-1}, t) = (F|_{M \times I_i})\big(f_{i-1}^{\ -1}\times \id_{I_i})(p_{i-1}, t) = F(p,t) = F_p(t) = \delta_\sigma(t) + p  
= \pi_S(\sigma t) + p.$ \\
This implies that $T_i := F^{(i)}_{p_{i-1}}(I_i) = \pi_S(\sigma I_i) + p$. 
Since the interval $\sigma I_i$ in $\IR$ has width $|\sigma|/\ell < 1$, 
we have $T_i \in {\cal I}(S)$ (with the end points $p_{i-1}$, $p_i$) if $\sigma \neq 0$ and $T_i = \{ p \}$ if $\sigma = 0$. 

\item By the backward induction on $i \in [\ell]$ we can construct a factorization $f^{(i)} = g^{(i)}h^{(i)}e^{(i)}$ $(i \in [\ell])$ and 
neighborhoods $V = V_\ell \supset V_{\ell-1} \supset \cdots \supset V_1$ of $S$ in $M$ 
which satisfies the following conditions. \\ 
\hsh For $i \in [\ell]$ \hsf  
\btab[t]{@{}l}
(a)$_i$ \ $g^{(i)} \in {\rm Diff}^r(M,S; {\rm supp} \Subset D_{J_i})_0$, where $J_i \in {\cal I}(S)$, $\lambda(J_i) > 1/2$ and $D_{J_i} \subset V_i$. \\[1mm] 
(b)$_i$ \ $h^{(i)} \in {\rm Diff}^r(M,S; {\rm supp} \Subset D_{L_i})_0$, where $L_i \in {\cal I}(S)$, $\lambda(L_i) > 1/2$ and $D_{L_i} \subset V_i$. \\[1mm] 
(c)$_i$ \ $e^{(i)} \in {\rm Diff}^r(M, S; {\rm supp} \Subset U - S)_0$. 
\etab \\[1mm] 
\hsh For $i \in [\ell - 1]$ \hsh 
(d)$_i$ \ $e^{(i+1)} \in {\rm Diff}^r(M, S; {\rm supp} \Subset U - V_i)_0$. \\[1mm]  
(Proof) For the data $V_\ell = V$, $f^{(\ell)}$, $F^{(\ell)}$, $p_{\ell - 1}$,  
Lemma~\ref{lem_factorization_S}\,(2) yields a factorization $f^{(\ell)} = g^{(\ell)}h^{(\ell)}e^{(\ell)}$ 
which satisfies the conditions (a)$_\ell$, (b)$_\ell$ and (c)$_\ell$. 
Suppose $i \in [\ell - 1]$ and assume that we have obtained a factorization $f^{(i+1)} = g^{(i+1)}h^{(i+1)}e^{(i+1)}$ and $V_{i+1}$ which satisfies 
(a)$_{(i+1)}$, (b)$_{(i+1)}$, (c)$_{(i+1)}$ (and (d)$_{(i+1)}$ when $i+1 \in [\ell-1]$). 
Then, $e^{(i+1)} \in {\rm Diff}^r(M, S; {\rm supp} \Subset U - V_i)_0$ for some neighborhood $V_i$ of $S$ with $V_i \subset V_{i+1}$. 
For the data $V_i$, $f^{(i)}$, $F^{(i)}$, $p_i$,  
Lemma~\ref{lem_factorization_S}\,(2) yields a factorization $f^{(i)} = g^{(i)}h^{(i)}e^{(i)}$ 
which satisfies the conditions (a)$_i$, (b)$_i$ and (c)$_i$. 

\item For each $i \in [\ell]$, $i \geq 2$, note that $e^{(i)}$ commutes $g^{(j)}, h^{(j)}$ $(j \in [i-1])$, since 
$\supp \,e^{(i)} \subset M - V_{i-1}$ and $\supp\,g^{(j)}, \supp\,h^{(j)} \subset V_j \subset V_{i-1}$. 
Hence, we have the following factorization of $f$ : \\[1mm] 
\hsp 
$\bary[t]{@{}l@{ \ }l}
f & = f^{(\ell)}f^{(\ell-1)} \cdots f^{(1)} = (g^{(\ell)} h^{(\ell)} e^{(\ell)}) (g^{(\ell-1)} h^{(\ell-1)} e^{(\ell-1)})\cdots (g^{(1)} h^{(1)} e^{(1)}) \\[2mm] 
& = \big((g^{(\ell)} h^{(\ell)}) (g^{(\ell-1)} h^{(\ell-1)})\cdots (g^{(1)} h^{(1)})\big) \circ \big(e^{(\ell)} e^{(\ell-1)} \cdots e^{(1)}\big) \ = \ g_{2\ell}' \cdots g_{1}' h, 
\eary$ \\[2mm]
where $g_{2i}' := g^{(i)}$, \,$g_{2i-1}' := h^{(i)}$ \ $(i \in [\ell])$ and \ $h := e^{(\ell)} e^{(\ell-1)} \cdots e^{(1)}$. 
It follows that 
\bit 
\itemI $g_i' \in {\rm Diff}^r(M,S; {\rm supp} \Subset D_{K_i})_0$, where $K_i \in {\cal I}(S)$, $\lambda(K_i) > 1/2$ and $D_{K_i} \subset V$.
\itemII $h \in {\rm Diff}^r(M, S; {\rm supp} \subset U - S)_0$
\eit 
\item We can apply Lemma~\ref{lem_intersection}\,(2) to $g_{2\ell}', \cdots, g_2', g_{1}'$ inductively (backward from $2\ell$ to 2) to obtain a factorization \hsf 
$f = g_{2\ell} \cdots g_2 g_{1} h$, where \\
\hsp $g_i \in {\rm Diff}^r(M,S; {\rm supp} \Subset D_{K_i})_0^c$ \ $(i = 2, \cdots, 2\ell)$ \ and \ 
$g_1 \in {\rm Diff}^r(M,S; {\rm supp} \Subset D_{K_1})_0$. 
\eenum
\vspace*{-9.5mm}
\epf 

\subsection{Factorization of isotopies along a union of circles} \mbox{} 

The result in the previous subsection extends directly to the case of a finite union of circles.     

\bprop\label{prop_factorization_L} Suppose $(M, L) \in {\cal C}^r(n,m)$ $(r \in \oIZ_{\geq 2}, n \geq 2, m \geq 1)$, 
$V \subset U$ are neighborhoods of $L$ in $M$ and $f \in {\rm Diff}^r(M, L; \supp \Subset U)_0$. 
If there exists $F \in {\rm Isot}^r(M, L; \supp \Subset U)_{\id, f}$ with $\|\nu(F)\| < \ell \in \IZ_{\geq 1}$, 
then $f$ has a factorization $f = g_{2\ell} \cdots g_{1} h$ which satisfies the following conditions. 
\bit 
\itemI $g_j \in {\rm Diff}^r(M,L; {\rm supp} \Subset D_j)_0$, where $D_j \in {\cal F}{\cal B}^r(M, L)_1^\ast$ and $D_j \Subset V$ \ $(j \in [2\ell])$ 
\itemII $g_j \in {\rm Diff}^r(M,L; {\rm supp} \Subset D_j)_0^c$ \ $(j = 2, \cdots, 2\ell)$. 
\itemiii $h \in {\rm Diff}^r(M, L; {\rm supp} \Subset U - L)_0$.
\eit 
\eprop 

\bpfb 
\benum 
\item Let $L = \bigcup_{i \in [m]} S_i$. For each $i \in [m]$ 
take compact neighborhoods $W_i \Subset V_i \Subset U_i \Subset V$ of $S_i$ in $M$   
such that $U_i$ $(i \in [m])$ are pairwise disjoint and $F(W_i \times I) \Subset V_i$ $(i \in [m])$. 
Let $W := \bigcup_{i \in [m]}W_i$, $V' := \bigcup_{i \in [m]}V_i$ and $U' := \bigcup_{i \in [m]}U_i$. 
By Proposition~\ref{prop_isot_ext}\,(1) there exists \\
\hsp $F' \in {\rm Isot}^r_c(M, L; {\rm rel}\,M - V')_0$ with $F' = F$ on $W \times I$ and $\supp\,F' \subset \supp\,F$. \\ 
It follows that \ $H' := (F')^{-1}F \in {\rm Isot}^r(M; {\rm rel}\,W, \supp \Subset U)_0 \subset {\rm Isot}^r(M, L; \supp \Subset U - L)_0$ \ and \\
$h' := H'_1 \in {\rm Diff}^r(M, L; \supp \Subset U - L)_0$. 

\item For each $i \in [m]$ define $F^{(i)} \in {\rm Isot}^r(M, S_i; {\rm rel}\,M - V_i)_0$ by $F^{(i)} = F'$ on $V_i \times I$. 
Then, $\supp\,F^{(i)} \subset V_i$ so that 
$F^{(i)} \in {\rm Isot}^r(M, S_i; \supp \Subset U_i)_0$ and $f^{(i)} := F^{(i)}_1 \in {\rm Diff}^r(M, S_i; \supp \Subset U_i)_0$. 
Since $F^{(i)} = F' = F$ on $S_i \times I$ we have $|\nu_i(F^{(i)})| = |\nu_i(F)| \leq \| \nu(F)\| < \ell$. 
Hence, for the data $V_i \subset U_i$, $f^{(i)}$ and $F^{(i)}$ we obtain  
a factorization $f^{(i)} = g_{2\ell}^{(i)} \cdots g_{1}^{(i)} h^{(i)}$ which satisfies the following conditions. 
\vskip 1mm 
\bit 
\item[(i)$'$\ ] $g^{(i)}_j \in {\rm Diff}^r(M,S_i; {\rm supp} \Subset D_{K_j^{(i)}})_0$, 
where $K_j^{(i)} \in {\cal I}(S_i)$ and $D_{K_j^{(i)}} \Subset V_i$ \ $(j \in [2\ell])$.
\item[(ii)$'$\,] $g^{(i)}_j \in {\rm Diff}^r(M,S_i; {\rm supp} \Subset D_{K_j^{(i)}})_0^c$ \ $(j = 2, \cdots, 2\ell)$. 
\item[(iii)$'$] $h^{(i)} \in {\rm Diff}^r(M, S_i; {\rm supp} \Subset U_i - S_i)_0$.
\eit 

\item For each $j \in [2\ell]$ consider the composition $g_j := g^{(1)}_j 
\cdots g^{(m)}_j \in {\rm Diff}^r(M, L; \supp \Subset D_j)_0$.
Here, $D_j \in {\cal F}{\cal B}^r(M, L)_1^\ast$ is obtained from $\bigcup_{i \in [m]} D_{K^{(i)}_j}$ by slightly enlarging and rounding its corners in ${\rm Int}\,V'$. 
Note that $g_j = g^{(i)}_j$ on $U_i$ $(i \in [m])$, since $\supp\,g^{(i)}_j \subset U_i$.  $(i \in [m])$, it follows that $g_j = g^{(i)}_j$ on $U_i$ $(i \in [m])$. 
Similarly we have the composition $h'' := h^{(1)} \cdots h^{(m)} \in {\rm Diff}^r(M, L; \supp \Subset U - L)_0$, 
which satisfies $h'' = h^{(i)}$ on $U_i$ $(i \in [m])$. Let $h := h''h' \in {\rm Diff}^r(M, L; \supp \Subset U - L)_0$.
 It is seen that $f' = g_{2\ell} \cdots g_1h''$ 
since $f' = f^{(i)}$, $g_j = g_j^{(i)}$, $h'' = h^{(i)}$ on $U_i$ and $f ' = g_j = h'' = \id$ on $M - U'$ $(i \in [m], j \in [2\ell])$.  
Hence, $f = f'h' = g_{2\ell} \cdots g_1h$. 
\item It remains to show that $g_j \in {\rm Diff}^r(M,L; {\rm supp} \Subset D_j)_0^c$ \ $(j = 2, \cdots, 2\ell)$. 
From (2)(ii)$'$ it follows that $g^{(i)}_j = [a^{(i)}_j, b^{(i)}_j]$ for some $a^{(i)}_j, b^{(i)}_j \in {\rm Diff}^r(M,S_i; {\rm supp} \Subset D_{K_j^{(i)}})_0$. 
Consider the compositions $a_j := a^{(1)}_j \cdots a^{(m)}_j, b_j := b^{(1)}_j \cdots b^{(m)}_j \in {\rm Diff}^r(M, L; \supp \Subset D_j)_0$. 
Then, $a_j = a^{(i)}_j$, $b_j = b^{(i)}_j$ on $U_i$ $(i \in [m])$ and $a_j = b_j = \id$ on $M - U'$, so that \\
\hsp $g_j = g^{(i)}_j = [a^{(i)}_j, b^{(i)}_j] = [a_j,b_j]$ \ on $U_i$ \ \ and \ \ $g_j = \id = [a_j,b_j]$ \ on $M - U'$. \\
Hence, $g_j = [a_j,b_j] \in {\rm Diff}^r(M,L; {\rm supp} \Subset D_j)_0^c$ as required. 
\eenum 
\vskip -7mm
\epf 

Recall Assumption $P(r, n,\ell)$ in \S4.1. From Fact~\ref{fact_ball}\,(2) we have the following conclusion. 

\bthm\label{thm_clb_G} 
Suppose $(M, L) \in {\cal C}^r(n,m)$ $(r \in \oIZ_{\geq 2}$, $n \geq 2, m \geq 1)$ and $f \in {\rm Diff}^r_{(c)}(M, L)_0$. 
Under Assumption $P(r, n,1)$, the folllowing holds. 
\benum 
\item[{\rm (1)}] If there exists $F \in {\rm Isot}^r_{(c)}(M, L)_{\id, f}$ with $\|\nu(F)\| < \ell \in \IZ_{\geq 1}$, then $clb_{/{\cal G}_{(c)}} f \leq 2\ell + 1$. 
\vskip 1mm 
\item[{\rm (2)}] If $\min \| \widehat{\nu}_{(c)}(f) \| < \ell \in \IZ_{\geq 1}$, then $cl_{/\cal G_{(c)}}f \leq clb_{/\cal G_{(c)}}f  \leq 2\ell + 1$. 
Therefore, \\[1mm] 
\hspp \hsh $cl\,f \leq 2\ell + 1 + cld\, {\cal G}_{(c)}$ \ and \ $clb\,f \leq 2\ell + 1 + clbd\, {\cal G}_{(c)}$. 
\vskip 1mm 

\item[{\rm (3)}] In the case $\ell_c = m$, let $\widehat{k}{}^c := 2\lfloor k^c/2\rfloor + 3$ {\rm (cf. Proposition~\ref{prop_nu_diagram}\,(5))}. Then, 
\bit 
\itemI $cl_{/\cal G_c}d\, {\rm Diff}^r_c(M,L)_0 \leq clb_{/\cal G_c}d\, {\rm Diff}^r_c(M,L)_0 \leq \widehat{k}{}^c$, 
\vskip 0.5mm 
\itemII $cld\, {\rm Diff}^r_c(M,L)_0 \leq \widehat{k}{}^c + cld\, {\cal G}_c$ \ and \ $clbd\, {\rm Diff}^r_c(M,L)_0  \leq \widehat{k}{}^c + clbd\, {\cal G}_c$. 
\eit 
\eenum 
\vspace*{-4mm} 
\ethm

\bpfb 
\benum 
\item Take $U = V$ as $U := M$ ($U :=$ a compact neighborhood of $L \cup \supp\,F$ in $M$). 
Then, we have a factorization $f = g_{2\ell} \cdots g_1 h$ as in Proposition~\ref{prop_factorization_L}. 
Since $h \in {\cal G}_{(c)}$, under Assumption $(\ast)$ \vspace*{1mm} it follows that 
\hsh $clb_{/{\cal G}_{(c)}} f \leq clb\,g_{2\ell} \cdots g_2g_1 \leq clb\,g_{2\ell} \cdots g_2 + clb\,g_1 \leq 2\ell - 1 + 2 = 2\ell + 1$. 
\vskip 1mm
\item By the definition there exists $F \in {\rm Isot}^r_{(c)}(M, L)_{\id, f}$ with $\| \nu(F) \| = \min \| \widehat{\nu}_{(c)}(f) \|$. 
\item For any $f \in {\rm Diff}^r_c(M,L)_0$, Proposition~\ref{prop_nu_diagram}\,(5) implies that \\
\hsp $\min \| \widehat{\nu}_c(f) \| \leq k^c/2 < \lfloor k^c/2\rfloor + 1 \in \IZ_{\geq 1}$. \hsh 
Hence, by (2) we have \\[1mm] 
\hsf $cl_{/\cal G_c}(f) \leq clb_{/\cal G_c}(f) \leq 2(\lfloor k^c/2\rfloor + 1) + 1 = \widehat{k}{}^c$ \ \ and \ \ 
$cl\,f \leq \widehat{k}{}^c + cld\, {\cal G}_c$, \ $clb\,f \leq \widehat{k}{}^c + clbd\, {\cal G}_c$. 
\eenum 
\vspace*{-7mm} 
\epf 

\noindent Note that $\widehat{k}{}^c = k^c + 3$ if $k^c$ is even and $\widehat{k}{}^c = k^c + 2$ if $k^c$ is odd. 

It remains to estimate $cld\,{\cal G}_c$ and $clbd\,{\cal G}_c$. Here, we can apply the following results in \cite{FRY}, since \\
\hsp \hsh ${\rm Diff}^r_c(M, L)_0 \vartriangleright {\cal G}_c := {\rm Diff}^r_c(M; \, {\rm rel}\  {\cal U}(L))_0 = 
{\rm Diff}^r_c(M, L; {\rm supp} \subset M - L)_0 \cong {\rm Diff}^r_c(M - L)_0$.

\begin{thm_FRY}\label{thm_III} $($\cite{FRY}$)$ \mbox{} Suppose $r \in \oIZ_{\geq 1}$ and $r \neq n+1$. 
\benum 
\item[{\rm (1)}] If $M$ is an open $n$-manifold and $n = 2i+1$ $(i \geq 0)$, then \\
\hsp $cld\,{\rm Diff}^r_c(M)_0 \leq 4$, $cld\,{\rm Diff}^r(M)_0 \leq 8$ \ and \ $clbd\,{\rm Diff}^r_c(M)_0 \leq 2n+4$. 
\item[{\rm (2)}] If $M$ is the interior of a compact $n$-manifold with nonempty boundary and $n = 2i$ $(i \geq 3)$, then \\
\hsp $cld\,{\rm Diff}^r_c(M)_0 < \infty$ \ and \ $clbd\,{\rm Diff}^r_c(M)_0 < \infty$. \\
Some upper bounds of these diameters are obtained in terms of handle decomposition or triangulation of $M$. 
\eenum 
\end{thm_FRY}  

\bcor\label{cor_clb_G} 
Suppose $(M, L) \in {\cal C}^r(n,m)$ $(r \in \oIZ_{\geq 2}$, $n \geq 2, m \geq 1)$. 
Under Assumption $P(r, n,1)$, the folllowing holds. 
\benum 
\item[{\rm (1)}] Suppose $M$ is an $n$-manifold without boundary with $n = 2i+1$ $(i \geq 1)$. 
\bit 
\itemI If $f \in {\rm Diff}^r_c(M, L)_0$ and $\min \| \widehat{\nu}_c(f) \| < \ell \in \IZ_{\geq 1}$, then $cl\,f \leq 2\ell + 5$ \ and \ $clb\,f \leq 2\ell + 2n + 5$. 
\itemII In the case $\ell_c = m$, \ \ 
$cld\, {\rm Diff}^r_c(M,L)_0 \leq \widehat{k}{}^c + 4$ \ and \ $clbd\, {\rm Diff}^r_c(M,L)_0  \leq \widehat{k}{}^c + 2n+4$. 
\eit 

\vskip 1mm 
\item[{\rm (2)}] Suppose $M$ is the interior of a compact $n$-manifold possibly with boundary with $n = 2i$ $(i \geq 3)$. \\
In the case $\ell_c = m$, \ \ $cld\, {\rm Diff}^r_c(M,L)_0 < \infty$ \ and \ $clbd\, {\rm Diff}^r_c(M,L)_0  < \infty$. 

\item[{\rm (3)}] In the cases {\rm (1)(ii)} and {\rm (2)}, if $M$ is connected, then 
the group ${\rm Diff}^r_c(M,L)_0$ is uniformly weakly simple relative to the subset ${\cal K \cal P}(M, L)$ and  
it is bounded. 
\eenum 
\vspace*{-2mm} 
\ecor

\bpfb The group ${\cal G}_c$ is isomorphic to ${\rm Diff}^r_c(M - L)_0$ under the restriction map. 
Hence, the statements (1) and (2) follow from Theorem~\ref{thm_clb_G} and the following observations. 
\benum 
\item Since $M - L$ is an open $n$-manifoldwith $n = 2i+1$ $(i \geq 0)$, \\
\hsp $cld\,{\rm Diff}^r_c(M - L)_0 \leq 4$ \ and \ $clbd\,{\rm Diff}^r_c(M-L)_0 \leq 2n+4$. 

\item Since $M - L$ is the interior of a compact $n$-manifold with nonempty boundary with $n = 2i$ $(i \geq 3)$, \\
\hsp $cld\,{\rm Diff}^r_c(M- L)_0 < \infty$ \ and \ $clbd\,{\rm Diff}^r_c(M-L)_0 < \infty$.

\item The assertion follows from Proposition~\ref{prop_zeta_leq_4clb_pi}\,(2)(ii) and $clbd\, {\rm Diff}^r_c(M,L)_0  < \infty$. 
\eenum
\vspace*{-7mm} 
\epf 

\section{Examples} 

In this section we determine the lattice $A < \IZ^m$ for some examples $(M, L) \in {\cal C}^r(n, m)$, $L = \bigcup_{i \in [m]} S_i$ $(r \in \oIZ_{\geq 0}$, $n \geq 2, m \geq 1)$. 

\subsection{Circle actions} \mbox{}

\blem\label{lem_S^1-action} 
Suppose $(M, L) \in {\cal C}^r(n,m)$ admits a $C^r$ $\IS^1$ action $\rho : \IS^1 \times M \to M$ such that 
each circle $L_i$ is an orbit of this action. 
For each $i \in [m]$ suppose the orbit map $\rho_{p_i} : \IS^1 \to S_i$, $\rho_{p_i}(z) = z \cdot p_i$, has the degree $\ell_i$. 
Then, $\ell = (\ell_1, \cdots, \ell_m) \in A$. In the case $m = 1$, we have $\ell \in A= k\IZ$ and $k|\ell$.
\elem

\bpfb The $S^1$ action $\rho$ induces an isotopy $F \in {\rm Isot}^r(M, L)_{\id, \id}$ : $F(x,t) = e^{2\pi it} \cdot x$. 
It follows that $\nu_i(F) = {\rm deg}_{S_i}F_{p_i} = {\rm deg}_{S_i} \rho_{p_i} = \ell_i$ $(i \in [m])$, so that 
$A \ni \nu(F) = (\nu_i(F))_{i \in [m]} = \ell$. 
\epf 

\bexp\label{exp_Seifert} \mbox{} 
\benum 
\item In the case where $(M, K) \in {\cal C}^r(3,1)$ and $M$ is a Seifert fibered 3-manifold, 
if $K$ a regular fiber, then $k = 1$ and if $K$ is a $(p, q)$ multiple fiber, then $k|p$. 
\item In particular, if $K$ is a torus knot in $\IS^3$, then $k = 1$, 
since $K$ is a regular fiber of a standard Seifert fibering of $\IS^3$ with two singular fibers.  
\eenum 
\eexp 

\subsection{Examples of $(M, K) \in {\cal C}(n, 1)$ with $k = 0$} \mbox{} 

\blem\label{lem_center} 
Suppose $X$ is a topological space, $F \in {\rm Homot}^0(X)_{\id, \id}$, $p \in X$ and $F_p := F(p, \ast) \in \Omega(X, p)$.
Then, $[F_p] \in Z\pi_1(X, p)$. 
\vskip 2mm 
\bit 
\item[\bec] For any $[\alpha] \in \pi_1 (X, p)$ we have a homotopy $\eta : I^2 \to X$, $\eta(s,t) := F(\alpha(s), t)$. It follows that 
\hsp \btab[t]{ll}
$\eta(0,t) = F(\alpha(0), t) = F(p,t)$ \hsh & $\eta(s,0) = F(\alpha(s), 0) = \alpha(s)$ \\[2mm] 
$\eta(1,t) = F(\alpha(1), t) = F(p,t)$ & $\eta(s,1) = F(\alpha(s), 1) = \alpha(s)$.
\etab \hspace{15mm} 
\smash{$\xymatrix@M+1pt{
p \ar[r]^-{\alpha} & p \\
p \ar[u]^-{F_p}_-{t} \ar[r]_-{\alpha}^-{s} & p \ar[u]_-{F_p}
}$} \\[2mm] 
This means that $F_p \ast \alpha \simeq_\ast \alpha \ast F_p$ and $[F_p] [\alpha] = [\alpha][F_p]$ as reqired. 
\eit 
\vskip 2mm 
\elem

The symbol $i_\ast$ denotes the inclusion induced homomorphism. 

\bprop\label{prop_center} For $(M, K) \in {\cal C}(n, 1)$, we have $A = \{ 0 \} \subset \IZ$ and $k = 0$ in the following cases : 
\benum 
\item
[$(\sharp)$] {\rm (i)} $i_\ast : \pi_1(K) \to \pi_1 (M)$ is injective \ and \ {\rm (ii)} $i_\ast \pi_1 (K) \cap Z(\pi_1 (M)) = \{ 1 \}$. 
\item[$(\flat)$] 
$n = 3$ and for some/any tubular neighborhood $D$ of $K$  in $M$ \\
{\rm (i)} $i_\ast : \pi_1(D - K) \to \pi_1 (M - K)$ is injective \ and \ 
{\rm (ii)} $i_\ast \pi_1(D- K) \cap Z(\pi_1 (M - K)) = \{ 1 \}$. 
\eenum  
\eprop 

\bpfb In both cases we have to show that for any $F \in {\rm Isot}(M, K)_{\id, \id}$, 
$[F_p] = 0 \in \pi_1 (K, p)$ so that $F_p \simeq_\ast \ast$ in $K$ and $\nu(F) = {\rm deg}_K F_p = 0$. 

\benum 
\item[$(\sharp)$]  From Lemma~\ref{lem_center} and the condition (ii) it follows that $i_\ast [F_p] = [i F_p] \in i_\ast \pi_1 (K, p) \cap Z(\pi_1 (M, p)) = \{ 1 \}$.  
Hence $[F_p] = 0 \in \pi_1 (K, p)$ by the condition (i). 

\item[$(\flat)$] 
Take a point $q \in D - K$ close to $p$, so that the loop $F_q \in \Omega(D - K, q)$ and $F_q \simeq_\ast F_p$ in $D$. 
Consider $[F_q] \in \pi_1(D - K, q)$. 
Since $F|_{(M-K) \times I} \in {\rm Isot}^r(M - K)_{\id, \id}$, from Lemma~\ref{lem_center} and the condition (ii) it follows that  
$i_\ast [F_q] = [iF_q] \in i_\ast \pi_1(D - K) \cap Z\pi_1 (M - K, q) = \{ 1 \}$. 
Then, the condition (i) implies that $[F_q] = 1 \in \pi_1 (D - K, q)$ and $F_q \simeq_\ast \ast$ in $D - K$. 
Hence, we have $F_p \simeq_\ast F_q \simeq_\ast \ast$ in $D$ and $F_p \simeq \ast$ in $K$ since there exists a retraction $\zeta : D \to K$. 
\eenum 
\vspace*{-6.5mm} 
\epf 

%\bexp\label{exp_torus} \mbox{}
%\benum 
%\item If $K$ is a non-torus knot in $\IS^3$, the $(\IS^3, K)$ satisfies  the condition ($\flat$) (\cite{BZ}). 
%Note that the condition ($\flat$) for $n \geq 4$ reduces to the condition ($\sharp$) by the general position argument. 
%\item For $n \geq 4$ an example $(M, K) \in {\cal C}^r(n, 1)$ which satisfies the condition ($\sharp$) is obtained as follows. 
%Take any finitely presented group $G$ which includes an element $a \in G$ such that ${\rm ord}\,a = \infty$ and $\langle a \rangle \cap Z(G) = \{ 1 \}$ 
%(for example, $G = \IZ \ast H$, $H$ is a nontrivial group).  
%For $n \geq 4$, there exists a closed connected $C^\infty$ $n$-manifold $M$ with $\pi_1(M) \cong G$. 
%Take a circle $K$ in $M$ which represents the element $a$ in $\pi_1(M)$, so that the inclusion $i : K \subset M$ induces an isomorphism 
%$i_\ast : \pi_1(K) \cong \langle a \rangle < \pi_1(M)$. 
%\eenum 
%\eexp 

\subsection{Parallel circles} \mbox{} 

\blem\label{lem_parallel} Suppose $(M, L) \in {\cal C}^r(n,2)$ satisfies the following data and conditions.  
\bit 
\itemI $L \subset N \subset M$. $i_1 := i_{L_1, N}$, \ $i_2 := i_{L_2, N}$. $i_1{}_\ast : \pi_1 (L_1, p_1) \to \pi_1 (N, p_1)$ is injective. 
\itemII $f : (L_1,p_1) \to (L_2, p_2)$ is a $C^0$ map. $\ell := \deg\,f$ \\ 
$H : L_1 \times I \to N$ : $i_1 \simeq i_2 f$ is a $C^0$ homotopy. 
$\alpha := H_{p_1} : I \to N$.  
 \eit 
If $F \in {\rm Isot}^r(M, L)_{\id, \id}$ and $F(\alpha(I) \times I) \subset N$, then $\nu_2(F) = \ell\, \nu_1(F)$. 
\elem

\bpfb
\benum 
\item Since $\alpha \in {\cal P}^0(N)_{p_1, p_2}$ and $H : i_1 \simeq i_2f$ in $N$ with $H_{p_1}  = \alpha$, 
we have the commutative diagram : \\[2mm]  
\hsp $\xymatrix@M+1pt{
\pi_1(L_1, p_1) \ar@{^(->}[d]_-{f_\ast} \ar@{^(->}[r]^-{{i_1}_\ast} & \pi_1(N, p_1) \\
\pi_1(L_2, p_2) \ar@{^(->}[r]_-{{i_2}_\ast} & \pi_1(N, p_2) \ar[u]_-{\alpha_\ast}^-{\cong}
}$
\hsp 
\btab[t]{l} 
$\alpha_\ast([\gamma]) = [\alpha \ast \gamma \ast \alpha^{-1}]$, \\[2mm]
$\alpha^{-1} \in {\cal P}^0(N)_{p_2, p_1}$, $\alpha^{-1}(t) = \alpha(1-t)$ $(t \in I)$. 
\etab \\[2mm] 
By the condition (i) $i_1{}_\ast$ is injective. Then, $f_\ast$ and $i_2{}_\ast$ are also injective.  
In fact, a group homomorphism $\phi : \IZ \to \Gamma$ is injective if and only if $|{\rm Im}\,\phi| = \infty$. 
The injectivity of $i_2{}_\ast f_\ast$ implies $|{\rm Im}\,i_2{}_\ast f_\ast | = \infty$. 
Since ${\rm Im}\,(i_2{}_\ast f_\ast) \subset {\rm Im}\,i_2{}_\ast$, it follows that $|{\rm Im}\,i_2{}_\ast | = \infty$
and $i_2{}_\ast$ is also injective. 

\item Given any $F \in {\rm Isot}^r(M, L)_{\id, \id}$ with $F(\alpha(I) \times I) \subset N$. 
Consider a $C^0$ homotopy \\
\hsp $G : I^2 \to N$ : $G(s,t) := F(\alpha(t), s) \in F(\alpha(I) \times I) \subset N$. \\
It follows that $G : i_1F_{p_1} \simeq i_2 F_{p_2}$ in $N$ and $G(0,t) = \alpha(t) = G(1,t)$ \ $(t \in I)$. 
This implies that $i_1 F_{p_1} \simeq_\ast \alpha \ast (i_2 F_{p_2}) \ast \alpha^{-1}$ in $N$, so that \\
\hsp $\alpha_\ast i_2{}_\ast f_\ast[F_{p_1}] =  i_1{}_\ast [F_{p_1}] = [i_1 F_{p_1}] 
= \alpha_\ast [i_2 F_{p_2}] = \alpha_\ast i_2{}_\ast [F_{p_2}]$ \ \ and \ \ $f_\ast[F_{p_1}] = [F_{p_2}]$. \\
Therefore, $\nu_2(F) = {\rm deg}_{L_2}F_{p_2} = {\rm deg}_{L_2} fF_{p_1} = ({\rm deg}\, f)({\rm deg}_{L_1}F_{p_1}) = \ell \,\nu_1(F)$. 
\eenum 
\vspace*{-7mm} 
\epf 

\bnotb Suppose $(M, L = S_1 \cup S_2) \in {\cal C}^r(n,2)$.
We say that two circles $S_1$ and $S_2$ are 
(i) $C^0$ parallel if there is a $C^0$ embedded annulus $E$ in $M$ such that $\partial E = S_1 \cup S_2$ and 
(ii) $\pm$ $C^0$ parallel if there is an oriented $C^0$ embedded annulus $E$ in $M$ such that $\partial E = S_1 \cup (\mp S_2)$ with respect to the induced orientation on $\partial E$.
In this case we say that $S_1$ and $S_2$ are parallel with respect to the annulus $E$. 
\enot 

\blemb\label{lem_annulus}
Suppose $(M, L) \in {\cal C}^r(n,2)$, $S_1$ and $S_2$ are $\pm$ parallel with respect to an oriented annulus $E$ in $M$, 
$E \subset N \subset M$ and $i_1{}_\ast : \pi_1 (S_1, p_1) \to \pi_1 (N, p_1)$ is injective. 
If $F \in {\rm Isot}^r(M, L)_{\id, \id}$ and $F(E \times I) \subset N$, then $\nu_1(F) = \pm \nu_2(F)$. 
\elem

\bpfb The annulus $E$ admits a homeomorphism $H : S_1 \times I \approx E$ 
such that $H_0 = \id_{L_1}$ and $H_1 : S_1 \approx S_2$ is an orientation-preserving/reversing homeomorphism. 
We can modify $H$ so that $H_1(p_1) = p_2$. 
Then, $f := H_1$ and $H : S_1 \times I \approx E \subset N$ satisfies the condition (ii) in Lemma~\ref{lem_parallel}. 
In this case $\deg \,f = \pm 1$ and $\alpha(I) \subset E$ so that 
$F(\alpha(I) \times I) \subset F(E \times I)$ for any $F \in {\rm Isot}^r(M, L)_{\id, \id}$. 
Hence, the assertion follows from Lemma~\ref{lem_parallel}.
\epf 

\bexpb\label{exp_Hopf} (Hopf link) \ 
Consider the Hopf fibration $\pi : \IS^3 \to \IS^2$. 
It is a principal $\IS^1$ bundle and this $\IS^1$ action $\rho$ induces a canonical orientation on each fiber. 
Let $m \in \IZ_{\geq 1}$ and consider a link $L^{(m)} := \bigcup_{i \in [m]} S_i \subset \IS^3$ 
consisting of $m$ distinct circle fibers $S_i := \pi^{-1}(q_i)$ $(i \in [m])$. 
Since the subset $\{ q_1, \cdots, q_m \} \subset \IS^2$ is unique up to an isotopy on $\IS^2$, 
the link $L^{(m)} \subset \IS^3$ is uniquely determined up to a bundle isotopy on $\pi$. 
Note that $L^{(1)}$ is a unknot and $L^{(2)}$ is the Hopf link. 

If $C$ is an arc in $\IS^2$ connecting two points $q$ and $q'$, then the restriction $\pi : \pi^{-1}(C) \to C$ is a trivial principal $\IS^1$ bundle, so that  
$E = \pi^{-1}(C)$ is an annulus and the fiber circles $\pi^{-1}(p)$ and $\pi^{-1}(p')$ are parallel with respect to $E$. 
For $m \geq 2$, any circles $L_i$ and $L_j$ in the link $L^{(m)}$ $(i \neq j)$ are parallel with respect to 
an annulus $E_{ij} := \pi^{-1}(C_{ij})$, where $C_{ij}$ is an arc in $\IS^2 - \{q_1, \cdots, \widehat{q}_i, \cdots, \widehat{q}_j, \cdots, q_m\}$ connecting $q_i$ and $q_j$. 

For $(\IS^3, L^{(m)}) \in {\cal C}^r(3,m)$ the lattice $A = A_{(\IS^3, L^{(m)})} < \IZ^m$ takes the following form. \\
\raisebox{3mm}{ 
\hspp $\bary[t]{|c|c|c|c|} \hline 
m & 1 & 2 & \geq 3 \makebox(0,12){} \\[0.5mm] \hline 
A & \IZ & \IZ^2 & \IZ(1, \cdots, 1) \makebox(0,13){} \\[0.5mm] \hline
{\rm rank}\,A & 1 & 2 & 1 \makebox(0,12){} \\[0.5mm] \hline 
\eary$} \\[2mm] 
Hence, the group ${\rm Diff}^\infty(\IS^3, L^{(m)})_0$ 
is bounded for $m = 1,2$ and unbounded for $m \geq 3$. 
The case $m = 2$ shows that some kind of injectivity condition is necessary in Lemma~\ref{lem_parallel} and Lemma~\ref{lem_annulus}. 
\eexp 

\bpfb 
\benum[(1)]
\item $m=1$ : The $\IS^1$ action $\rho$ on $\IS^3$ induces the orbit map $\rho_{p_1} : \IS^1 \approx S_1$ of degree 1. 
Hence, $1 \in A$ and $A = \IZ$ by Lemma~\ref{lem_S^1-action}. 

\item $m=2$ : The pair $(\IS^3, L^{(2)})$ has the following representation. 
Consider a solid torus $T$ with a core $K$, a longitude $\ell$ and a meridian $m$.
Then, $(\IS^3, L^{(2)}= S_1 \cup S_2)$ is obtained by pasting two copies $(T_i, L_i,  \ell_i, m_i)$ $(i = 0,1)$ of $(T, K, \ell, m)$ 
by a diffeomorphism $h : (\partial T_1, \ell_1, m_1) \approx (\partial T_2, m_2, \ell_2)$. 
It follows that $A \ni \nu(F) = (1,0)$ if 
$F \in {\rm Isot}^\infty(\IS^3, L^{(2)})_{\id,\id}$ rotates $\IS^3$ around $S_2$ by $2\pi$ so that $S_2$ is fixed pointwise, 
while $S_1$ rotates once along itself.   
Similarly, $(0,1) \in A$. Hence, $A = \IZ \oplus \IZ$. 

\item $m \geq 3$ : 

The $C^\infty$ $\IS^1$ action $\rho$ on $\IS^3$ induces the orbit map $\rho_{p_i} : \IS^1 \approx S_i$ of degree 1 for each $i \in [m]$. 
Hence, $(1, \cdots, 1) \in A$ by Lemma~\ref{lem_S^1-action}.

It remains to show that $\nu_i(F) = \nu_j(F)$ $(i,j \in [m], i \neq j)$ for any $F \in {\rm Isot}^r(\IS^3, L^{(m)})_{\id,\id}$. 
We apply Lemma~\ref{lem_annulus} to the pair $(\IS^3, S_i \cup S_j) \in {\cal C}(3,2)$. 
Take an index $\ell \in [m] - \{ i, j \}$ and let $N := \IS^3 - S_\ell$.  
Recall the annulus $E_{ij} := \pi^{-1}(C_{ij})$. 
The circles $S_i$ and $S_j$ are parallel with respect to $E_{ij}$ in $\IS^3$ and $E_{ij} \subset N$ since $C_{ij} \subset \IS^2 - \{ q_\ell \}$.  
The inclusion $i : S_i \subset N$ induces the injective homomorphism $i{}_\ast : \pi_1 (S_i, p_i) \to \pi_1 (N, p_i)$, 
since $S_i \cup S_\ell$ is the Hopf link in $\IS^3$ and $N$ is an open solid torus with the core $S_i$. 
Since $F \in {\rm Isot}(\IS^3, S_\ell)_{\id,\id}$, it follows that $F(E_{ij} \times I) \subset F(N \times I) = N$.
Therefore, $\nu_i(F) = \nu_j(F)$ by Lemma~\ref{lem_annulus}. 
\eenum 
\vspace*{-8mm} 
\epf 

\subsection{Some applications to foliations and fibrations by circles} \mbox{} 

Lemma~\ref{lem_annulus} has an application to the group of leaf-preserving diffeomorphisms on a manifold foliated by circles. 

\bnotb Suppose $(M, {\cal F})$ is a foliated $C^r$ manifold.
Let ${\rm Diff}_L^r(M, {\cal F})$ and ${\rm Isot}_L^r(M, {\cal F})$ denote the group of leaf-preserving $C^r$ diffeomorphisms and isotopies of $(M, {\cal F})$. 
Let ${\rm Isot}_L^r(M, {\cal F})_0 = \{ F \in {\rm Isot}_L^r(M, {\cal F}) \mid F_0 = \id_M \}$ and 
${\rm Diff}_L^r(M, {\cal F})_0 := \{ F_1 \mid F \in {\rm Isot}_L^r(M, {\cal F})_0 \}$. 
\enot 

\bexpb\label{exp_foliation} (Foliation) \ Suppose $(M, {\cal F})$ is a $C^r$ $n$-manifold 
equipped with a 1-dimensional $C^r$ foliation ${\cal F}$. 
If ${\cal F}$ includes two circle leaves $S_1, S_2 \subset {\rm Int}\,M$ which satisfies the following conditions, 
then the group ${\rm Diff}_L^r(M, {\cal F})_0$ admits a surjective quasimorphism $\phi : {\rm Diff}_L^r(M, {\cal F})_0 \to \IR$. 
\bit 
\itemI $S_1, S_2$ are parallel with respect to an annulus $E$ in $M$. % (for some orientations of $S_1$ and $S_2$). 
\itemII There exists a ${\cal F}$-saturated subset $N$ of $M$ such that \ \ (a) $E \subset N$ \ and 
\item[] \hsh (b) the inclusion induced homomorphism $i_1{}_\ast : \pi_1 (S_1) \to \pi_1 (N)$ is injective.
\eit 
If $E$ itself is ${\cal F}$-saturated, then we can take $N = E$ in (ii). 
\eexp 

\bpfb  
For each $i = 1,2$ take a distinguished point $p_i \in S_i$ and 
a universal covering $\pi_i : \IR \to \IR/\IZ \approx S_i$ for which $S_1$ and $S_2$ are $+$ parallel with respect to an orientation of $E$. 
%with is compatible with the orientation of $S_i$. 
Let $L = S_1 \cup S_2$. 
Then, for $(M, L) \in {\cal C}^r(n,2)$ we obtain the surjective quasimorphism 
$\nu = (\nu_1, \nu_2) : {\rm Isot}^r(M, L)_0 \to \IR^2$, which restricts to a group homomorphism
$\nu| : {\rm Isot}^r(M, L)_{\id, \id} \to \IZ^2$. 

By restricting $\nu$ to ${\rm Isot}_L^r(M, {\cal F})_0$ we obtain the following diagram with exact lows (cf. Fact~\ref{fact_restriction}). \\[2mm]
\hsh $\xymatrix@M+1pt{
 {\rm Isot}^r_L(M, {\cal F})_{\id, \id} \ar@{}[r]|*{\subset} \ar@{->>}[d]_{\nu|} & {\rm Isot}^r_L(M, {\cal F})_0 \ar@{->>}[d]_-{\nu} \ar@{->>}[r]^-{R} 
 & {\rm Diff}^r_L(M, {\cal F})_0 \ar@{->>}[d]^-{\widehat{\nu}} \\
A_{\cal F} \ar@{}[r]|*{\subset} & \IR^2 \ar@{->>}[r] & \IR^2/A_{\cal F} \\
\raisebox{18.5mm}{$
\bary[t]{c}
 \rotatebox{90}{$\supset$} \\[1mm] 
\IZ^2
\eary$} &  
}$ \hsp 
\btab[t]{l} 
$R(F) = F_1$, \\[2mm] 
$\nu$ is a surjective quasimorphism. \\[2mm] 
$\nu|$ is a group epimorphism \\[1.5mm]
\hsp \hsh onto $A_{\cal F} := {\rm Im}\,\nu| < \IZ^2$.  
\etab \\[-10mm] 
For any $F \in {\rm Isot}_L(M, {\cal F})_{\id, \id}$ we have $F \in {\rm Isot}^r(M, L)_{\id, \id}$ and 
$F(E \times I) \subset F(N \times I) = N$ since $N$ is ${\cal F}$-saturated. 
Hence, $\nu_1(F) = \nu_2(F)$ by Lemma~\ref{lem_annulus}. 
This implies that $A_{\cal F} \subset \IZ(1,1)$ and ${\rm rank}\,A_{\cal F} \leq 1$. 
Hence, by Proposition~\ref{prop_rank<m} there exists a group epimorphism $\chi : \IR^2/A_{\cal F} \to \IR$ such that \\
\hsh $\phi := \chi \widehat{\nu} : {\rm Diff}_L^r(M, {\cal F})_0 \lra \IR$ \ is a surjective quasimorphism. 
\epf 

\begin{compl}\label{comp1} In the proof of Example~\ref{exp_foliation}, if ${\cal G} < {\rm Diff}_L^r(M, {\cal F})_0$ and $\widehat{\nu}| : {\cal G} \to \IR^2/A_{\cal F}$ is surjective, 
then $\phi = \chi \,\widehat{\nu}| : {\cal G} \to \IR$ is also a surjective quasimorphism. 
\end{compl}

\bnotb Suppose $\pi : M \to B$ is a $C^r$ $(N, \Gamma)$ bundle (i.e., a $C^r$ fiber bundle with fiber $N$ and structure group $\Gamma < {\rm Diff}^r(N)$). 
Let ${\cal D}^r_\pi(M)$ and ${\cal I}^r_\pi(M)$ denote the groups of $C^r$ fiber-preserving bundle diffeomorphisms and isotopies of $\pi$ respectively. 
Let ${\cal I}^r_\pi(M)_0 = \{ F \in {\cal I}^r_\pi(M) \mid F_0 = \id_M \}$ and 
${\cal D}^r_\pi(M)_0 := \{ F_1 \mid F \in {\cal I}^r_\pi(M)_0 \}$ (the identity connected component of ${\cal D}^r_\pi(M)$). 
\enot 
\bexpb\label{exp_fibering} (Circle bundle) \ Suppose $\pi : M \to B$ is a $C^r$ $({\Bbb S}^1, \Gamma)$ bundle with ${\rm SO}(2) < \Gamma < {\rm Diff}^r({\Bbb S}^1)$. 
Then the group ${\cal D}^r_\pi(M)_0$ admits a surjective quasimorphism $\phi : {\cal D}^r_\pi(M)_0 \to \IR$. 
\eexp 

\bpfb Consider the 1-dimensional $C^r$ foliation ${\cal F} = \{ \pi^{-1}(b) \mid b \in B \}$ on $M$. 
Take an arc $C$ in $B$ and let $\partial C = \{ q_1, q_2 \}$, $S_1 = \pi^{-1}(q_1)$, $S_2 = \pi^{-1}(q_2)$ and $E := \pi^{-1}(C)$. 
Then, (i) the circle leaves  $S_1, S_2$ of ${\cal F}$ are parallel with respect to the annulus $E$ in $M$ (for some orientations of $S_1$ and $S_2$) 
and (ii) $E$ is ${\cal F}$-saturated. 
Therefore, we can apply the proof of Example~\ref{exp_foliation} and Complement~\ref{comp1}. 
Restricting $\nu$ to ${\cal I}^r_\pi(M)_0$ we obtain the following diagram. 

\hsh $\xymatrix@M+1pt{
{\cal I}^r_\pi(M)_0  \ar[d]_-{\nu} \ar[r]^-{R} & {\cal D}^r_\pi(M)_0 \ar[d]^-{\widehat{\nu}} \\
\IR^2 \ar@{->>}[r] & \IR^2/A_{\cal F}   
}$ \hsp 
\btab[t]{l} 
Since ${\rm SO}(2) < \Gamma$, it is seen that $\nu$ is surjective and so is $\widehat{\nu}$. \\[2mm] 
Hence, $\phi = \chi\, \widehat{\nu} : {\cal D}^r_\pi(M)_0 \to \IR$ is also a surjective quasimorphism. 
\etab 
\vspace*{-5mm} 
\epf

\end{document}